\documentclass{article}
\usepackage[utf8]{inputenc}
\DeclareUnicodeCharacter{2011}{-}
\usepackage[margin = 2.5cm]{geometry}
\usepackage{float}
\usepackage{caption}
\usepackage{graphicx}
\usepackage{xcolor}
\definecolor{revisionyellow}{RGB}{180,135,0}
\usepackage{amsmath}
\usepackage{amsthm}
\usepackage{amssymb}
\usepackage{mathrsfs}
\usepackage{amsfonts}
\usepackage{lmodern}
\usepackage{lineno}
\usepackage{fancyhdr}
\usepackage{indentfirst}
\usepackage{cite}
\usepackage{hyperref}     
\usepackage{cleveref}
\usepackage{enumitem}
\usepackage{aliascnt} 
\usepackage{tikz}  
\usetikzlibrary{shapes,arrows.meta,positioning}

\allowdisplaybreaks

\newtheorem{theorem}{Theorem}[section]

\newaliascnt{lemma}{theorem}
\newtheorem{lemma}[lemma]{Lemma}
\aliascntresetthe{lemma}
\Crefname{lemma}{Lemma}{Lemmas}
\Crefname{lemma}{Lemma}{Lemmas}

\newaliascnt{corollary}{theorem}
\newtheorem{corollary}[corollary]{Corollary}
\aliascntresetthe{corollary}
\Crefname{corollary}{Corollary}{Corollaries}
\Crefname{corollary}{Corollary}{Corollaries}

\newaliascnt{proposition}{theorem}
\newtheorem{proposition}[proposition]{Proposition}
\aliascntresetthe{proposition}
\Crefname{proposition}{Proposition}{Propositions}
\Crefname{proposition}{Proposition}{Propositions}

\newaliascnt{observation}{theorem}
\newtheorem{observation}[observation]{Observation}
\aliascntresetthe{observation}
\Crefname{observation}{Observation}{Observations}
\Crefname{observation}{Observation}{Observations}

\newaliascnt{example}{theorem}

\aliascntresetthe{example}
\Crefname{example}{Example}{Examples}
\Crefname{example}{Example}{Examples}

\newaliascnt{question}{theorem}

\aliascntresetthe{question}
\Crefname{question}{Question}{Questions}
\Crefname{question}{Question}{Questions}

\theoremstyle{definition}  
\newaliascnt{remark}{theorem}
\newtheorem{remark}[remark]{Remark}
\aliascntresetthe{remark}
\Crefname{remark}{Remark}{Remarks}
\Crefname{remark}{Remark}{Remarks}
 
\newtheorem{claim}{Claim}
\newenvironment{proofofclaim}[1][Proof of Claim]{%
  \begin{proof}[#1]%
}{%
  \end{proof}%
}
\newtheorem*{claim*}{Claim}
\newenvironment{proofofclaim*}[1][Proof of Claim*]{%
  \begin{proof}[#1]%
}{%
  \end{proof}%
}
\Crefname{claim}{Claim}{Claims}
\theoremstyle{plain} 
\newtheorem*{proposition*}{Proposition}  

\theoremstyle{definition}

\newtheorem{definition}{Definition}[section]

\newtheorem{step}{Step}
\Crefname{step}{Step}{Steps}
\newtheorem{step*}{Step*}

\Crefname{theorem}{Theorem}{Theorems}
\Crefname{theorem}{Theorem}{Theorems}

\newcommand{\orw}{\overrightarrow}
\newcommand{\U}{\mathcal{U}}
\newcommand{\D}{\mathcal{D}}
\newcommand{\pt}{\partial}

\usepackage[utf8]{inputenc}

\title{An improved upper bound for oriented diameter of graphs with diameter $4$}
\author{Yaokun Feng\thanks{Center for Combinatorics and LPMC, Nankai University, Tianjin 300071, 
China. Email: 1120250006@mail.nankai.edu.cn.}, Hui Lei\thanks{School of Statistics and Data Science, LPMC and KLMDASR, Nankai University, Tianjin 300071, China. Email: hlei@nankai.edu.cn.}, Xiaopan Lian\thanks{Center for Combinatorics and LPMC, Nankai University, Tianjin 300071, 
China. Email: Lian@nankai.edu.cn.}, Zijian Ren\thanks {School of Mathematical Sciences and LPMC, Nankai University, Tianjin 300071, 
China. Email: 2120250113@mail.nankai.edu.cn.}}
\date{}

\begin{document}
	
	\maketitle
	\begin{abstract}
Let $f(d)$ denote the smallest integer such that every bridgeless graph of diameter $d$ admits a strong orientation with diameter at most $f(d)$. It is known that $f(2)=6$ and $f(3)=9$. For $d=4$, the classical bounds of Chv\'{a}tal and Thomassen  [JCTB, 1978] imply $12\le f(4)\le40$, and subsequent work reduced the upper bound to 21. Very recently, Lin, Wang and You further established the substantially stronger bound $f(4)\le18$. Pushing this bound below $18$ turns out to be considerably more difficult, since the remaining extremal configurations cannot be handled by existing techniques based on $R-S$ orientations and related local constructions. 

In this paper, we prove that $f(4)\le16$. Our approach is entirely different from previous ones. Instead of constructing a strong orientation directly, we develop a sequential orientation framework together with auxiliary distance functions and a potential-function analysis. This enables us to control directed distances globally while avoiding the intricate case analysis required by earlier methods. We believe that the framework introduced here may be useful for studying oriented diameter problems of larger diameter.

	\end{abstract}	
 \section{Introduction}
 All graphs in this paper are finite and without loops or multiple edges.  The {\it diameter} of a graph $G$ is the greatest distance between two vertices of $G$. A {\it bridge} of a graph $G$ is an
edge whose removal disconnects $G$. A graph $G$ is {\it bridgeless} if it has no bridge. 
An {\it orientation} $\overrightarrow{G}$ of a graph $G$ is an assignment of exactly one direction to each edge of $G$. 
An orientation $\overrightarrow{G}$ is {\it strong} if there is a dipath from $u$ to $v$ for any $u, v\in V (\overrightarrow{G})$.  The  {\it distance} from $u$ to $v$ is the length of a shortest dipath
from $u$ to $v$ in $\overrightarrow{G}$, denoted by   $\partial(u, v)$. If $\overrightarrow{G}$ is strong, then the {\it diameter} $\mathrm{diam}(\overrightarrow{G})$  is the maximum of
 $\partial(u, v)$ taken over all vertices $u$ and $v$ in $\overrightarrow{G}$. The {\it oriented diameter}
of a bridgeless graph $G$ is
$\overrightarrow{\mathrm{diam}}(G)= \min \{\mathrm{diam}(\overrightarrow{G}): 
\overrightarrow{G}$ is a strong orientation of $G$\}.

 Let $f(d)$ be the smallest value for which every bridgeless graph
$G$ with diameter $d$ admits that $\overrightarrow{\mathrm{diam}}(
G)\leq f(d)$. This notion
 was introduced by Chv\'{a}tal and Thomassen in \cite{CT1978}, and they proved that $\frac{1}{2}d^2+d\leq f(d)\leq 2d^2+2d$ for $d\geq2$. Babu et al. \cite{BBRV2021} improved the upper bound of $f(d)$ to  $1.373d^2+6.971d-1$ 
which is smaller than $2d^2+2d$ when $d\geq8$.  Researchers have tried to
obtain sharp upper bounds for some special classes of graphs, including complete $k$-partite graphs \cite{G1994, P1985, S1986}, near triangulations \cite{GLW, MPR2023, WCDGSV2021} and so on. Further investigations of such bounds on
the oriented diameter have been studied with respect to other graph parameters such as
the domination number \cite{KL2012}, minimum degree \cite{S2017}, and maximum degree \cite{DGS2018}.  

The known exact values of $f(d)$ are $f (1)=3, f(2)=6$ \cite{CT1978} and $f(3)=9$ \cite{WC2022}. For the family of bridgeless graphs of diameter
$4$, the lower and upper bounds of $f(4)$ provided by Chvátal and Thomassen \cite{CT1978} are $12$ and $40$, respectively. Babu et al. \cite{BBRV2021} improved
this upper bound of $f(4)$ to $21$. 
In this paper, we improved
this upper bound of $f(4)$ to $16$. While we were finalizing this work, Lin, Wang and You  \cite{lin2026improvedupperboundoriented}  showed that $f(4)\le 18$.  A  technical tool in their work is the so called \emph{R-S orientation}, which is developed by Kwok et al. in \cite{KWOK2010265} and also  applied in \cite{lin2025oriented,WC2022}. Since our paper does not rely on this construction, we do not elaborate on it here. The method in this work will be presented in full detail in \Cref{sec:sketch}.

\begin{theorem}\label{maintheorem}
$f(4)\leq 16$.    
\end{theorem}

The primary challenge in bounding $f(4)$ lies in the complexity of constructing a single strong orientation that simultaneously controls distances between all vertex pairs, despite the graph's diameter constraint. Previous approaches often relied on explicit cycle‑by‑cycle orientations or case analyses based on local neighborhood structures, which become increasingly intricate when the graph contains edges that are not contained in short cycles.

To overcome these difficulties, we introduce a novel framework that avoids explicit construction of a full orientation at once. Instead, we design a sequential orientation process guided by two auxiliary functions $\mu(x)$ and $\nu(x)$,  which measure adjusted directed distances from a fixed edge $uv$ to $x$ and from $x$ to $uv$, respectively.  The key innovation is the concept of \emph{good paths} and the derived parameters $\U(x)$ and $\D(x)$,  which measure the minimum adjusted lengths of good paths in a fixed partial orientation.   This functional approach not only simplifies the handling of vertices lying in different layers of the graph but also provides a flexible tool that may be applicable to larger diameters.

The paper is organized as follows.  In  \Cref{sec:preliminary}  we present some related works, necessary notations,  and the proof sketch of the main result; in   \Cref{sec:lemless5} and \Cref{sec:F234} we  state  the detailed orientation steps and the analysis of the functions $\mu,\nu,\U,\D$;  in 
 \Cref{sec:finalmain}   we combine these bounds with the path analysis to prove that every ordered pair is joined by a dipath of length at most $16$, thereby completing the proof that $f(4)\le16$.

\section{Preliminary}\label{sec:preliminary}
There are some  results  that support \Cref{maintheorem}.  Chv\'{a}tal  and Thomassen \cite{CT1978} proved the following useful theorem about general graphs.  
\begin{theorem} [\cite{CT1978}] \label{lemct} Let $G$ be a simple graph, then it admits an orientation $\overrightarrow{G}$ with the following property: if an edge $uv$ belongs to a cycle of length $k$, then there exists a dicycle of length at most $\psi(k)$ which contains $\overrightarrow{uv}$ or $\overrightarrow{vu}$, while $\psi(k) = (k - 2) \cdot 2^{\lfloor\frac{k - 1}{2}\rfloor} + 2$. 
\end{theorem}

Let $G$ be bridgeless with $\mathrm{diam}(G)=4$. If every edge of $G$ is contained in a triangle, then by  \Cref{lemct}, there exists an orientation $\overrightarrow{G}$  of $G$ in which every edge belongs to a dicycle of length at most $4$.   For any $x,y\in V(G)$, choose a shortest path
$x=x_0x_1\cdots x_\ell=y$ in $G$, where $\ell\le4$.
Each edge gives $\partial(x_{i-1},x_i)\le3$, so
$\partial(x,y)\le3\ell\le12$. Thus,
$\mathrm{diam}(\overrightarrow{G})\le12$. 
 Additionally, since $\mathrm{diam}(G)=4$,  for each edge of $G$, the shortest cycle contains it has length at most 9.  Lin and You \cite{lin2025oriented}  showed that if there exists an edge $uv\in E(G)$ such that the shortest cycle containing $uv$ is of length    6, 7, 8, or 9, then $G$  admits a desired orientation.   
\begin{theorem}[\cite{lin2025oriented}]\label{lemmore5}
Let $G=(V, E)$ and $ \mathrm{diam}(G)=4$ with no bridge. Suppose that there exists an edge $uv\in E(G)$ such that the shortest cycle containing $uv$ is of length at least  $6$. Then  $\overrightarrow{\mathrm{diam}}(G)\leq 13$.   
\end{theorem}

Therefore, by \Cref{lemct} and \Cref{lemmore5}, for the sake of \Cref{maintheorem}, we construct desired orientations for graphs $G$ which have an edge $uv\in E(G)$ such that $N(u)\cap N(v)=\emptyset$, and that every edge of  $G$ lies on a cycle of length at most $5$, formally stated in  \Cref{lemless5}.  

\begin{theorem}\label{lemless5}
Let $G=(V, E)$ and $ \mathrm{diam}(G)=4$. Suppose that there exists an edge $uv\in E$ such that $N(u)\cap N(v)=\emptyset$, and that every edge of  $ G $ lies on a cycle of length at most $5$. Then   $\overrightarrow{\mathrm{diam}}(G)\leq 16$.  
\end{theorem}
To prove \Cref{lemless5}, we construct an orientation of $G$ which has the same property as $G^O$ described in \Cref{lem:Gxy16}. In any orientation $\overrightarrow{G}$  of $G$,  for each $x\in V$, we define  $$
  \begin{aligned}
\nu(x):=
\min\{\partial(x,u), \partial(x,v)+2\}.\\
\mu(x):=
\min\{\partial(v,x), \partial(u,x)+2\}. 
\end{aligned}
 $$
 
\begin{lemma}\label{lem:Gxy16}
Let $G^O$ be an orientation of the graph in \Cref{lemless5}, with
$\partial(u,v)=1$ and $\partial(v,u)\le4$.
For every $x,y$ with finite adjusted parameters, 
\begin{equation}\label{eq:root-distance-bound}
 \partial(x,y)\le\nu(x)+\mu(y)+1.
\end{equation}
In particular, $\nu(x)+\mu(y)\le15$ implies $\partial(x,y)\le16$.
\end{lemma}

\begin{proof}
 
Choose the terms attaining $\nu(x)$ and $\mu(y)$.
Concatenating the corresponding dipaths through $u$ and $v$ gives
\[
\partial(x,y)\le\nu(x)+\mu(y)+
\begin{cases}
1,&\nu(x)=\partial(x,u),\quad\mu(y)=\partial(v,y),\\
0,&\nu(x)=\partial(x,v)+2,\quad\mu(y)=\partial(u,y)+2,\\
-2,&\text{otherwise}.
\end{cases}
\]
Here we use $\partial(u,v)=1$ and $\partial(v,u)\le4$.
The result follows. 
\end{proof}
\subsection{Additional Notations}
For a vertex $v\in V(G)$, $N(v)=\{u:uv\in E(G)\}$ and $N[v]=N(v)\cup \{v\}$. Let $U\subset V(G)$. Then $N(U)=\bigcup_{u\in U}N(u)\setminus U$ and $N[U]=\bigcup_{u\in U}N[u] $.  Denote by $G[U]$ the subgraph {\it induced} by $U$ in $G$. Let $ V\subset V(G)$ be disjoint with $U$. Denote by $[U, V]$ the set of all the edges with one endpoint in $U$ and another in $V$ (here, $U$ and $V$ need not be disjoint).  A  {\it $UV$-path} contained in $G$ is a path   with one endpoint in $U$ and the other in $V$ where the internal vertices are disjoint with $U\cup V$. Denote by $\mathrm{dist}(U,V)$ the length of a shortest $UV$-path. The {\it diameter} of $G$ is denoted by $\mathrm{diam}(G)$, i.e., $\mathrm{diam}(G)=\max\{\mathrm{dist}(U,V):U,V\subseteq V(G)\text{ with }|U|=|V|=1\}$. Let $V_1,\ldots, V_k$ be vertex subsets of $G$. Then we say a sequence of vertices $ v_1\cdots v_k$ is of \emph{form $V_1\cdots V_k$} if $v_i\in V_i$ for each $i\in [k]$.  

Let $\overrightarrow{G}$ be an orientation of $G$.  For a directed edge $e\in E(\overrightarrow{G})$ with endpoints $x,y$, both   $\overrightarrow{xy}$  and $\overrightarrow{yx}$ denote $e$ itself regardless of direction.
We write  by  $U\to V$ to  indicate that all the edges in $[U,V]$ are oriented from $U$ to $V$ in $\overrightarrow{G}$.   
A  {\it $(U, V)$-dipath} is a   dipath from $U$ to $V$ where the internal vertices are disjoint with $U\cup V$. 
Denote by $\partial(U, V)$ the length of a shortest $(U, V)$-dipath. If there is no such a path, define $\partial(U, V)=\infty$.   Additionally,  let $\theta(U, V)=\max\{\partial(U, V), \partial(V, U)\}$.   
For edges $ww_1$ and $ww_2$, if both are oriented toward $w$ or away from $w$, say they are of \emph{same} direction, otherwise, \emph{different}. 

If some subset $X=\{x\}$ in the above notions, we will replace $X$ by $x$ directly. 
For a (oriented) graph  $G$, denote by $e(G)$  the number of edges in it.

\medskip
\noindent
{\bf Convention.} In the coming sections, we fix $G=(V,E)$, $\mathrm{diam}(G)=4$ and $uv$ to be the edge such that $N(u)\cap N(v)=\emptyset$,  as described in \Cref{lemless5}.  The desired construction is given in a stepwise manner.    
Whenever a collection of orientation rules is given in some Step $i$, the rules are applied sequentially according to their order. More precisely, all edges corresponding to vertices satisfying the first rule are oriented first, then those satisfying the second rule, and so on. Take \Cref{step:F0X31X41}  for example,  all the  edges or vertices satisfying \ref{item:step:F0X31X411} are processed simultaneously, then all edges or vertices satisfying \ref{item:step:F0X31X412}, and so forth.   Within a rule, process the eligible choices one at a time and use the current status of every edge. A prescribed dipath is processed as one choice. Orient only the edges that are still unoriented, and never reverse an existing arc. Throughout the whole orientation process, the oriented graph is updated edge by edge once any edge is oriented. The final oriented graph obtained after  Step $i$ is denoted by  $G^{O_i}$, i.e.,  the oriented graph obtained from $G$ which only contains directed edges and corresponding vertices after  Step $i$ and by $E_i$   the set of directed edges resulted by Step $i$. Then  $G^{O_i}\subseteq G^{O_{i+1}}$, $E_i\cap E_j=\emptyset$ for $i\neq j$ and $E(G^{O_{i}})=\bigcup_{j\le i}E_j$. Denote by $M^O_i$
the resulted mixed  graph  after Step $i$, i.e., a graph whose edges are either from $E(G^{O_i})$ or unoriented. We emphasis here that if there is an $(x,y)$-dipath in $G^{O_i}$, then the path is also contained in $G^{O_{j}}$ for $j\ge i$. Therefore, if $\mu(x)\le k$ (resp. $\nu(x)\le k$) for some integer $k$ in $G^{O_i}$, then  $\mu(x)\le k$ in $G^{O_{j}}$  (resp. $\nu(x)\le k$)  for $j\ge i$.

\subsection{\texorpdfstring{Proof sketch of \Cref{lemless5}}{Proof sketch of the main construction}}\label{sec:sketch}

 The construction is motivated by the vital result, \Cref{lem:Gxy16}. Then the target is to orient edges of $G$ so that the obtained orientation of $G$ has the the same property as $G^O$.

First, in \Cref{sec:partition}, we study relationship between  two distinct partitions. The first partition, which is a classical partition commonly used in dealing with such problems, is obtained by considering the distances between vertices in $V\setminus \{u,v\}$ and $\{u,v\}$. The second   partition, introduced in this paper, is defined as $F_0,F_1,\ldots, F_4$ of $V$ with $F_0$ having the property that each vertex  in $F_0$ lies on a cycle passing   $uv$ of length at most 5 while $\mathrm{dist}(F_i,F_0)=i$ for each $i\in [4]$.

Second, we present orientation \Cref{step:F0X31X41}--\Cref{step:sumatmost8} in \Cref{sec:F1} which mainly handle edges incident to vertices in $F_1$. Especially, after each step, we show that if $x$ is incident to some edge that we just oriented, then $x$ satisfies that  $\mu(x)+\nu(x)\le 9$ where equality holds with certain condition, see \Cref{claimcenter} as a summary. The final construction in this section is \Cref{step:L23R23}, after which we obtained four subsets  $D_a,U_a,D_b,U_b$ of $V$ which form a cover of $V$:   each   $x\in D_a$  has $\nu(x)\le 4$;  each  $x\in U_a$  has  $\mu(x)\le 4$; each $x\in D_b$ satisfies that $[x,D_a]\neq\emptyset$,  and each $x\in U_b$ satisfies that $[x,U_a]\neq\emptyset$. Furthermore, all the edges incident to $D_b\cup U_b$ are unoriented after \Cref{step:L23R23}, see \Cref{claimcenter2}.

In \Cref{sec:F234}, paths loosely of form $U_aD_a$ and  $U_aU_bD_aD_b$ are reasonably oriented as $U_a\to D_a$ and  $U_a\to U_b\to D_a\to D_b$ in \Cref{step:DUb} so that the related vertices have $\mu,\nu$-values at most 7. In the final construction, \Cref{step:complicate}, we deal with edges in $D_a$, $D_b$, $[D_a,D_b]$,  $U_a$, $U_b$, $[U_a,U_b]$. Informally, to orient the edges in a optimal way, we introduce the definition of \emph{good path} and orient all the  edges  along the  good path  direction. Roughly for $x\in D_a\cup D_b$, if there is a $vx$-mixpath $P$ in $M^O_{\ref{step:DUb}}$ satisfying that 
\begin{itemize}
 \item  there exists a vertex $w_x\in V(P)$ such that the vertices in sub $w_xx$-path  are in $D$,  
\item edges in  the sub $w_xx$-path are either oriented in the direction towards $x$ or unoriented, 
 and 
    \item  the subpath, $vw_x$-path, is oriented as a dipath from $v$ to $w_x$ in $G^{O_{\ref{step:DUb}}}$, 
\end{itemize}
then we call $P$ a  {\it good path} of $x$ in $M^O_{\ref{step:DUb}}$ (in the definition $v$ could be replaced by $u$). Let $\ell(P)=e(P)$ if $P$ is a $vx$-path, and $\ell(P)=e(P)+2$, if $P$ is a $ux$-path. 
And define $$\U(x)=\min\{\ell(P): P \text{ is a good path for } x\}.$$
Then, for $x,y\in [D_a,D_a]\cup [D_b,D_b]\cup [D_a,D_b]$ with $\U(x)\ge \U(y)$, the basic rule is orienting $xy$ as $ y\to x$.  Similarly, we define $\D(x)$ for vertices in $U_a\cup U_b$ which we do not list here. Following the final construction, we show that $\nu(x)\le 6$ and $\mu(x)\le \U(x)+1$ for each $x\in D_a\cup D_b$ with a characterization of vertices such that $\mu(x)=\U(x)+1$, see \Cref{cl:complicateDmu} and \Cref{claimcenter4}, respectively.

The main remaining task is to show that for any distinct $x,y\in V(G)$, it holds that $\pt(x,y)\le 16$ in $G^{O_{\ref{step:complicate}}}$. More precisely,   we prove that $\U(x)\le 9$ for every $x\in D_a\cup D_b$ and that $\D(x)\le 9$ for every $x\in U_a\cup U_b$, with the precise bound depending on the set $F_i$ containing $x$. These estimates are then used in \Cref{sec:finalmain} to verify the desired inequality $\pt(x,y)\le 16$ for all $x,y\in V$. The only delicate cases are those in which the preliminary estimates allow the sum to exceed $16$; these cases are treated separately, completing the proof.

\section{The preliminary orientation}\label{sec:lemless5}

\subsection{\texorpdfstring{Partitions of $V$}{Partitions of V}}\label{sec:partition}
To prove Theorem~\ref{lemless5}, we begin by partitioning $V$ according to distances to $u$ and to $v$. Using this initial partition, we then construct a refined partition of $V$. Based on these two partitions, we proceed to orient the edges of $G$ in a stepwise manner.

First,   define a partition of $V\setminus\{u,v\}$ as follows, see \Cref{fig:completeedges}(a). For $i\ge 1$, let 
$$
\begin{aligned}
L_{i}&=\{w\in V :~\mathrm{dist}(w,u)=i,~\mathrm{dist}(w,v)=i+1\};\\
R_{i}&=\{w\in V :~\mathrm{dist}(w,u)=i+1,~\mathrm{dist}(w,v)=i\};\\
X_{i}&=\{w\in V :~\mathrm{dist}(w,v)=i,~\mathrm{dist}(w,u)=i\};\\
L&=L_1\cup L_2\cup L_3;\quad~~ R=R_1\cup R_2\cup R_3. 
\end{aligned}$$
 
Note that for any $w\in V $, if $\mathrm{dist}(w,u)=i$, then $i-1\leq \mathrm{dist}(w,v)\leq i+1$. Moreover, since $N(u)\cap N(v)=\emptyset$, it implies $X_1=\emptyset$. Also, $\mathrm{diam}(G)=4$ means that $\mathrm{dist}(w,u)\leq4$ and $\mathrm{dist}(w,v)\leq4$ for any $w\in V $. Then $L_1,L_2,L_3,R_1,R_2,R_3,X_2,X_3,X_4$ (some of them may be emptyset) is a partition of $ V\setminus\{u,v\}$.
 
We further partition all of  $L_1,L_2,L_3,R_1,R_2,R_3,X_2,X_3,X_4$ except for $X_2$ as follows to obtain the local structures that simplify our orientation process, see \Cref{fig:completeedges}(b).  
\begin{figure}[htbp]
\begin{center}
    \begin{tikzpicture}[
    rect/.style={rectangle, draw, minimum width=1cm, minimum height=0.6cm, align=center, thick, inner sep=2pt},
    circ/.style={circle,
    draw,
    fill=black,
    minimum size=0.15cm,
    inner sep=0pt},
    arr/.style={-{Stealth[scale=1.2]}, shorten >=1pt, shorten <=1pt},scale=0.65
]

\node[circ] (v) at (3,-0.2) {};
\node[circ] (u) at (-3,-0.2) {};

\node[rect] (r2) at (-3,1.5) {};
 
\node[rect] (r3) at (3,1.5) {}; 
 
\node[rect] (r6) at (-3,4 ) {};
\node[rect] (r7) at (3,4 ) {}; 

 
\node[rect] (r10) at (-3,6.5) {};
\node[rect] (r11) at (3,6.5) {}; 

\node[rect] (r13) at (0,2.75) {};
\node[rect] (r14) at (0,5.25) {};
\node[rect] (r15) at (0,8) {};

\node[below]  at (v) {$v$}; 
\node[below]  at (u) {$u$}; 
 
\node at (r3) {$R_{1 }$};
\node at (r2) {$L_{1 }$}; 

\node at (0,-1) {(a)}; 

\node at (r6) {$L_{2 }$};
\node at (r7) {$R_{2 }$};

\node at (r10) {$L_{3  }$};
\node at(r11) {$R_{3 }$};

\node  at (r13) {$X_2$};
\node  at (r14) {$X_{3 } $};
\node at (r15)   {$X_{4 }$};

\end{tikzpicture}\hspace{1cm}
\begin{tikzpicture}[
    rect/.style={rectangle, draw, minimum width=0.6cm, minimum height=0.6cm, align=center, thick, inner sep=2pt},
 rect1/.style={rectangle, draw, minimum width=0.6cm, minimum height=0.6cm, align=center, thick, inner sep=2pt},
    circ/.style={circle,
    draw,
    fill=black,
    minimum size=0.15cm,
    inner sep=0pt},
    arr/.style={-{Stealth[scale=1.2]}, shorten >=1pt, shorten <=1pt},scale=0.65
]

\node[circ] (v) at (3,-0.2) {};
\node[circ] (u) at (-3,-0.2) {};

\node[rect] (r1) at (-4,1.5) {};
\node[rect] (r2) at (-2,1.5) {};
\node[rect] (l5) at (-6,1.5) {};
\node[rect] (r3) at (2,1.5) {};
\node[rect] (r4) at (4,1.5) {};
\node[rect] (l6) at (6,1.5) {};
\node  at (0,-1) {(b)};

\node[rect] (r5) at (-4,4) {};
\node[rect] (r6) at (-2,4) {};
\node[rect] (r7) at (2,4) {};
\node[rect] (r8) at (4,4) {};
\node[rect] (l7) at (-6,4) {};
\node[rect] (l8) at (6,4) {};


\node[rect] (r9) at (-4,6.5) {};
\node[rect] (r10) at (-2,6.5) {};
\node[rect] (r11) at (2,6.5) {};
\node[rect] (r12) at (4,6.5) {};

\node[rect] (r13) at (-0,2.75) {};
\node[rect1] (r14) at (-0.8,5.25) {};
\node[rect1] (r15) at (-0.8,8.75) {};
\node[rect1] (r16) at ( 0.8,5.25) {};
\node[rect1] (r17) at ( 0.8,8.75) {};
\draw (v) -- (r3);
\draw (r1) -- (u);
\draw  (v) -- (r4);
\draw  (r4) -- (r7);
\draw (v) -- (l6);
\draw (u) -- (l5);
\draw  (r3) -- (r7);
\draw  (r7) -- (r11);
\draw  (r11) -- (r15);
\draw  (r15) -- (r10);
\draw  (r10) -- (r6);
\draw  (r6) -- (r2);
\draw (r2) -- (u);
\draw  (r6) -- (r1);

\draw  (r5) -- (l5);
\draw  (r5) -- (r1);
\draw  (r5) -- (r2);
\draw  (r8) -- (l6);
\draw  (r8) -- (r3);
\draw  (r8) -- (r4);

\draw  (l7) -- (l5);
\draw  (l7) -- (r1); 
\draw  (l8) -- (l6);
\draw  (l8) -- (r4);
\draw  (r3) -- (r2);
\draw  (r7) -- (r6);
\draw  (r11) -- (r10);
\draw  (r7) -- (r13);
\draw   (r13) -- (r6);
\draw  (r7) -- (r14);
\draw  (r14) -- (r6);
\draw  (r9) -- (r6);
\draw  (r9) -- (r5);
\draw  (r9) -- (l7);
\draw  (r12) -- (r7);
\draw  (r12) -- (r8);
\draw  (r12) -- (l8);
 
\draw  (r3) -- (r13);
\draw  (r13) -- (r2);
\draw (u) -- (v);
\draw  (r14) -- (r13);
\draw  (r16) -- (r13);
\draw  (r15) -- (r14);
\draw  (r15) -- (r16);
\draw  (r14) -- (r17);
\draw  (r17) -- (r16);
\draw  (r17) -- (r11);
\draw  (r17) -- (r12);

\node[below]  at (v) {$v$}; 
\node[below]  at (u) {$u$};

\node at (r1) {$L_{12}$};
\node at (r3) {$R_{11}$};
\node at (r2) {$L_{11}$};
\node at  (l5) {$L_{13}$};
\node at   (r4) {$R_{12}$};
\node at  (l6) {$R_{13}$};

\node at (r6) {$L_{21}$};
\node at (r7) {$R_{21}$};
\node at (r5) {$L_{22}$};
\node at (r8) {$R_{22}$};
\node at (l7) {$L_{23}$};
\node at (l8) {$R_{23}$};

\node at (r10) {$L_{31}$};
\node at(r11) {$R_{31}$};
\node at (r9) {$L_{32}$};
\node  at (r12) {$R_{32}$};

\node  at (r13) {$X_2$};
\node  at (r14) {$X_{31}$};
\node at (r15)   {$X_{41}$};

\node  at (r16) {$X_{32}$};
\node at (r17)   {$X_{42}$};

\end{tikzpicture}
\caption{(a) denotes the partition $L\cup R\cup X_2\cup X_3\cup X_4\cup \{u,v\}$. (b) denotes the further partition obtained from  $L\cup R\cup X_2\cup X_3\cup X_4\cup \{u,v\}$, moreover, all possible edges    are presented except for edges in $[X_{31}\cup X_{32}, L_3\cup R_3]\cup [X_{42}, L_3]$,   edges in $L_i$, $R_i$ and $X_j$ where $i\in [3]$ and $j\in \{2,3,4\}$ and edges in $[X_2,L_{22}\cup L_{23}\cup R_{22}\cup R_{23}]$.}\label{fig:completeedges}
\end{center}
\end{figure}
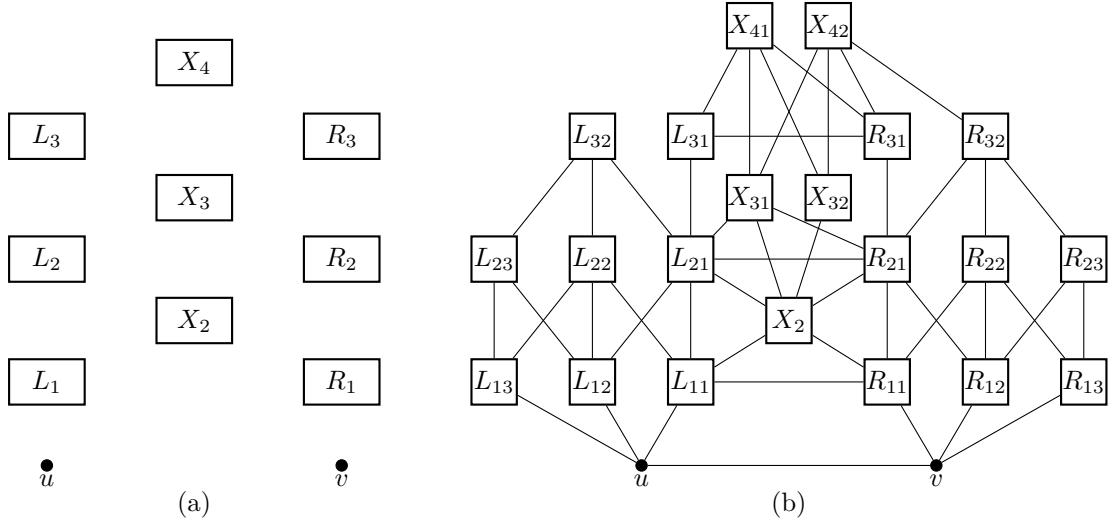
\begin{align*}
X_{41}&=\{w\in X_4: N(w)\cap  L_3\neq \emptyset \text{ and }N(w)\cap  R_3\neq \emptyset\},  &X_{42}&=X_4\setminus X_{41};\\[1ex] 
X_{31}&=\{w\in X_3: |N(w)\cap (L_2\cup X_2\cup R_2)|\ge 2 \},   &X_{32}&=X_3\setminus X_{31};\\[1ex] 
\tilde{X}_3&= \{w\in X_3:N(w)\cap (L_2\cup R_2\cup (N(X_2)\cap X_3))=\emptyset\}; &~&\\[1ex] 
L_{31}&=\{w\in L_3:~N(w)\cap(R_3\cup X_{41}\cup \tilde{X}_3)\neq\emptyset\},    &L_{32}&=L_3\setminus L_{31};\\[1ex] 
R_{31}&=\{w\in R_3:~N(w)\cap(L_3\cup X_{41}\cup \tilde{X}_3)\neq\emptyset\},   &R_{32}&=R_3\setminus R_{31};\\[1ex] 
L_{21}&=\{w\in L_2:~N(w)\cap(R_2\cup X_{31}\cup L_{31})\neq\emptyset\}; &~&\\[1ex] 
R_{21}&=\{w\in R_2:~N(w)\cap(L_2\cup X_{31}\cup R_{31})\neq\emptyset\}; &~&\\[1ex] 
L_{11}&=\{w\in L_1:~N(w)\cap   N[R_1] \neq\emptyset\}, &L_{12}&=\{w\in L_1\setminus L_{11}:~N(w)\cap L_{21}\neq\emptyset\}, \\
L_{13}&=L_1\setminus (L_{11}\cup L_{12}); &~&\\[1ex] 
R_{11}&=\{w\in R_1:~N(w)\cap N[L_1] \neq\emptyset\}, &R_{12}&=\{w\in R_1\setminus R_{11}:~N(w)\cap R_{21}\neq\emptyset\},\\
R_{13}&=R_1\setminus(R_{11}\cup R_{12}); &~&\\[1ex] 
L_{22}&=\{w\in L_2\setminus L_{21}:~N(w)\cap L_{11}\neq\emptyset\},  &L_{23}&=L_2\setminus(L_{21}\cup L_{22});\\[1ex] 
                                                                                                                                                                                                                                                                                                                                                                                                                                                                                                                                                   R_{22}&=\{w\in R_2\setminus R_{21}:~N(w)\cap R_{11}\neq\emptyset\},  &R_{23}&=R_2\setminus(R_{21}\cup R_{22}). \end{align*}  

Notice that  the subsets $L_{ij},R_{ij}(i=1,2;j=1,2,3),L_{31},L_{32},R_{31},R_{32}$ and $X_2,X_{31},X_{32},$ $ X_{41},X_{42}$ is a partition of $V \setminus \{u,v\}$ by definition. Also,  $G[\{u\}\cup L_{11}\cup X_2\cup R_{11}\cup\{v\}]$ is a structure similar to a 5-cycle. Based on this,  we present a new partition of  $V $, where each vertex is assigned to a layer according to its distance to $\{u\}\cup L_{11}\cup X_2\cup R_{11}\cup\{v\}$. Let 

$$\begin{aligned}F_0&=\{u\}\cup L_{11}\cup X_2\cup R_{11}\cup\{v\}.\\
F_i&=\{x\in V :\mathrm{dist}(x,F_0)=i\}~(i=1,2,3,4).
\end{aligned}$$
Then $F_i~(i=0,1,2,3,4)$ is a partition of $V $ since $\mathrm{diam}(G)=4$.

Moreover,  we have the following observations based on our partitions.
\begin{observation}\label{claim1}
All the followings hold. 
\begin{itemize}
\item  $N(w)\cap X_3\neq\emptyset$  for $w\in X_{42}$ and  $N(w)\cap X_2\neq\emptyset$ for $w\in X_{32}\cup \tilde{X}_3$. 

 If $w\in X_{42}$, then the definition of $X_{42}$  implies that $\mathrm{dist}(w,u)=\mathrm{dist}(w,v)=4$ and $ N(w)\cap  L_3 =\emptyset$ or $ N(w)\cap  R_3 =\emptyset$ and so  $N(w)\cap X_3\neq\emptyset$. The case for $X_{32}\cup \tilde{X}_3$ is analogous.
\item $ [L_{12}\cup L_{13},R_1]\cup[L_{22}\cup L_{23},R_2] \cup [L_{32} ,R_3] =\emptyset$,  $  [L_{21},L_{13}]\cup [L_{23},L_{11}]=\emptyset$ and  $  [L_{31},L_{22}\cup L_{23}] =\emptyset$. The analougous hold when change the role of $R$ and $L$. 
\item $[X_2, L_{12}\cup L_{13}\cup R_{12}\cup R_{13}]=\emptyset$, $[  X_3,L_{23}\cup L_{22}\cup R_{22}\cup R_{23}]=\emptyset$, $[X_{32}, L_{21}\cup R_{21}]=\emptyset$ and $[X_{41},R_{32}\cup L_{32}]=\emptyset$. 
\item For $x\in X_{42}$, if $[x,R_3]\neq\emptyset$, then $[x,L_3]=\emptyset$; if $[x,L_3]\neq\emptyset$, then $[x,R_3]=\emptyset$.  
\end{itemize}
\end{observation}

\begin{observation}\label{claim3}
    $L_{12}\cup L_{13}\cup R_{12}\cup R_{13}\cup L_{22}\cup R_{22}\cup X_{32}\subset F_1$, $L_{21}\cup R_{21}\cup L_{23}\cup R_{23}\subset F_1\cup F_2$, $L_{3}\cup R_{3}\subset F_2\cup F_3$, $X_{31}\subset F_1\cup F_2\cup F_3$, and $X_4\subset F_2\cup F_3\cup F_4$.
\end{observation}

 Denote $I_1=\{u\}$, $I_2=L_{11}$, $I_3=X_2$, $I_4=R_{11}$ and $I_5=\{v\}$ for convenience.   Additionaly, by the definition of $I_k$, it holds that $1\le |\{k\in [5]: N(x)\cap I_k\neq\emptyset\}|\le 2$  for each $x\in F_1$. Furthermore,  if $|\{k\in [5]:N(x)\cap I_k\neq\emptyset\}|= 2$, then $\{k\in [5]:N(x)\cap I_k\neq\emptyset\}=\{i,i+1\}$ for some $i\in [4]$. For each $x\in F_1$, let 
\begin{align*}
f(x)=\left\{\begin{array}{ll}
 i  \quad&\text{if} ~\{i\}=\{k\in [5]:N(x)\cap I_k\neq\emptyset\};\\
 \frac{2i+1}{2} \quad  &\text{if}~ \{i,i+1\}=\{k\in [5]:N(x)\cap I_k\neq\emptyset\}.
\end{array}\right.
\end{align*}
We know that $f(x)$ is well-defined. 

We record the lower bounds for the adjusted distances that are used
in the subsequent arguments. Every dipath is an underlying walk, so
in any partial orientation $H$ we have
 \begin{align*}
\mu_H(x)&\ge\min\{\mathrm{dist}_G(v,x),\mathrm{dist}_G(u,x)+2\},\\
\nu_H(x)&\ge\min\{\mathrm{dist}_G(x,u),\mathrm{dist}_G(x,v)+2\}.
\end{align*}
Consequently, the respective lower bounds for $(\mu_H(x),\nu_H(x))$
are $(i+1,i)$ on $L_i$, $(i,i+1)$ on $R_i$, and $(i,i)$ on $X_i$.

\medskip

 We construct an orientation with $\partial(u,v)=1$
and $\partial(v,u)\le4$. By \Cref{lem:Gxy16}, pairs satisfying
$\nu(x)+\mu(y)\le15$ have directed distance at most $16$;
the remaining pairs are treated in \Cref{sec:finalmain}.  A general but not exactly idea is to orient $u\to v\to R\to L\to u$,  $u\to L\to L\to u$, and $v\to R\to R\to v$.  We start with  edges related to  $F_1$.

\subsection{\texorpdfstring{Orient edges related to  $F_1$}{Orient edges related to F1}}\label{sec:F1}

Note that by the definition of $X_4$ and $X_{41}$, each vertex in $X_{41}$ lies on a  closed trail containing $uv$.  Similarly, by definition, each vertex in $L_{31}\cup R_{31}\cup L_{21}\cup R_{21}\cup L_{11}\cup R_{11}\cup X_2$ also lies on a cycle containing $uv$. These cases are straightforward to handle, as  it could be oriented as a dicycle.  Therefore, we first orient edges involved in these cycles.

\begin{step}\label{step:F0X31X41}
The following edges are oriented in order, as shown in  \Cref{fig:step1}. 
\begin{enumerate}[label=(S1.\arabic*)]

 \item\label{item:step:F0X31X411} Orient $u\to v\to R_{11}\to X_2\to L_{11}\to u$, $v\to R_{12}\to R_{21}\to L_{21}\to L_{12}\to u$, $R_{21}\to R_{31}\to L_{31}\to L_{21}$,  $R_{11}\to L_{11}$, $R_{11}\to R_{21}$, $L_{21}\to L_{11}$, $R_{31}\to X_{41}\cup \tilde{X}_3\to L_{31}$,  $R_{21}\to  X_{31}\to L_{21}$, $R_{21}  \to X_2$ and $X_2\to L_{21} $
 \item\label{item:step:F0X31X412} Let $x\in X_{31}$. If $N(x)\cap R_2=\emptyset$ and $N(x)\cap L_2\neq\emptyset$, orient $X_2\to x$. If $N(x)\cap R_2\neq \emptyset$, orient $x\to X_{2}$.  
 
 Let $x\in \tilde{X}_3$. If $N(x)\cap R_{31}=\emptyset$ and $N(x)\cap L_{31}\neq \emptyset$, oriented $  X_2\to x$. If $N(x)\cap R_{31}\neq \emptyset$, oriented $ x\to X_2$. 
 \item\label{item:step:F0X31X413} For each $x\in L_{11}$ with $N(x)\cap (X_2\cup R_1)=\emptyset$, fix arbitrary $y\in N(x)\cap L_{11}$ with $N(y)\cap R_1\neq\emptyset$, orient $y\to x$. 
 
 For each $x\in R_{11}$ with $N(x)\cap (X_2\cup L_1)=\emptyset$, fix arbitrary $y\in N(x)\cap R_{11}$ with $N(y)\cap L_1\neq\emptyset$, orient $ x\to y$.

 \item\label{item:step:F0X31X414} 

For all paths $P=vv_1v_2v_3v_4$   of form $vR_1R_2R_{21}L_{21}$ in $G$ where  $N(v_2)\cap (X_2\cup X_3\cup L_2)=\emptyset $, orient edges in  $ vv_1v_2v_3v_4$ as $v\to v_1\to v_2\to v_3$.    

For all paths $Q=uu_1u_2u_3u_4$  of form $uL_1L_2L_{21}R_{21}$ where $N(u_2)\cap (X_2\cup X_3\cup R_2)=\emptyset $, orient edges in  $uu_1u_2u_3$ as $u_3\to u_2\to u_1\to u$.  

 \item\label{item:step:F0X31X415} 
 For all $v_1v_1'$ satisfying that  $v_1\in R_{12}$, $v'_1\in (R_{12}\cup R_{13})\setminus (N(R_{11})\cup N(R_{21}\cap N(L_2)))$ and there is a dipath $v\to   v_1\to v_2\to v_3$ such that   $v_2\in R_{21}$ and $v_3\in L_{21}$,  orient the edges in  $vv_1'v_1$ such that  $v\to v'_1\to v_1$.

 For all $u_1u_1'$ satisfying that  $u_1\in L_{12}$, $u'_1\in (L_{12}\cup L_{13})\setminus (N(L_{11})\cup N(L_{21}\cap N(R_2)))$ and there is a dipath $u_3\to   u_2\to u_1\to u$ such that   $u_2\in L_{21}$ and $u_3\in R_{21}$,  orient the edges in  $uu_1'u_1$ such that  $ u_1\to u'_1\to u$ 
\end{enumerate}
\end{step}

 \begin{figure}[htbp]
\begin{center}
    
\begin{tikzpicture}[
    rect/.style={rectangle, draw, minimum width=0.8cm, minimum height=0.6cm, align=center, thick, inner sep=2pt},
    rect1/.style={rectangle, draw, minimum width=1.5cm, minimum height=0.8cm, align=center, thick, inner sep=2pt},
    circ/.style={circle,
    draw,
    fill=black,
    minimum size=0.15cm,
    inner sep=0pt},
    arr/.style={-{Stealth[scale=1.2]}, shorten >=1pt, shorten <=1pt},scale=0.85
]

\node[circ] (v) at (3,-0.2) {};
\node[circ] (u) at (-3,-0.2) {};

\node[rect] (r1) at (-4,1.5) {};
\node[rect] (r2) at (-2,1.5) {};
\node[rect] (l5) at (-6,1.5) {};
\node[rect] (r3) at (2,1.5) {};
\node[rect] (r4) at (4,1.5) {};
\node[rect] (l6) at (6,1.5) {};

\node[rect] (r5) at (-4,3.5) {};
\node[rect] (r6) at (-2,3.5) {};
\node[rect] (r7) at (2,3.5) {};
\node[rect] (r8) at (4,3.5) {};
\node[rect] (l7) at (-6,3.5) {};
\node[rect] (l8) at (6,3.5) {};


\node[rect] (r9) at (-4,5.5) {};
\node[rect] (r10) at (-2,5.5) {};
\node[rect] (r11) at (2,5.5) {};
\node[rect] (r12) at (4,5.5) {};

\node[rect] (r13) at (0,2.5) {};
\node[rect1] (r14) at (0,4.5) {};
\node[rect] (r15) at (0,6.5) {};

\draw[arr] (v) -- (r3);
\draw[arr] (r1) -- (u);
\draw[arr] (v) -- (r4);
\draw[arr] (r4) -- (r7);
\draw[arr] (r3) -- (r7);
\draw[arr] (r7) -- (r11);
\draw[arr] (r11) -- (r15);
\draw[arr] (r15) -- (r10);
\draw[arr] (r10) -- (r6);
\draw[arr] (r6) -- (r2);
\draw[arr] (r2) -- (u);
\draw[arr] (r6) -- (r1);

\draw[arr] (r3) -- (r2);
\draw[arr] (r7) -- (r6);
\draw[arr] (r11) -- (r10);
\draw[arr] (r7) -- (r13);
\draw[arr] (r13) -- (r6);
\draw[arr] (r7) -- (r14);
\draw[arr] (r14) -- (r6);

\draw[arr] (r11) -- (r14);
\draw[arr] (r14) -- (r10);
\draw[arr] (r3) -- (r13);
\draw[arr] (r13) -- (r2);
\draw[arr] (u) -- (v);
\node[below] at (v) {$v$}; 
\node[below] at (u) {$u$};

\node at (r1) {$L_{12}$};
\node at (r3) {$R_{11}$};
\node at (r2) {$L_{11}$};
\node at  (l5) {$L_{13}$};
\node at   (r4) {$R_{12}$};
\node at  (l6) {$R_{13}$};

\node at (r6) {$L_{21}$};
\node at (r7) {$R_{21}$};
\node at (r5) {$L_{22}$};
\node at (r8) {$R_{22}$};
\node at (l7) {$L_{23}$};
\node at (l8) {$R_{23}$};

\node at (r10) {$L_{31}$};
\node at(r11) {$R_{31}$};
\node at (r9) {$L_{32}$};
\node  at (r12) {$R_{32}$};

\node  at (r13) {$X_2$};
\node  at (r14) {$X_{31}\cup \tilde{X}_3$};
\node at (r15)   {$X_{41}$};

\end{tikzpicture}
\caption{An example for   partial orientations in \Cref{step:F0X31X41}.}\label{fig:step1}
\end{center}
\end{figure}
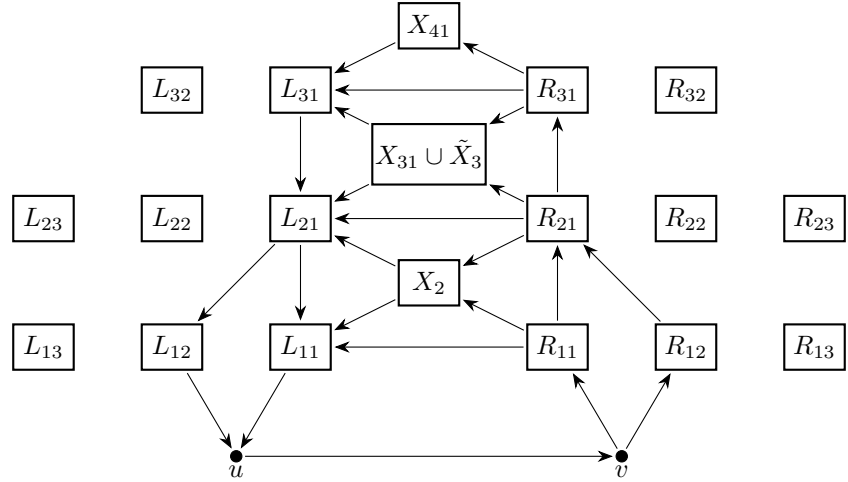



Observe that \ref{item:step:F0X31X415} has no conflict with \ref{item:step:F0X31X411}.  
By \Cref{step:F0X31X41}, we have that $I_1\to I_5\to I_4\to I_3\to I_2\to I_1$ and $I_4\to I_2$ in $G^{O_{\ref{step:F0X31X41}}}$.   Recall for each $x\in V$, we define  $$
  \begin{aligned}
\nu(x):=
\min\{\partial(x,u), \partial(x,v)+2\}.\\
\mu(x):=
\min\{\partial(v,x), \partial(u,x)+2\}. 
\end{aligned}
 $$ 
Now, for a path $P=v_1 \cdots  v_k$ where  all the edges are treated  as a whole and oriented  simultaneously so as to form a dipath $\orw{P}$,   we define $$\omega(\orw{P})=\mu(v_1)+e(P)+\nu(v_k).$$ 
Observe that  for each $e\in E(P)$, it holds that $\omega(\orw{e})\le \omega(\orw{P})$. 
 
\begin{observation}\label{obs:munu} All the following hold. 
\begin{itemize}
   \item In $G^{O_{\ref{step:F0X31X41}}}$, for $x\in I_k$ where $k\in [5]$, we have that  $\partial(x,u)\leq k-1$ and $\partial(v,x)\leq5-k$, which implies that  $\nu(x)\leq k-1$ and $\mu(x)\leq 5-k$. 
\item  In \ref{item:step:F0X31X414}, it holds that $v_2\in R_{21}$ with $ N(v_2) \cap (X_2\cup L_2\cup X_3)=\emptyset$ or $v_2\in R_{22}$ with $f(v_2)=4$ or $v_2\in F_2\cap  R_{23} $;  $u_2\in L_{21}$ with $ N(u_2) \cap (X_2\cup R_2\cup X_3)=\emptyset$ or $u_2\in L_{22}$ with $f(u_2)=2$ or $u_2\in F_2\cap (L_{21}\cup L_{23})$. 
\item  In \ref{item:step:F0X31X415}, it holds that $f(v'_1)=5$ and $f(u'_1)=1$. 
\item For $x,y\in F_1$ with $xy\in E$, if $f(x)<f(y)=3$ and $y\notin X_3\cup L_{21}\cup R_{21}$, then there is no $z\in N(y)\cap F_1$ such that $f(z)>f(y)$;    if $3=f(x)<f(y)$ and $x\notin X_3\cup L_{21}\cup R_{21}$, then there is no $z\in N(x)\cap F_1$ such that $f(z)<f(x)$.   
\end{itemize}
\end{observation}
We shall also use the following consequence of the core definition. In every extension of the initial orientation,
\[
\mu(x)\le2\text{ and }\nu(x)\le2\quad\Longrightarrow\quad x\in F_0.
\]

\begin{proposition}\label{prop:xinf0munu}
 For $e\in E_{\ref{step:F0X31X41}}$, it holds that $\omega(e)\le 8$ in $G^{O_{\ref{step:F0X31X41}}}$. 
\end{proposition}
\begin{proof}
 Since all the edges oriented in \Cref{step:F0X31X41} are contained in a closed diwalk passing $uv$ of length at most 9, we are done.
\end{proof}

Next, we proceed to orient the edges incident to vertices in $F_1 \setminus V(G^{O_{\ref{step:F0X31X41}}})$, as described in \Cref{step:F12}--\Cref{step:F13}. Prior to this, we partition the set $F_1 \setminus V(G^{O_{\ref{step:F0X31X41}}})$ based on the neighborhoods of its vertices. 
Since $G$ is bridgeless, by the definition of $F_1$, every vertex in $F_1 \setminus V(G^{O_{\ref{step:F0X31X41}}})$ has at least one neighbour contained in $F_0$. We then partition $F_1 \setminus V(G^{O_{\ref{step:F0X31X41}}})$ into the following three subsets. 
$$\begin{aligned}
F_{11}&=\{x\in F_1\setminus V(G^{O_{\ref{step:F0X31X41}}}):N(x)\cap F_1\neq\emptyset\};\\
F_{12}&=\{x\in F_1\setminus V(G^{O_{\ref{step:F0X31X41}}}):N(x)\cap F_1=\emptyset \text{ and }|N(x)\cap F_0|=1\};\\
F_{13}&=\{x\in F_1\setminus V(G^{O_{\ref{step:F0X31X41}}}):  N(x)\cap F_1=\emptyset \text{ and } |N(x)\cap F_0|\geq2\}.    
\end{aligned}$$
We orient edges incident to  vertices in $F_{11}$ in \Cref{step:F12}, $F_{12}$ in \Cref{step:cycle(a)}, and $F_{13}$ in \Cref{step:F13}. 

For $x\in F_{11}$, by the definition of $F_{11}$, we know that there is $y\in F_0$ and  $ x'\in N(x)\cap F_1$ with $y'\in N(x')\cap F_0$.  Since the mapping $f$ determines the positions of $x, x',y,y'$, the coming orientation follows from the comporation of $f(x)$ and $f(x')$.  For example, $f(x)<f(x')$ implies that $y$ is closer to $u$ than $y'$ along $I_5I_4I_3I_2I_1$, naturally, we orient $x'\to x$.  For the edges in $[x,F_0]$, the oritation depends on $f(N(x)\cap F_1)$.

\begin{step}\label{step:F12}
We orient the following edges in order, as shown in \Cref{fig:step2}.

\begin{enumerate}[label=(S2.\arabic*)]
 \item\label{item:F121} For  $x\in F_{1}$  with $f(x)$ being a non-integer,  orient $I_{f(x)+\frac{1}{2}}\to x$ and $x\to I_{f(x)-\frac{1}{2}}$.
\item\label{item:F122} For all $x\in F_{11}$ with $f(x)\in [5]$ such that no edge in
$[x,I_{f(x)}]$ has been oriented, perform the following steps in order:
 
     \begin{enumerate}[label=(S2.2.\arabic*)]
     \item\label{item:F12221} for every such $x$ satisfying  that  $f(x)\ge f(x')$ for all $x'\in N(x)\cap F_1$ and there is $x'\in  N(x)\cap F_1$ such that $f(x)>f(x')$, orient  $I_{f(x)}\to x$. 
         \item\label{item:F12222} for  every such $x$ satisfying   that $f(x)\le f(x')$ for all $x'\in N(x)\cap F_1$ and there is $x'\in  N(x)\cap F_1$ such that $f(x)<f(x')$, orient $ x\to I_{f(x)}$.
          \item\label{item:F12223}  for  every such $x$ satisfying that  there are $x',x''\in N(x)\cap F_1$ such that   $f(x')> f(x)>f(x'')$, orient $x\to I_{f(x)}$ when $f(x)<3$ and $I_{f(x)}\to x$ when $f(x)>3$.  
         
         \item\label{item:F12224} for  every such $x$ with $f(x)\in \{1,5\}$ satisfying that   $f(x)=f(x')$ for all $x'\in N(x)\cap F_1$, fix some $z\in N(x)\cap F_1$, and  orient the unoriented edges in  the paths and cycles of form $I_{f(x)}xzI_{f(x)}$  so that every resulting cycle is a dicycle.  
   \end{enumerate}
   
 \item\label{item:F123}  For unoriented $xy\in [F_1,F_1]$ with $f(x)>f(y)$, orient $x\to y$. 
 \end{enumerate}   
\end{step} 

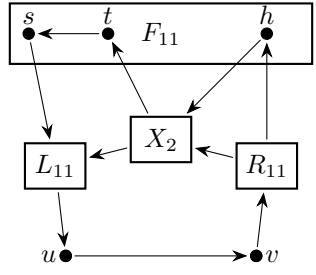
\begin{figure}[htbp]
\begin{center}
\begin{tikzpicture}[
    rect/.style={rectangle, draw, minimum width=0.8cm, minimum height=0.6cm, align=center, thick, inner sep=2pt},
       rect1/.style={rectangle, draw, minimum width=4cm, minimum height=0.8cm, align=center, thick, inner sep=2pt},
    circ/.style={circle,
    draw,
    fill=black,
    minimum size=0.15cm,
    inner sep=0pt},
    arr/.style={-{Stealth[scale=1.2]}, shorten >=1pt, shorten <=1pt},scale=0.7
]

\node[circ] (v) at (1.8,-0.2) {};
\node[circ] (u) at (-1.8,-0.2) {};

\node[rect] (r1) at (0,2) {};
\node[rect] (r2) at (-2,1.5) {};
\node[rect] (r3) at (2,1.5) {};
\node[rect1] (r4) at (0,4) {};
\node[circ] (v1) at (2,4) {};
\node[circ] (v2) at (-1,4) {};
\node[circ] (v3) at (-2.5,4) {};
\draw[arr] (r3) -- (v1);
\draw[arr] (v1) -- (r1);
\draw[arr] (r1) -- (v2);
\draw[arr] (v2) -- (v3);
\draw[arr] (v3) -- (r2);

\draw[arr] (v) -- (r3);
\draw[arr] (r3) -- (r1);
\draw[arr] (r1) -- (r2);
\draw[arr] (r2) -- (u);
\draw[arr] (u) -- (v);

\node[circ] (v) at (1.8,-0.2) {};
\node[circ] (u) at (-1.8,-0.2) {};

\node[right] at (v) {$v$};
\node[left]  at (u)  {$u$};
\node at (r1) {$X_2$};
\node  at (r2) {$L_{11}$};
\node at (r3) {$R_{11}$};
\node at (r4) {$F_{11}$};
\node[above]  at (v1) {$h$};
\node[above] at (v2) {$t$};
\node[above] at (v3) {$s$};

\end{tikzpicture}
\caption{An example for \Cref{step:F12} where $s,t,h\in F_{11}$ with $f(s)=2,$ $f(t)=3$ and $f(h)=3.5$.}\label{fig:step2}
\end{center}
\end{figure}

 \begin{remark}\label{rem:afterstep2}
 The following hold. 
\begin{itemize}
    \item[(1)]  \Cref{step:F12} only orients some edges in $ [F_1,F_1]\cup [F_1,F_0]$. 
     \item[(2)] The orientations in \ref{item:F121} and     \ref{item:F123} have no conflict with that in \Cref{step:F0X31X41}, i.e., for $x,y\in F_1$ with $xy\in E$, if $f(x)>f(y)$, then $x\to y \in E_1\cup E_2$. 
    \item[(3)] For each $z\in F_{11}$ with $f(z)$ being an integer, it holds that for $w_1\neq w_2\in N(z)\cap I_{f(z)}$, the edges $w_1z$ and $w_2z$ are oriented in the same direction.
  \item[(4)] In \ref{item:F12223}, we have that  
            $x''\in L_{12}\cup L_{13}$, $x\in L_{22}$ and $x'\in L_2$  with   $f(x'')\in \{1,1.5\}$, $f(x)=2$ and  $f(x')\in \{2.5,3\}$, or $x'\in R_{12}\cup R_{13}$, $x\in R_{22}$ and $x''\in R_2$  with   $f(x')\in \{5,4.5\}$, $f(x)=4$ and  $f(x'')\in \{3.5,3\}$. 
%
         
\end{itemize} 
\end{remark} 

\begin{proposition}\label{prop:F118} 
For $e\in E_{\ref{step:F12}}$, it holds that  $\omega(e)\leq 7$ in $G^{O_{\ref{step:F12}}}$.  
\end{proposition}
\begin{proof}
   Let $\orw{st}\in E_{\ref{step:F12}}$. We proceed by analyzing where $\orw{st}$ is oriented. Note that if $\orw{st}$ is oriented in \ref{item:F121}, we are done. 
   
   First, suppose that $\orw{st}$ is oriented in \ref{item:F123}.
Suppose that $s,t\in V(G^{O_{\ref{step:F0X31X41}}})$.
Write $f(s)>f(t)$, so the direction is $s\to t$.
We compute the parameters after \Cref{step:F12}.
If $f(s)$ is not an integer, \ref{item:F121} gives
$\mu(s)\le5-\lfloor f(s)\rfloor$.
If $f(s)\in\{4,5\}$, then
\Cref{step:F0X31X41} gives $\mu(s)\le2$.

Suppose that $f(s)=3$.
Then $s\in L_{21}\cup R_{21}\cup X_{31}$.
Indeed, a vertex in $X_3$ with a neighbor of smaller
$f$-value has a neighbor in $L_2$, and hence belongs to
$X_{31}$; \ref{item:step:F0X31X414} processes no vertex
of $L_2\cup R_2$ adjacent to $X_2$.
By \ref{item:step:F0X31X411} and
\ref{item:step:F0X31X412}, it follows that $\mu(s)\le3$.

Suppose that $f(s)=2$.
Then $t\in L_1\setminus L_{11}$.
If $s\in L_{21}$, the edge $st$ would already be oriented
in \ref{item:step:F0X31X411}, a contradiction.
Thus, $s\in L_{22}$ and it is processed in
\ref{item:step:F0X31X414}, which gives $\mu(s)\le4$.
Therefore, $\mu(s)\le4$ when $f(s)\le3$, and
$\mu(s)\le2$ when $f(s)>3$.
The symmetric argument gives $\nu(t)\le2$ when $f(t)<3$,
and $\nu(t)\le4$ when $f(t)\ge3$.
Since $f(s)>f(t)$, we obtain that
$\omega(s\to t)=\mu(s)+1+\nu(t)\le7$.

Otherwise, at least one of $s,t$ belongs to $F_{11}$,
and it suffices to prove the following claim. 

   \begin{claim}\label{claim:fy>fx}
    If   $s\in F_{11}$ and $f(t)>f(s)$ or  $s\in F_{11}$ and $f(t)<f(s)$, then $\omega(\overrightarrow{st})\le 7$. 
   \end{claim}
\begin{proofofclaim}
    We only prove the result for  $s\in F_{11}$ and $f(t)>f(s)$. If both   $f(s)$ and $f(t)$ are non-integers, then $\omega(\overrightarrow{st})\le 6$. We proceed according to the integrality of $f(s)$ and $f(t)$, and  divide the proof into the following two cases. 
 \medskip
 
\noindent {\bf Case 1.}  $f(s)$ is a non-integer and $f(t)$ is an integer or   $f(s)$ is an integer and $f(t)$ is a non-integer.  

       We prove the result for the former case.    It suffices to assume that $tw$ is oriented as $t\to w$  for all the  $tw$ where $w\in N(t)\cap I_{f(t)}$.  
       
       Suppose that $t\in F_{11}$. Since $f(t)>f(s)$ and the assumption that $t\to w$,  it follows that  $f(t)<3$. Then $f(s)=1.5$, $f(t)=2$ and there are $t'\in N(t)\cap L_2$ with $f(t')\in \{2.5,3\}$  by \Cref{rem:afterstep2}. If $f(t')=2.5$, we are done. If $f(t')=3$, then $t'\in F_{11}\cup V(G^{O_{\ref{step:F0X31X41}}})$. Suppose that  $t'\in F_{11}$.  Then there is no $t''\in N(t')\cap F_1$ such that $f(t'')>f(t')$. Hence, there is some $w'\in N(t')\cap F_0$ such that $w'\to t'$ by \ref{item:F12221} and so we are done. Suppose that  $t'\in V(G^{O_{\ref{step:F0X31X41}}})$. Since $t\in F_{11}$, we have  $X_2\to t'$ and so we are done. 
 
   Therefore, assume that $t\in V(G^{O_{\ref{step:F0X31X41}}})$. If $t\in R$, since $f(s)<f(t)$, we know that $t\notin R_{21}$. Consequently, $t\in R_{22}$ with $f(t)=4$ and some edge incident to $t$ being oriented in \ref{item:step:F0X31X414} or $t\in R_1\setminus R_{11}$ with $f(t)=5$.    It follows that $\mu(t)\le 2$ and so we are done. Suppose that $t\in L$. Then $t\in L_{21}\cup L_{22}$ with $\mu(t)\le 4$ and so we are done.

\medskip       
      
\noindent {\bf Case 2.}   Both $f(s)$ and $f(t)$ are integers.

Since $s\in F_{11}$, we know that edges in $[s,F_0]$ are oriented in \Cref{step:F12}.

Suppose first  that there is some $r\in N(s)\cap I_{f(s)}$ such that $sr$ is oriented as $s\to r$. Similar as before, we consider the position of $t$. Assume that $t\in V(G^{O_{\ref{step:F0X31X41}}}) $. 
If $t\in V(G^{O_{\ref{step:F0X31X41}}})\cap R$, then $f(t)\in \{4,5\}$ which implies that $\mu(t)\le 2$ and so we are done. 
Assume that $t\in V(G^{O_{\ref{step:F0X31X41}}})\cap L$, then $f(t)\in \{2,3\}$ and $\nu(s)\le 2$.  Suppose first $f(t)=3$.  Then $t\in L_{21}$ with $X_2\to t$.  Consequently, $\mu(t)\le 3$.  Thus, assume that $f(t)=2$, then $\mu(t)=4$ and $t$ plays the role of $u_2$ in \ref{item:step:F0X31X414} since edges in $[L_{21},L_1]$ has been oriented in \ref{item:step:F0X31X411}. 

Then assume that $t\in F_{11}$ and there is some $w\in N(t)\cap I_{f(t)}$ such that $tw$ is oriented as $t\to w$. 
 By the assumption that $f(t)>f(s)$ and $t\to w$ oriented in \Cref{step:F12},  we know that $t\to w$ is oriented in \ref{item:F12223}. Since $f(t)$ is an integer, by \Cref{rem:afterstep2}, we obtain that $f(t)=2$ and there are $t'\in N(t)\cap F_1\cap  L_2$ with $f(t')\in \{2.5,3\}$. By the same argument used in the end of last case, we are done. 

Therefore, suppose that $r\to s$ is oriented in \Cref{step:F12} with  $r\in N(s)\cap I_{f(s)}$.  Since $f(t)>f(s)$,  it holds  that  $r\to s$ is oriented in \ref{item:F12223} which implies that $f(s)=4$, $f(t)=5$, and  there is $s'\in N(s)\cap F_1$ with $f(s')\in \{3,3.5\}$ and $s\to s'$. By similar analysis as before, we have $\nu(s')\le 4$. 
   If $t\in V(G^{O_{\ref{step:F0X31X41}}})$, then $v\to t$ by \Cref{step:F0X31X41}. If $t\in F_{11}$, since $f(t)>f(s)$, again $v\to t$ is oriented in \ref{item:F12221}. 
   
 \medskip
Thus, the claim holds. 
\end{proofofclaim}

 Then assume that $s\in F_{11}$ and $f(t)=f(s)$. By \ref{item:F12224},  there is a $(F_0,F_0)$-dicycle of length 3 passing   $\orw{st}$, and so we are done.

Finally, suppose that $st$ is oriented in \ref{item:F122} with $s\in F_{11}$ and $t\in I_{f(s)}$. By \Cref{claim:fy>fx} and \ref{item:F12224}, we are done. 
 \end{proof}

 \begin{remark}\label{rem:I2I4}
  For each $x\in I_2$, we have $\mu(x)\le 3$ and $\nu(x)=1$;
for each $x\in I_4$, we have $\mu(x)=1$ and $\nu(x)\le 3$.
The bounds in \Cref{prop:xinf0munu,prop:F118} remain valid
when, in the corresponding estimates, we replace $\mu(x)$
by $3$ for $x\in I_2$ and $\nu(x)$ by $3$ for $x\in I_4$.
 \end{remark}

By the definitions of $F_{12},F_{13}$ and \Cref{step:F12}, no edge
incident to a vertex in
$F_{12}\cup\{x\in F_{13}:f(x)\in\mathbb Z\}$ has been oriented.
The edges from the remaining vertices of $F_{13}$ to $F_0$
have already been oriented in \ref{item:F121}. 
Now, we deal with edges incident to vertices in $F_{12}$. 

For each  $x\in F_{12}$, let  $\{y\}= N(x)\cap F_0$. By our assumption, there is a cycle,  say $C_{xy}=yxv_1\cdots v_ty$,  of length at most 5 containing $xy$.   If $C_{xy}$ is not unique, choose one that is contains $u$ or $v$ whenever possible, otherwise, choose one arbitrarily. For all such  $xy$, we fix $C_{xy}$. Then $C_{xy}$ must be one of the following forms, since each edge is contained in a cycle of length at most 5 by the assumption, as shown in \Cref{fig:cycles}. 
 
\begin{itemize}
\item[(a)] $C_{xy}=yxv_1v_2v_3y,$ where $v_1\in F_2,v_2\in F_2, v_3\in F_1$;
\item[(b)] $C_{xy}=yxv_1v_2y,$  where $v_1\in F_2,v_2\in F_1$;
\item[(c)] $C_{xy}=yxv_1v_2v_3y,$ where $v_1\in F_2,v_2\in F_1, v_3\in F_1;$
\item[(d)] $C_{xy}=yxv_1v_2v_3y,$ where $v_1\in F_2,v_2\in F_1, v_3\in F_0$.
\end{itemize} 
\begin{remark}
By forms (a)--(d) and the layer definitions, a considered cycle
containing $u$ is contained in $\{u\}\cup L_1\cup L_2$, while
one containing $v$ is contained in $\{v\}\cup R_1\cup R_2$.
Thus, if $u\in V(C_{xy})$ and $v\in V(C_{x'y'})$, then
$V(C_{xy})\cap V(C_{x'y'})=\emptyset$.
Moreover, if $\{u,v\}\cap V(C_{xy})\neq\emptyset$ and
$y\notin\{u,v\}$, then $C_{xy}$ is of form (d).
\end{remark}
\begin{figure}[htbp]
\begin{center}
\begin{tikzpicture}[
    rect/.style={rectangle, draw, minimum width=1.5cm, minimum height=0.8cm, align=center, thick, inner sep=2pt},
         rect1/.style={rectangle, draw, minimum width=9cm, minimum height=0.5cm, align=center, thick, inner sep=2pt},
    circ/.style={circle,
    draw,
    fill=black,
    minimum size=0.15cm,
    inner sep=0pt},
    arr/.style={-{Stealth[scale=1.2]}, shorten >=1pt, shorten <=1pt}
]

\node[rect1] (r1) at (0,0) {};
\node[rect1] (r2) at (0,2) {};
\node[rect1] (r3) at (0,4) {};

\node[circ] (v1) at (-4,4) {};
\node[circ] (v2) at (-3,4) {};
\node[circ] (v3) at (-4,2) {};
\node[circ] (v4) at (-3,2) {};
\node[circ] (v5) at (-3.5,0) {};

\node[circ] (v6) at (4.8-2,0) {};
\node[circ] (v7) at (4.8-1,0) {};
\node[circ] (v8) at (4.8-2,2) {};
\node[circ] (v9) at (4.8-1,2) {};
\node[circ] (v10) at (4.8-1.5,4) {};

\node[circ] (v11) at (1,4) {};
\node[circ] (v12) at (0,2) {};
\node[circ] (v13) at (1.8,2) {};
\node[circ] (v14) at (1,2) {};
\node[circ] (v15) at (1,0) {};

\node[circ] (v16) at (3.3-4.8,4) {};
\node[circ] (v17) at (2.8-4.8,2) {};
\node[circ] (v18) at (3.8-4.8,2) {};
\node[circ] (v19) at (3.3-4.8,0) {};

\draw (v1) -- (v2);
\draw (v2) -- (v4);
\draw (v4) -- (v5);
\draw (v5) -- (v3);
\draw (v3) -- (v1);

\draw (v6) -- (v7);
\draw (v7) -- (v9);
\draw (v9) -- (v10);
\draw (v10) -- (v8);
\draw (v8) -- (v6);

\draw (v11) -- (v12);
\draw (v12) -- (v14);
\draw (v14) -- (v15);
\draw (v15) -- (v13);
\draw (v13) -- (v11);

\draw (v16) -- (v17);
\draw (v17) -- (v19);
\draw (v19) -- (v18);
\draw (v18) -- (v16);

\node[left]  at (v1)  {$v_2$};
\node[right]  at (v2) {$v_1$};
\node[left]  at (v3)  {$v_3$};
\node[right]  at (v4)  {$x$};
\node[left]   at (v5) {$y$};
\node[left]  at (v6)   {$v_3$};
\node[right]  at (v7)   {$y$};
\node[left]  at (v8)   {$v_2$};
\node[right] at (v9)  {$x$};
\node[left]  at (v10)  {$v_1$};
\node[left] at (v11)  {$v_1$};
\node[left] at (v12)  {$v_2$};
\node[right] at (v13)  {$x$};
\node[right] at (v14)   {$v_3$};
\node[left] at (v15)   {$y$};
\node[left]  at (v16)  {$v_1$};
\node[left] at (v17)  {$v_2$};
\node[right]   at (v18) {$x$};
\node[left]  at (v19) {$y$};

\node[xshift=-1.5cm]  at (v1) {$F_2$};
\node[xshift=-1.5cm]  at (v3)  {$F_1$};
\node[xshift=-2cm]   at (v5) {$F_0$};

\node[yshift=-0.6cm]   at (v5) {(a)};
\node[xshift=-0.5cm,yshift=-0.6cm]  at (v7)   {(d)};
\node[yshift=-0.6cm] at (v15)   {(c)};
\node[yshift=-0.6cm]  at (v19) {(b)};

\end{tikzpicture}
\caption{Cycles of distinct forms.}\label{fig:cycles}
\end{center}
\end{figure}
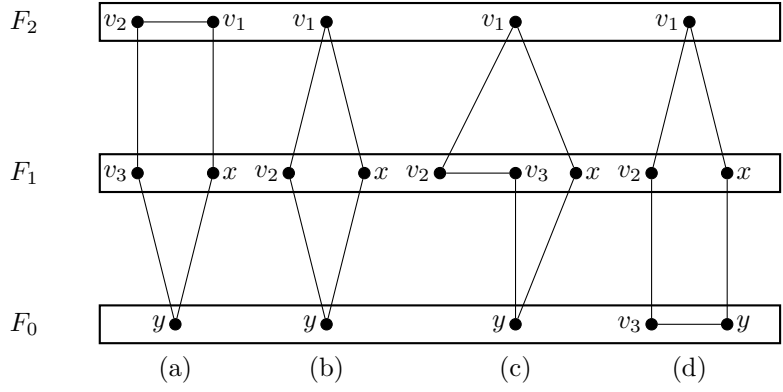

Recall for each $x\in F_{12}$,  the edges $xy$ and $xv_1$ in $C_{xy}$ are not oriented yet.   
Let $\mathcal{C}=\{C_{xy}:x\in F_{12}\text{ and } \{y\}=N(x)\cap F_0\}$.  Note that  $\mathcal{C}$ may be a multi-set since it is possible that  $C_{xy}=C_{x'y'}$ for some $x\neq x'\in F_{12}$.  
 Order the cycles in the multiset $\mathcal C$ as
$C_1,C_2,\ldots,C_d$ by listing the following three classes
consecutively:
\begin{itemize}
    \item first, the cycles associated with $y\in\{u,v\}$;
    \item then, the cycles intersecting $\{u,v\}$ and associated
    with $y\notin\{u,v\}$;
    \item finally, all the remaining cycles.
\end{itemize}
Let $d^*$ be the number of cycles in the first two classes.
In particular, $d^*=0$ if neither class contains a cycle.
Thus, $C_i$ intersects $\{u,v\}$ if and only if $i\le d^*$.

In the next step, we orient $C_1$, $C_2$, $C_3,\dots$ in order. Let $H_0=G^{O_{\ref{step:F12}}}$. For each $i\in[d]$, assume inductively that $H_{i-1}$ is the partial
orientation obtained after processing $C_1,\ldots,C_{i-1}$. We process
$C_i$ with respect to $H_{i-1}$. After all prescribed edges associated
with $C_i$ have been oriented, we update the current partial
orientation by adding these newly oriented edges, while leaving all
previously assigned orientations unchanged. Denote the resulting
partial orientation by $H_i$. The next cycle, $C_{i+1}$, is then
processed with respect to $H_i$.    Then $H_{d^*}=G^{O_{\ref{step:cycle(auv)}}}$. 

Suppose that the current partial orientation is $H_{i-1}$ and
$C_i=C_{xy}$. If $xy$ has already been oriented, set
$H_i=H_{i-1}$ and proceed to the next index. Otherwise, define
$L_{xy}$ in $H_{i-1}$ as follows.

If no edge of $C_{xy}$ has been oriented in $H_{i-1}$,
set
\[
L_{xy}=C_{xy}.
\]
Otherwise, if $C_{xy}$ is of form (c) and no edge of
$yxv_1v_2$ has been oriented, let
$$L_{xy}=yxv_1v_2.$$
Moreover,  choose an auxiliary vertex $v_3^*$ satisfying:
\[
v_3^*\in
\begin{cases}
N(v_2)\cap I_{f(v_2)}, & \text{if } f(v_2)  \text{ is an integer},\\
N(v_2)\cap I_{f(v_2)-\frac12}, & \text{if } f(x)\ge f(v_2),\\
N(v_2)\cap I_{f(v_2)+\frac12}, & \text{if } f(x)\le f(v_2).
\end{cases}
\]
In all remaining cases, let
\[
L_{xy}=yxv_1\cdots v_p
\]
be the maximal initial subpath of $C_{xy}$ beginning with $yx$ such
that every edge of $L_{xy}$ is unoriented in $H_{i-1}$. In this
case, set
\[
v_{p+1}^*=v_{p+1},
\]
where $v_{p+1}$ is the vertex immediately following $v_p$ in
$C_{xy}$. 
We call $C_{xy}$ the \emph{considered} cycle of $L_{xy}$.

In the cycle steps, put $(\alpha_1,\ldots,\alpha_5)=(2,3,2,1,0)$ and $(\beta_1,\ldots,\beta_5)=(0,1,2,3,2)$. Thus, for $a\in I_k$, it holds that $\mu_H(a)\le\alpha_k$ and $\nu_H(a)\le\beta_k$ in every extension $H$ of  $G^{O_{\ref{step:F12}}}$. 

For every $c\in F_{13}$ with $f(c)=k$ an integer, we keep one choice of direction for $[c,I_k]$ throughout the cycle steps. Since $|N(c)\cap F_0|\ge2$ and $N(c)\cap F_0\subseteq I_k$, while $I_1=\{u\}$ and $I_5=\{v\}$, it holds that $k\in\{2,3,4\}$.
The choice is made when the first edge of $[c,I_k]$ is oriented. Every later edge of this set is oriented in the same direction.
This convention only concerns edges prescribed by the current operation. 
Suppose that the first such operation belongs to a considered cycle containing $v$. By the classification of the considered cycles, this cycle has form $vxaczv$ with $z\in I_4$. Since $cz\in E$ and $f(c)=k$ is an integer, it follows that $k=4$.
The edge $vz$ has already been oriented as $v\to z$. We therefore choose $I_k\to c$, so that the cycle is oriented as $v\to z\to c\to a\to x\to v$. Similarly, if the first such operation belongs to a considered cycle containing $u$, then $k=2$, and we choose $c\to I_k$.

Let $H$ denote the partial orientation immediately before the current operation. For an unoriented path $P=yxv_1\cdots t$ of length $m$, where $y\in I_k$ and $t\notin F_0$, define its two \emph{scores} by \[\sigma_H(P^+)=\alpha_k+m+\nu_H(t)  \text{ and }  \sigma_H(P^-)=\mu_H(t)+m+\beta_k.\] Here $P^+$ starts at $y$ and $P^-$ ends at $y$. All the values in these scores are computed before the operation. The score of the selected direction is recorded when the operation is performed. Adding the selected arcs cannot increase the actual distance parameters, and hence $\omega(P^+)\le\sigma_H(P^+)$ and $\omega(P^-)\le\sigma_H(P^-)$ in the resulting partial orientation.

We also record the consequence for the internal vertices of a selected
dipath. Suppose that $Q=q_0\to \cdots \to q_r$, and let $W$ be
its recorded score in the pre-operation orientation $H$. 
For every subsequent orientation $K$ and every $0\le j\le r$, we know 
\[
\mu_K(q_j)\le\mu_H(q_0)+j,\qquad
\nu_K(q_j)\le r-j+\nu_H(q_r).
\]
Therefore, 
\[
\mu_K(q_j)+\nu_K(q_j)
\le\mu_H(q_0)+r+\nu_H(q_r)\le W.
\]
The last inequality uses the fixed core bound in the recorded score.
For a completed dicycle through a core vertex, use the two portions
of the cycle instead; their lengths sum to its length and give the
same vertex-sum estimate. Finally, an existing arc $a\to b$ gives
$\mu_K(a)+\nu_K(a)\le\omega_K(a\to b)$ and
$\mu_K(b)+\nu_K(b)\le\omega_K(a\to b)$.
Thus the vertex bounds used in the cycle induction follow directly
from the earlier operations and do not require a later corollary.

If $\mu_H(t)+\nu_H(t)\le9$, the two scores have sum at most $13+2m$. Thus, for $m\le2$ their minimum is at most $8$. For $m=3$, their minimum is at most $9$, and is at most $8$ if $\mu_H(t)+\nu_H(t)\le7$. We shall prove separately that the case with score $9$ has the required dicycle.

\begin{step}\label{step:cycle(auv)}
Process $C_1,\ldots,C_{d^*}$ in order. Suppose that the current
cycle is $C_i=C_{xy}$, where $i\le d^*$. If $xy$ is already
oriented, proceed to the next cycle. Otherwise, determine $L_{xy}$
and the auxiliary vertex in the current partial orientation $H$.
All scores are computed in $H$.

For each $c\in V(C_{xy})\cap F_{13}$ with
$f(c)=k\in\mathbb{Z}$, respect any previously fixed direction
for $[c,I_k]$. When this operation first fixes that direction,
choose $I_4\to c$ if $v\in V(C_{xy})$ and $k=4$, and choose
$c\to I_2$ if $u\in V(C_{xy})$ and $k=2$.
Only the edges prescribed by the current operation are oriented.
Apply the following rules in order.
Suppose that $v\in V(C_{xy})$. Then either $y=v$ or
$y\in R_{11}$. We distinguish these two cases.

\begin{enumerate}[label=(S\thestep.\arabic*)]
\item\label{item:cycle-yv}
Suppose that $y=v$.

\begin{enumerate}[label=(S\thestep.\arabic{enumi}.\arabic*)]
\item\label{item:abc1}
If $L_{xy}=C_{xy}$, orient $C_{xy}$ as a dicycle,
respecting all previously fixed directions.
If both directions are permitted, choose either one.

\item\label{item:cycle-four}
Suppose that $L_{xy}=vxv_1v_2v_3$ has length $4$.

If $C_{xy}$ is of form (d), then $v_3\in R_{11}$ and
$v\to v_3$ is already oriented. Orient $L_{xy}$ as  
$v_3\to v_2\to v_1\to x\to v$.  
When this operation first fixes the direction of
$[v_2,I_4]$ for $v_2\in F_{13}$, record $I_4\to v_2$.

If $C_{xy}$ is of form (a), choose the direction of
$L_{xy}$ with smaller score. If the scores are equal,
choose the direction that completes $C_{xy}$ as a dicycle.

\item\label{item:cycle-short}
Suppose that $L_{xy}=vxv_1\cdots v_p$ has length $m\le3$.
Its two scores are
\[
\sigma_H(L_{xy}^+)=m+\nu_H(v_p),\qquad
\sigma_H(L_{xy}^-)=\mu_H(v_p)+m+2.
\]
Choose the direction with smaller score.
If the scores are equal, first prefer a direction for
which the whole selected path is contained in a dicycle
of length at most $5$ through $v$, whenever such a direction
exists. If this does not distinguish the two choices,
prefer the direction in which
$L_{xy}\cup\{v_pv_{p+1}^*\}$ is a dipath or dicycle.
\end{enumerate}

\item\label{item:cycle-yR11}
Suppose that $y\in R_{11}$.
Then $C_{xy}=yxv_1v_2vy$ is of form (d).
\begin{enumerate}[label=(S\thestep.\arabic{enumi}.\arabic*)]
\item\label{item:cycle-r11-four}
If $L_{xy}=yxv_1v_2v$ has length $4$, orient it as $y\to x\to v_1\to v_2\to v. $ Thus $C_{xy}$ becomes a dicycle.

\item\label{item:cycle-r11-short}
Suppose that $L_{xy}=yxv_1\cdots v_p$ has length $m\le3$.
Its two scores are
\[
\sigma_H(L_{xy}^+)=1+m+\nu_H(v_p),\qquad\text{ and }
\sigma_H(L_{xy}^-)=\mu_H(v_p)+m+3.
\]
Choose the direction with smaller score.
If the scores are equal, prefer the direction in which
$L_{xy}\cup\{v_pv_{p+1}^*\}$ is a dipath or dicycle.
\end{enumerate}
\end{enumerate}

The cycles containing $u$ are treated symmetrically, with
the two cases $y=u$ and $y\in L_{11}$: interchange $u,v$
and $\mu,\nu$, replace $R_{11}=I_4$ by $L_{11}=I_2$,
and reverse all prescribed directions. 
\end{step}

The cycles disjoint from $\{u,v\}$ will be oriented in
\Cref{step:cycle(a)} using similar orientation rules; in that case,
a weight of $9$ suffices. For completeness, we describe these
orientations explicitly in that step. The following remark applies
both to the cycles oriented here and to those to be oriented in
\Cref{step:cycle(a)}.

\begin{remark}
    \label{lem:cycle-star-choice}
All the following hold. 
\begin{itemize}
    \item 
For each $c\in F_{13}$ with $f(c)\in\mathbb{Z}$, the fixed direction is consistent with all the cycle operations. If the first operation involving $c$ contains $v$, then $f(c)=4$ and $I_4\to c$. The case involving $u$ is symmetric.

\item Suppose that $C_{xy}=yxv_1tc y$ is of form (a), $L_{xy}=yxv_1t$, and the prescribed or existing directions of $tc$ and $cy$ form a dipath. Then every new edge has weight at most $9$. If a new path or edge has weight $9$, it is contained in the current dicycle $C_{xy}$. For $y\in \{u,v\}$, the corresponding bound is $7$, with a dicycle through $y$ in the equality case. 
\end{itemize}
\end{remark}
\begin{proof} 
Since $N(c)\cap F_1=\emptyset$ for $c\in F_{13}$,
each considered cycle contains at most one vertex of $F_{13}$,
and such a vertex occurs only in forms (a), (b) or (d).
The first operation incident to $c$ orients an edge in
$[c,F_0]$: it either orients a whole cycle or a four-edge
path of form (d).
Indeed, a shorter initial path can end at $c$ only when
the following edge from $c$ to $F_0$ is already oriented.
Thus, whenever an edge incident to $c$ is already oriented,
some edge in $[c,F_0]$ is already oriented as well.

A whole cycle permits either direction of its edge at
$[c,I_{f(c)}]$, and a four-edge path of form (d) is oriented
in the fixed direction of that edge.
All other selected paths contain no new edge in this set;
the additional edge in \ref{item:cycle5-short} explicitly
uses its fixed direction.
For cycles containing $v$, the only possible form is (d),
with $f(c)=4$, and the prescribed direction is $I_4\to c$.
The case of $u$ is symmetric.
This proves the first assertion.

For the second assertion, put $y\in I_k$ and $P=yxv_1t$,
and let $H$ denote the partial orientation before the operation.
By symmetry, suppose that $t\to c\to y$ is the prescribed
or existing dipath.
If $cy$ is unoriented, the operation belongs to
\Cref{step:cycle(a)}.
By \Cref{step:F12,step:F12224}, we have
$c\in F_{13}$ and $f(c)=k\in\{2,3,4\}$.
The first assertion then gives an existing arc
$c\to y'$ with $y'\in I_k$.
Thus, in either case, $\nu_H(t)\le2+\beta_k$, and
$\sigma_H(P^+)\le\alpha_k+\beta_k+5$.
The rules for equal scores prefer $P^+$, which completes
$C_{xy}$ as a dicycle.
If $P^-$ is selected, its score is therefore at most
$\alpha_k+\beta_k+4$.
Moreover, a newly oriented edge $c\to y$ has weight at most
$\mu_H(t)+2+\beta_k=\sigma_H(P^-)-1$.
If $P^+$ is selected, the completed dicycle bounds every
new edge and $P$ by $\alpha_k+\beta_k+5$.
Since $\alpha_k+\beta_k=4$ for $k\in\{2,3,4\}$ and
$\alpha_k+\beta_k=2$ for $k\in\{1,5\}$, both bounds
and their equality assertions follow. 
\end{proof}

\begin{proposition}\label{prop:cycleauv} Let $P=L_{xy}$ be a path oriented in \Cref{step:cycle(auv)} whose considered cycle contains $u$ or $v$. Then its selected score, when a score is used, and its actual weight are at most $7$.\\ If $y\in\{u,v\}$ and $\omega(\overrightarrow P)=7$, then $\overrightarrow P$ is contained in a dicycle of length $5$ through $y$.\\ Every other edge newly oriented in this step also has weight at most $7$. 
\end{proposition}
\begin{proof}
We prove by induction on the linear order of the considered cycles.
Without loss of generality, assume that $v\in V(C_{xy})$.
The proof for $u$ is symmetric.
Let $H$ be the partial orientation immediately before the current
operation. Recall that all cycles with associated vertex $y=v$
precede those with $y\in R_{11}$.

\medskip
\noindent{\bf Case 1.} $y=v$.

If $L_{xy}=C_{xy}$, then $C_{xy}$ is oriented as a dicycle
through $v$, and every new edge has weight at most
$\mu_H(v)+e(C_{xy})+\nu_H(v)\le7$.

Suppose that $P=vxv_1v_2v_3$ has $4$ edges.
If $C_{xy}$ is of form (d), then $v_3\in R_{11}$ and
$v\to v_3$ is already oriented.
By the orientation rule, $P$ is oriented as
$v_3\to v_2\to v_1\to x\to v$,
so $C_{xy}$ becomes a dicycle of length $5$ through $v$.
If $C_{xy}$ is of form (a), the direction completing
$C_{xy}$ as a dicycle has score at most $7$.
The other direction is selected only when its score is
strictly smaller, and hence at most $6$.
Thus, the required bound and equality assertion hold
in both cases.

Suppose that $P=vxv_1$ has $2$ edges.
Then $v_1\in F_2\cap R_2$.
If $v_1$ has an incident edge oriented in
\Cref{step:F0X31X41}, then $\mu_H(v_1)\le2$, and so 
$\sigma_H(P^-)=\mu_H(v_1)+2+2\le6$. 
Otherwise, $v_1$ was first reached by an earlier cycle
operation.
By induction and the vertex-sum estimate preceding
\Cref{step:cycle(auv)}, we have
$\mu_H(v_1)+\nu_H(v_1)\le7$.
Consequently,
$\sigma_H(P^+)+\sigma_H(P^-)
=\mu_H(v_1)+\nu_H(v_1)+4+2\le13$.
Thus, the selected score is at most $6$.

Therefore, it suffices to consider $P=vxv_1v_2$
with $3$ edges.

Suppose that $v_2\in F_1$.
If $v_2v\in E$, then this edge is already oriented,
and a direction completing $vxv_1v_2v$ as a dicycle
has score at most $6$.
Thus, assume that $v_2v\notin E$, and write
$C_{xy}=vxv_1v_2v_3v$.
If $v_3\in F_1$, then $f(v_3)>f(v_2)$, and the initial
rules and \Cref{step:F12} give $v\to v_3\to v_2$.
 
The same dipath exists if $v_3\in R_{11}$ and $v_3\to v_2$.
It remains to consider $v_3\in R_{11}$ and $v_2\to v_3$.
This direction is not prescribed in
\Cref{step:F0X31X41} or \ref{item:F121}.
If the edges joining $v_2$ to $F_0$ were unoriented
before the cycle steps, every such edge oriented in an
earlier cycle through $v$ would be directed towards $v_2$.
Thus, $v_2\in F_{11}$, $f(v_2)=4$, and $v_2\to v_3$
is prescribed by \ref{item:F12222}.
Hence, there is some $w\in N(v_2)\cap F_1$ with $f(w)>4$.
By \Cref{step:F0X31X41,step:F12}, we have $v\to w\to v_2$.
Since $x\in F_{12}$ and $v_1\in F_2$, it holds that
$w\notin\{x,v_1\}$.
Thus, the reverse direction of $P$ completes the dicycle
$v\to w\to v_2\to v_1\to x\to v$.

In all these cases, $\mu_H(v_2)\le2$, and hence
$\sigma_H(P^-)\le2+3+2=7$.
If the selected weight is seven, the selected score
is also seven.
By \ref{item:cycle-short}, the selected path is contained
in a dicycle of length at most five through $v$.
Its length must be five, since a shorter dicycle
would give weight at most six. 

Now, suppose that $v_2\in F_2$.
Then $v_3\in F_1$, and $v_3v$ is already oriented.
If $v_2v_3$ and $v_3v$ form a dipath, then
by \Cref{lem:cycle-star-choice}, we are done.
Otherwise, consider the first operation orienting $v_2v_3$.
\Cref{step:F0X31X41}, whole cycles in
\Cref{step:cycle(auv)}, cycles of form (d) in
\Cref{step:cycle(auv)}, and operations with designated
vertex $v_3$ would make these two edges form a dipath.
Moreover, \Cref{step:F12} orients no edge between
$F_1$ and $F_2$.
By the ordering of the cycles, the remaining possibilities
are previous paths $Q=vx'v'_1v_2v_3$ or $Q=vx'v_2v_3$
oriented in \Cref{step:cycle(auv)}.

In the first case, the selected direction of $Q$ does not
complete its considered cycle, so its recorded score
is at most $6$.
Thus, by copying the corresponding direction and using
monotonicity, the score of $P$ is also at most $6$.
 
In the second case, let $H'$ be the partial orientation
before $Q=vx'v_2v_3$ is oriented.
A direction completing $vx'v_2v_3v$ as a dicycle
has score at most six.
Since $v_3\in R_1\setminus R_{11}$, we have
$\mu_{H'}(v_3)\ge1$ and $\nu_{H'}(v_3)\ge3$,
so both scores are at least six.
Thus, the direction not completing this dicycle
can be selected only when both scores equal six.

If $f(v_3)=5$, then the auxiliary edge is $v_3v$.
By \ref{item:cycle-short}, the direction completing
the dicycle is selected.

Suppose that $f(v_3)=4.5$.
Then $v\to v_3$ by \ref{item:F121}, and the reverse
direction completes $v\to v_3\to v_2\to x'\to v$.
The forward direction could belong to a dicycle of
length at most five through $v$ only if there were
a dipath $v_3\to w\to v$.
Since $N(u)\cap N(v)=\emptyset$, we have $w\ne u$.
Also, $w\notin R_{11}$, since $v\to R_{11}$.
Thus, $w\in R_1\setminus R_{11}$.
If $f(w)=9/2$, then $v\to w$ by \ref{item:F121}.
If $f(w)=5$, then $w\to v_3$ by
\Cref{step:F0X31X41} and \ref{item:F123}.
Both cases are impossible.
Therefore, \ref{item:cycle-short} selects the reverse
direction, again contradicting the assumption that
$v_2v_3$ and $v_3v$ do not form a dipath. 

\medskip
\noindent{\bf Case 2.} $y\in R_{11}$.

Then $C_{xy}=yxv_1v_2vy$ is of form (d), with
$v\to y$ already oriented.
If $P=yxv_1v_2v$, its prescribed direction completes
$C_{xy}$ as a dicycle through $v$, giving weight at most $7$.

If $P=yxv_1$, then, as above, either $\mu_H(v_1)\le2$
or $\mu_H(v_1)+\nu_H(v_1)\le7$.
Accordingly, $\sigma_H(P^-)\le7$ or
$\sigma_H(P^+)+\sigma_H(P^-)\le15$.
Thus, the selected score is at most $7$.

If $P=yxv_1v_2$, then $v_2v$ is already oriented.
When $v\to v_2$, we have $\mu_H(v_2)\le1$ and
$\sigma_H(P^-)\le1+3+3=7$.
When $v_2\to v$, we have $\nu_H(v_2)\le3$ and
$\sigma_H(P^+)\le1+3+3=7$.

Every new edge belongs to the selected path or cycle,
so the same bounds apply to all new edges.
This completes the induction.
\end{proof}

 Because of cycles of type (c), to orient cycles that disjoint with $\{u,v\}$, we thus next deal with the  unoirented edges in $[F_{11},F_0]$. 
\begin{step}\label{step:F12224}
  
First, for every $x\in F_{11}$ such that some, but not all,
edges in $[x,F_0]$ are oriented, orient the remaining edges
in the same direction as the already oriented edges.

Then perform the following operations one at a time,
checking the conditions in the current partial orientation. 

For every $x\in F_{11}$ with $f(x)\in\{2,3,4\}$
satisfying that no edge in $[x,F_0]$ has been oriented
and $f(x)=f(x')$ for all $x'\in N(x)\cap F_1$,
fix some $z\in N(x)\cap F_1$, and orient the unoriented
edges in the paths and cycles of form
$I_{f(x)}xzI_{f(x)}$ so that every resulting path is a
dipath and every resulting cycle is a dicycle,
all with the same direction. 
\end{step}

\begin{proposition}\label{prop:F12224}
For $e\in E_{\ref{step:F12224}}$, it holds that
$\omega(e)\le7$ in $G^{O_{\ref{step:F12224}}}$.
\end{proposition}
\begin{proof} 
Suppose first that $x\in F_{11}$ and some, but not all,
edges in $[x,F_0]$ are oriented before this step.
Since \Cref{step:F12} either orients all these edges
or leaves all of them unoriented, the oriented edges
in $[x,F_0]$ are oriented in \Cref{step:cycle(auv)}.

By the forms of the considered cycles, this can only
occur in a cycle of form (d).
If the cycle contains $v$, then $f(x)=4$ and its edge
in $[x,I_4]$ is oriented towards $x$.
Every later operation in \Cref{step:cycle(auv)}
preserves this direction.
The resulting dicycle gives $\nu(x)\le5$.
Thus, the remaining edges can all be oriented from
$I_4$ to $x$, each with weight at most $1+1+5=7$.
The case of a cycle containing $u$ is symmetric.

Now, suppose that $x$ is considered in the second part
of the step, and put $k=f(x)$.
We first show that $xz$ is unoriented.
Since $x\in F_{11}$ and $f(x)=f(z)$,
the edge $xz$ is not oriented in
\Cref{step:F0X31X41,step:F12}.
Among the considered cycles containing $u$ or $v$,
only form (c) has an edge between two vertices of $F_1$.
One endpoint of that edge is adjacent to $u$ or $v$,
and hence its $f$-value is not in $\{2,3,4\}$.
Thus, $xz$ is not oriented in \Cref{step:cycle(auv)}.
An earlier operation in the present step involving $xz$
would already have oriented $[x,I_k]$.
Therefore, $xz$ is unoriented.

By \Cref{step:F0X31X41,step:F12,step:cycle(auv)}
and the preceding operations, all oriented edges in
$[z,I_k]$ have the same direction.
Since $[x,I_k]$ and $xz$ are unoriented, we can orient
$I_k\to x\to z\to I_k$, or its reverse, consistently
with these edges.
If $[z,I_k]$ is also unoriented, choose either direction.
Thus, the operation is well-defined.

Every new edge is contained in a dipath of length three
with both ends in $I_k$, or in a dicycle of length three
through $I_k$.
Since $\mu(y)\le5-k$ and $\nu(y)\le k-1$ for $y\in I_k$,
its weight is at most $(5-k)+3+(k-1)=7$. 
\end{proof} 
After \Cref{step:F12224}, all edges in $[F_{11},F_0]$
are oriented.
Indeed, the partially oriented sets are completed in
the first part of this step, and the remaining sets
are oriented in \Cref{step:F12} or the second part
of \Cref{step:F12224}.
Moreover, by
\Cref{prop:F118,prop:cycleauv,prop:F12224},
it holds that $\mu(x)+\nu(x)\le7$ for every
$x\in F_{11}$. 

Now, we orient the cycles $C_{d^*+1},\ldots, C_d$. Set $H_{d^*}=G^{O_{\ref{step:F12224}}}$. 

\begin{step}\label{step:cycle(a)}
Suppose that $d^*+1\le i\le d$ and $C_i=C_{xy}$.
Then $V(C_{xy})\cap\{u,v\}=\emptyset$.
If $xy$ is already oriented, proceed to the next cycle.
Otherwise, determine $L_{xy}$ and the auxiliary vertex in the
current partial orientation $H$, as defined for \Cref{step:cycle(auv)}. 
All scores are computed in $H$.
For every $c\in F_{13}$ with $f(c)=k\in\mathbb{Z}$,
if a direction has already been fixed for $[c,I_k]$,
use that direction for every edge of $[c,I_k]$
oriented in this step. 
Apply the following rules in order.
\begin{enumerate}[label=(S\thestep.\arabic*)]
\item\label{item:cycle5-whole}
Suppose that $L_{xy}=C_{xy}$.
Orient $C_{xy}$ as a dicycle, respecting any previously fixed
direction for $[c,I_k]$ whenever
$c\in V(C_{xy})\cap F_{13}$ and $f(c)=k\in\mathbb{Z}$.
If $C_{xy}$ is of form (a), $v_3\in F_{13}$,
$f(v_3)\in\mathbb{Z}$, and no edge incident with $v_3$ has
been oriented, the direction of $[v_3,I_{f(v_3)}]$ is being
fixed for the first time.
When this choice is free, choose among the permitted
orientations one that minimizes $\nu(v_3)$ if
$V(C_{xy})\subseteq F_0\cup R$, and one that minimizes
$\mu(v_3)$ if $V(C_{xy})\subseteq F_0\cup L$.
Here the parameters are evaluated after orienting the cycle
in the trial direction, with all other edges unchanged.
In all other cases, choose any permitted orientation.
Record each direction for $[c,I_k]$ first fixed by this operation.

\item\label{item:cycle5-four}
Suppose that $L_{xy}=yxv_1v_2v_3$ has length $4$.
If $C_{xy}$ is of form (d), then $v_3\in F_0$.
If $v_2\in F_{13}$ with $f(v_2)=k\in\mathbb{Z}$ and the
direction of $[v_2,I_k]$ has been fixed, orient $L_{xy}$
in the direction required by $v_2v_3$.
Otherwise, orient $L_{xy}$ so that $C_{xy}$ becomes a dicycle,
and record the resulting direction of $[v_2,I_k]$ when
$v_2\in F_{13}$ and $f(v_2)=k\in\mathbb{Z}$.

If $C_{xy}$ is of form (a), choose the direction of $L_{xy}$
with smaller score. If the scores are equal, choose the
direction that completes $C_{xy}$ as a dicycle.

\item\label{item:cycle5-short}
Suppose that $L_{xy}=yxv_1\cdots v_p$ has length at most $3$.
Choose its direction with smaller score.
If the scores are equal, prefer the direction in which
$L_{xy}\cup\{v_pv_{p+1}^*\}$ is a dipath or dicycle.

Furthermore, if $C_{xy}$ is of form (a), $L_{xy}=yxv_1v_2$,
and $v_3\in F_{13}$ with $f(v_3)\in\mathbb{Z}$, orient
$v_3y$ in its fixed direction when this edge is still
unoriented. This additional edge does not enter either score.
Its weight is estimated separately below.
\end{enumerate}
\end{step}

\begin{proposition}\label{prop:F12a}
Let $P=L_{xy}$ be a path oriented in \Cref{step:cycle(a)}.
Then $\omega(\orw{P})\le9$, with equality only if
$\orw{P}$ lies on a dicycle of length $5$ meeting $F_0$.\\
Moreover, every edge newly oriented in this step, including
an additional edge prescribed by \ref{item:cycle5-short},
has weight at most $9$, with equality only if it lies on
such a dicycle.
\end{proposition}

\begin{proof}
We prove by induction on the linear order of the cycles.
Let $H$ be the partial orientation before the current
operation, and let $y\in I_k$.
Since $C_{xy}$ avoids $u$ and $v$, we have
$k\in\{2,3,4\}$.
Recall that $\alpha_k=5-k$, $\beta_k=k-1$, and all scores
are computed before the corresponding operation.

By \Cref{prop:xinf0munu,prop:F118,prop:F12224,prop:cycleauv}
and induction, every vertex $t$ incident to an oriented
edge in $H$ satisfies $\mu_H(t)+\nu_H(t)\le9$.
For a path $P$ from $y$ to $t$, denote its two directions
by $P^+$ and $P^-$, where $P^+$ starts at $y$.
Then
$\sigma_H(P^+)=\alpha_k+e(P)+\nu_H(t)$ and
$\sigma_H(P^-)=\mu_H(t)+e(P)+\beta_k$.
In particular, if $e(P)=3$ and
$\mu_H(t)+\nu_H(t)\le7$, their sum is at most $17$,
and the selected score is at most $8$.

We first note that if $yxaz$ is an unoriented path with
$x\in F_{12}$, $a\in F_2$, and
$f(z)=\ell\in\{2,3,4\}$, then $|k-\ell|\le1$.
Otherwise, by symmetry, $k=2$ and $\ell=4$.
Thus, $x\in L_2$, $z\in R_2$, and
$a\in L_2\cup R_2\cup X_3$.
In the first two cases, $az$ or $xa$ is already oriented
in \Cref{step:F0X31X41}.
In the last case, $a\in X_{31}$ and $x\in L_{21}$,
contrary to $x\in F_{12}$.

Suppose that $L_{xy}=C_{xy}$.
Then $C_{xy}$ becomes a dicycle, and every new edge has
weight at most $\alpha_k+e(C_{xy})+\beta_k\le9$,
with equality only if $e(C_{xy})=5$.

Suppose that $P=yxv_1v_2v_3$ has four edges.
If $C_{xy}$ is of form (a), the direction completing
$C_{xy}$ as a dicycle has score at most nine.
The other direction is selected only if its score is
smaller, and hence at most eight. 
Suppose that $C_{xy}$ is of form (d).
If the selected direction completes $C_{xy}$ as a
dicycle, the result follows as above.
Otherwise, $P$ and the already oriented edge $yv_3$
have the same direction.
Let $v_3\in I_\ell$.
If $y\to v_3$, then $k\ge\ell$ by
\Cref{step:F0X31X41}, and
$\omega(P^+)\le\alpha_k+4+\beta_\ell
=8+\ell-k\le8$.
If $v_3\to y$, then $\ell\ge k$, and
$\omega(P^-)\le\alpha_\ell+4+\beta_k
=8+k-\ell\le8$.
Thus, both assertions hold for paths of length four.

Suppose that $e(P)\le2$.
The two scores have sum at most $4+4+9=17$,
so the selected score is at most $8$.

Therefore, assume that $P=yxv_1v_2$ and
$C_{xy}=yxv_1v_2v_3y$.

\medskip
\noindent{\bf Case 1.} $v_1\in F_2$ and $v_2\in F_1$.

If $v_2y$ is already oriented, one direction of $P$
completes a dicycle of length $4$ and has score at most $8$.
Thus, assume that this does not occur.

Suppose that
$v_2\in V(G^{O_{\ref{step:F0X31X41}}})$.
Since $v_2\in F_1$, we have $v_2\in L\cup R\cup X_3$.
By \Cref{step:F0X31X41}, if $v_2\in R$, then
$\mu_H(v_2)\le2$, and hence
$\sigma_H(P^-)\le2+3+\beta_k=k+4\le8$.
If $v_2\in L$, then $\nu_H(v_2)\le2$, and hence
$\sigma_H(P^+)\le\alpha_k+3+2=10-k\le8$.
If $v_2\in X_3$, then
$\mu_H(v_2)+\nu_H(v_2)\le7$,
so one score is at most eight.
Thus, assume that
$v_2\notin V(G^{O_{\ref{step:F0X31X41}}})$.

If $v_2\in F_{11}$, then
$\mu_H(v_2)+\nu_H(v_2)\le7$ by
\Cref{prop:F118,prop:cycleauv,prop:F12224}.
Thus, one score is at most $8$.
The same holds if $v_2\in F_{13}$ and $f(v_2)$ is not
an integer, since \ref{item:F121} gives
$\mu_H(v_2)+\nu_H(v_2)\le5$.

Thus, $v_2\in F_{12}\cup F_{13}$ and
$f(v_2)=\ell$ is an integer.
Since $N(v_2)\cap F_1=\emptyset$, the considered cycle
is of form (d).
Then $v_3\in I_\ell$, $v_2v_3$ is already oriented,
and $|k-\ell|\le1$.
If $v_2\to v_3$, then
$\sigma_H(P^+)\le\alpha_k+4+\beta_\ell=8+\ell-k$.
If $v_3\to v_2$, then
$\sigma_H(P^-)\le\alpha_\ell+4+\beta_k=8+k-\ell$.
A bound of $9$ is possible only for the direction
completing $C_{xy}$.
The rule for equal scores prefers this direction.
Hence, the selected score is at most $9$, and equality
gives the required dicycle.

\medskip
\noindent{\bf Case 2.} $v_1,v_2\in F_2$.

Then $v_3\in F_1$ and $v_2v_3$ is already oriented.
If $v_3y$ is still unoriented, the preceding steps and
\Cref{lem:cycle-star-choice} imply that
$v_3\in F_{13}$, $f(v_3)=k$, and an edge in
$[v_3,I_k]$ is already oriented.
The present operation prescribes the same direction
on $v_3y$.

Suppose that $v_2v_3$ and the existing or prescribed
direction on $v_3y$ form a dipath.
By \Cref{lem:cycle-star-choice}, both the assertion
for $P$ and that for the possible additional edge hold.

Thus, assume that these two edges do not form a dipath.
If $f(v_3)=k+\frac12$, then
$I_{k+1}\to v_3\to y$ and $v_3\to v_2$.
Consequently,
$\sigma_H(P^-)\le\alpha_{k+1}+2+3+\beta_k=8$.
The case $f(v_3)=k-\frac12$ is symmetric.
No additional edge is prescribed when $f(v_3)$ is
not an integer. 
Therefore, assume that $f(v_3)=k$.
By symmetry, we may assume that the oriented edges in
$[v_3,I_k]$ point towards $v_3$ and that $v_2\to v_3$.

Suppose that $v_2v_3$ was oriented in
\Cref{step:F0X31X41}.
Then we have the following cases:
\begin{itemize}
\item If $v_3\in X_3$ and
$v_2\in R_{21}\cup R_{31}$, then $k=3$ and
$\mu_H(v_2)\le3$.
\item If $v_3\in R_{21}$ and
$v_2\in(R_{21}\cup R_{23})\cap F_2$, then
$\mu_H(v_2)\le2$.
\item If $v_3\in L_{21}$, its incoming edges from $F_0$
give $k=3$.
Since $\nu_H(v_3)\le2$ and $v_2\to v_3$,
we have $\nu_H(v_2)\le3$.
\end{itemize}
Thus, one score is at most $8$. 
No additional edge is prescribed in this case,
since $v_3\notin F_{13}$. 

Now, consider the operation that first orients $v_2v_3$,
and let $H'$ be the partial orientation before it.
This edge is not oriented in
\Cref{step:F12,step:F12224}.
 
If $v_2v_3$ is oriented in \Cref{step:cycle(auv)}, then
$\mu_H(v_2)+\nu_H(v_2)\le7$,
and one score is at most eight.
If, in this case, $y\to v_3$ is newly oriented, then
$v_3\in F_{13}$.
Since the fixed direction is towards $v_3$, the earlier
cycle contains $v$ and $k=4$; a cycle through $u$
would fix the opposite direction.
By \Cref{prop:cycleauv}, we have
$\mu_H(v_2)+1+\nu_H(v_3)\le7$.
Since $v_2\in F_2$, it holds that $\mu_H(v_2)\ge2$,
and hence $\nu_H(v_3)\le4$.
Thus, the additional edge has weight at most
$\alpha_4+1+\nu_H(v_3)\le6$.

Therefore, assume that $v_2v_3$ is oriented in
\Cref{step:cycle(a)}.
If the earlier operation considers $C_{v_3w}$ for some
$w\in I_k$, or orients a whole cycle or a four-edge path
of form (d), then $v_2v_3$ and the oriented edges in
$[v_3,I_k]$ have opposite directions at $v_3$,
a contradiction. 

Here form (c) cannot be a whole unoriented cycle,
by the preceding steps and the order of the cycles.
It remains to consider the following two cases.

\medskip
\noindent{\bf Subcase 2.1.}
The earlier path is $y'x'v'_1v_2v_3$ with four edges.

Then $y'\in I_k$.
The earlier direction is
$y'\to x'\to v'_1\to v_2\to v_3$, which does not
complete its considered cycle.
Hence, its score at that operation satisfies
$S=\alpha_k+4+\nu_{H'}(v_3)\le8$.
Since $\nu_H(v_2)\le1+\nu_{H'}(v_3)$, it follows that
$\sigma_H(P^+)\le S\le8$.
If $y\to v_3$ is an additional edge, its weight is at most
$\alpha_k+1+\nu_{H'}(v_3)=S-3\le5$.

\medskip
\noindent{\bf Subcase 2.2.}
The earlier path is $y'x'v_2v_3$ with three edges.

Let $y'\in I_\ell$.
Since the earlier operation belongs to
\Cref{step:cycle(a)}, we have $\ell\in\{2,3,4\}$. 
Its direction is $y'\to x'\to v_2\to v_3$,
and its auxiliary edge is $w\to v_3$ for some $w\in I_k$.
Put $S=\alpha_\ell+3+\nu_{H'}(v_3)$.
The two parts of this earlier dipath give
$\mu_H(v_2)+\nu_H(v_2)\le S$.

Since $\mu_{H'}(v_3)\le\alpha_k+1$, the reverse
direction has score at most $\alpha_k+4+\beta_\ell$.
The rule for equal scores prefers that reverse direction.
Thus, $S\le\alpha_k+3+\beta_\ell=7+\ell-k$.
 
If $y\to v_3$ is newly oriented, then $v_3\in F_{13}$.
Applying the observation at the beginning to
$y'x'v_2v_3$, which is unoriented in $H'$, we have
$|k-\ell|\le1$.
Since $S=\alpha_\ell+3+\nu_{H'}(v_3)
\le\alpha_k+3+\beta_\ell$, it follows that
$\nu_H(v_3)\le\alpha_k+\beta_\ell-\alpha_\ell$.
Thus, the additional edge has weight at most
$\alpha_k+1+\nu_H(v_3)
\le2\alpha_k+1+\beta_\ell-\alpha_\ell
=5+2(\ell-k)\le7$. 

If $\ell\le k$, then $S\le7$, and the current two
scores have sum at most $17$.
If $\ell\ge k+1$, the earlier path gives
$\mu_H(v_2)\le\alpha_\ell+2$, and hence
$\sigma_H(P^-)\le\alpha_\ell+5+\beta_k
=9+k-\ell\le8$.
In both cases, the selected score is at most $8$.
 
Therefore, every selected path has weight at most nine,
with equality only if it is contained in a dicycle
of length five meeting $F_0$.
Every new edge on the selected path has weight at most
the weight of that path.
The possible additional edge has been considered above.
Thus, both assertions hold. 
\end{proof}

The next step is to orient all the remaining edges in $[F_0,F_1]$.

\begin{step}\label{step:F13} 
For each $x\in F_{13}$ with an unoriented edge in $[x,F_0]$,
choose $y_1\in N(x)\cap F_0$ such that $xy_1$ is oriented
whenever possible; otherwise choose $y_1$ arbitrarily.
Process the remaining unoriented edges $xy\in[x,F_0]$ with $y\ne y_1$
one at a time. If $xy_1$ is oriented, orient $xy$ in the
opposite direction at $x$; otherwise, orient $y_1\to x\to y$.
 
\end{step}

\begin{proposition}\label{prop:E5}
    For $e\in E_{\ref{step:F13}}$, it holds that  $\omega(e)\le 6$ in $G^{O_{\ref{step:F13}}}$.  
\end{proposition}
\begin{proof}
 
Let $e\in E_{\ref{step:F13}}$ be incident to $x\in F_{13}$.
Put $k=f(x)$.
By \ref{item:F121}, $k$ is an integer, and
$N(x)\cap F_0\subseteq I_k$ with $k\in\{2,3,4\}$.
By \Cref{step:F13}, there are distinct $y,y'\in I_k$
such that $e$ belongs to $y'\to x\to y$ or its reverse.
Since $\mu(y),\mu(y')\le5-k$ and
$\nu(y),\nu(y')\le k-1$, either dipath has weight at most
$(5-k)+2+(k-1)=6$.
Thus, $\omega(e)\le6$. 
\end{proof}
Briefly, our next step is to repeatedly orient any remaining path  $P$  that currently admits an orientation as a directed path with  $\omega$-value at most  8.

\begin{step}\label{step:sumatmost8}
Repeatedly choose a path $P$  whose edges have not yet been oriented and which admits an orientation  $\overrightarrow{P}$  as a directed path with  $\omega(\overrightarrow{P})\le 8 $. Among the two possible orientations of $P $, orient $P $ in the direction that minimizes $\omega(\overrightarrow{P})$. After each such orientation, update the current oriented graph. Continue this process until no such path $P $ remains.
\end{step}

Now,  we show that in $G^{O_{\ref{step:sumatmost8}}}$, vertices in $ F_1$ satisfy the following property.
\begin{corollary}\label{claimcenter}
    Let $e$ be an oriented edge in $G^{O_{\ref{step:sumatmost8}}}$. Then $\omega(e)\le 9$ in $G^{O_{\ref{step:sumatmost8}}}$ with equality only if   there is a dicycle containing $e$ passing   $F_0$ of length  5. Especially, $\omega(e)= 9$ only possible when it is oriented in \Cref{step:cycle(a)} and lies on a dicycle of length 5 that intersects $F_0$ at $I_2\cup I_3\cup I_4$. 
It follows that $\min \{\nu(x),\mu(x)\}\leq 4$ for each $x\in V(G^{O_{\ref{step:sumatmost8}}})$.
\end{corollary}

Also, the following lemma holds. 
\begin{lemma}\label{lem:gold-deep-layer}
 In $G^{O_{\ref{step:sumatmost8}}}$, every vertex $z\in V(G^{O_{\ref{step:sumatmost8}}})$ satisfies  
 $\mu(z),\nu(z)\le8-\mathrm{dist}_G(z,F_0)$.  
Moreover, every arc $x\to y$ of $G^{O_{\ref{step:sumatmost8}}}$ satisfies 
 $\mu(x)\le7-\mathrm{dist}_G(y,F_0)$, 
 $\nu(y)\le7-\mathrm{dist}_G(x,F_0)$,$\mu(x)+\nu(y)\le8$. 
\end{lemma}
\begin{proof} 
Put $H=G^{O_{\ref{step:sumatmost8}}}$ and
$d(z)=\mathrm{dist}_G(z,F_0)$.
All parameters below are computed in $H$.
Since $u,v\in F_0$, we have $\mu(z),\nu(z)\ge d(z)$.
Moreover, $\mu(a),\nu(a)\le3$ for every $a\in F_0$:
by \Cref{obs:munu} on $I_2\cup I_3\cup I_4$ and
$\mu(u)\le2$, $\nu(u)=0$, $\mu(v)=0$, $\nu(v)\le2$.

Let $x\to y$ be an arc of $H$.
By \Cref{claimcenter},
$\mu(x)+\nu(y)=\omega(x\to y)-1\le8$.
If $\omega(x\to y)\le8$, then 
$\mu(x)\le7-\nu(y)\le7-d(y)$, and 
$\nu(y)\le7-\mu(x)\le7-d(x)$. 
Otherwise, \Cref{claimcenter} gives a dicycle $C$ of length five
containing $x\to y$ and a vertex $a\in F_0$.
Let $p,q$ be the lengths of the directed portions of
$C-(x\to y)$ from $a$ to $x$ and from $y$ to $a$,
respectively, allowing length zero.
Then $p+q=4$, $p\ge d(x)$ and $q\ge d(y)$, so 
$\mu(x)\le3+p\le7-d(y)$,  and 
$\nu(y)\le3+q\le7-d(x)$. 
This proves all the asserted arc bounds.
Finally, $\mu(y)\le\mu(x)+1$, $\nu(x)\le\nu(y)+1$ and
$|d(x)-d(y)|\le1$ give
$\mu(z),\nu(z)\le8-d(z)$ for both $z=x$ and $z=y$.
Every vertex of $H$ is incident to an arc, so the vertex bounds follow. 
\end{proof}
So far, we have already oriented some edges in $[F_0,F_0]\cup [F_1,F_1]\cup [F_1,F_2]\cup [F_0,F_1]$. Our next strategy is to partition $V$ into two subsets, one having small $\nu$-value and the other having small $\mu$-value. By \Cref{claimcenter}, $\min{\nu(x),\mu(x)}\leq 4$ for every $x\in V(G^{O_{\ref{step:sumatmost8}}})$. Then, for $x\in F_2\setminus V(G^{O_{\ref{step:sumatmost8}}})$, by the definition of $F_2$, there exists $y\in N(x)\cap F_1$. Since $F_1\subseteq V(G^{O_{\ref{step:sumatmost8}}})$, \Cref{claimcenter} implies that $\min{\nu(y),\mu(y)}\leq 4$. Thus, one may potentially orient $xy$ as $x\to y$ if $\nu(y)$ is small, or as $y\to x$ if $\mu(y)$ is small, so as to ensure that at least one of $\nu(x)$ and $\mu(x)$ is small. However, when $x\in (F_3\cup F_4)\setminus V(G^{O_{\ref{step:sumatmost8}}})$, the situation is more delicate than in the case $x\in F_2$. Formally, after the next specific orientation step, our goal is to prove that every $x\in V$ satisfies at least one of the following properties:

\begin{itemize}
    \item[(a1)] $\nu(x)\leq4$;
\item[(a2)] $\mu(x)\leq4$;
\item[(b1)] there is $y\in N(x)$ such that $y$  meets (a1), and either $xy$ is not oriented or $x\to y$;
\item[(b2)]  there is $y\in N(x)$ such that $y$ meets (a2), and either $xy$ is not oriented or $y\to x$.
\end{itemize}

 Based on \Cref{claim3},  an orientation on edges incident to  vertices in  $L_2\cup R_2$ is given.

\begin{step}\label{step:L23R23} 
The following edges are oriented in order, as shown in \Cref{fig:step6}. After each operation below, we immediately update the current oriented graph.
\begin{enumerate}[label=(S\thestep.\arabic*)]
\item\label{s6-1} For $z\in L_3\cap F_3$ that satisfies neither (b1) nor (b2) in the current oriented digraph, choose $y\in N(z)\cap L_2$. For each $x\in N(y)\cap L_1$, orient $y\to x$ if $\nu(x)\le 2$; otherwise, orient $x\to y$. Then update the current oriented digraph.
\item\label{s6-2} For $z\in R_3\cap F_3$ that satisfies neither (b1) nor (b2) in the current oriented digraph, choose $y\in N(z)\cap R_2$. For each $x\in N(y)\cap R_1$, orient $x\to y$ if $\mu(x)\le 2$; otherwise, orient $y\to x$. Then update the current oriented digraph.
\end{enumerate}
\end{step}

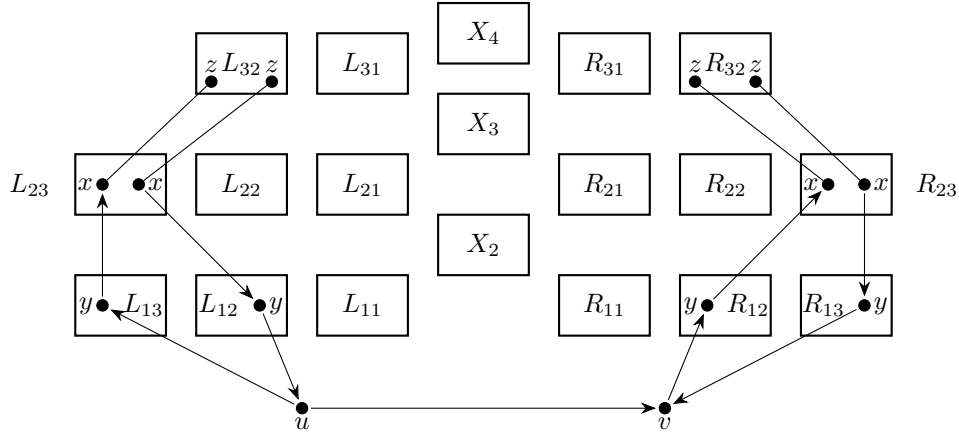
\begin{figure}[htbp]
\begin{center}
    
\begin{tikzpicture}[
    rect/.style={rectangle, draw, minimum width=1.2cm, minimum height=0.8cm, align=center, thick, inner sep=2pt},
    circ/.style={circle,
    draw,
    fill=black,
    minimum size=0.15cm,
    inner sep=0pt},
    arr/.style={-{Stealth[scale=1.2]}, shorten >=1pt, shorten <=1pt}, scale=0.8
]

\node[circ] (v) at (3,-0.2) {};
\node[circ] (u) at (-3,-0.2) {};
\node[circ] (y) at (-6.3,1.5) {};
\node[circ] (x) at (-6.3,3.5) {};
\node[circ] (z) at (-4.5,5.2) {};
\node[circ] (y1) at (-3.7,1.5) {};
\node[circ] (x1) at (-5.7,3.5) {};
\node[circ] (z1) at (-3.5,5.2) {};
\node[circ] (y') at (6.3,1.5) {};
\node[circ] (x') at (6.3,3.5) {};
\node[circ] (z') at (4.5,5.2) {};
\node[circ] (y'1) at (3.7,1.5) {};
\node[circ] (x'1) at (5.7,3.5) {};
\node[circ] (z'1) at (3.5,5.2) {};

\node[rect] (r1) at (-4,1.5) {};
\node[rect] (r2) at (-2,1.5) {};
\node[rect] (l5) at (-6,1.5) {};
\node[rect] (r3) at (2,1.5) {};
\node[rect] (r4) at (4,1.5) {};
\node[rect] (l6) at (6,1.5) {};

\node[rect] (r5) at (-4,3.5) {};
\node[rect] (r6) at (-2,3.5) {};
\node[rect] (r7) at (2,3.5) {};
\node[rect] (r8) at (4,3.5) {};
\node[rect] (l7) at (-6,3.5) {};
\node[rect] (l8) at (6,3.5) {};


\node[rect] (r9) at (-4,5.5) {};
\node[rect] (r10) at (-2,5.5) {};
\node[rect] (r11) at (2,5.5) {};
\node[rect] (r12) at (4,5.5) {};

\node[rect] (r13) at (0,2.5) {};
\node[rect] (r14) at (0,4.5) {};
\node[rect] (r15) at (0,6) {};

 \draw[arr] (u)--(v);

\node[xshift=-0.3cm] at (r1) {$L_{12}$};
\node at (r3) {$R_{11}$};
\node at (r2) {$L_{11}$};
\node[xshift=0.3cm] at  (l5) {$L_{13}$};
\node[xshift=0.3cm] at   (r4) {$R_{12}$};
\node[xshift=-0.3cm] at  (l6) {$R_{13}$};

\node at (r6) {$L_{21}$};
\node at (r7) {$R_{21}$};
\node at (r5) {$L_{22}$};
\node at (r8) {$R_{22}$};
\node[xshift=-1.2cm] at (l7) {$L_{23}$};
\node[xshift=1.2cm] at (l8) {$R_{23}$};

\node at (r10) {$L_{31}$};
\node at(r11) {$R_{31}$};
\node at (r9) {$L_{32}$};
\node  at (r12) {$R_{32}$};

\node  at (r13) {$X_2$};
\node  at (r14) {$X_{3}$};
\node at (r15)   {$X_{4}$};

\draw[arr](u)--(y);
\draw[arr](y)--(x);
\draw (z)--(x);
\draw[arr] (y1)--(u);
\draw[arr](x1)--(y1);
\draw (z1)--(x1);

\draw[arr](y')--(v);
\draw[arr](x')--(y');
\draw (z')--(x');
\draw[arr] (v)--(y'1);
\draw[arr](y'1)--(x'1);
\draw (z'1)--(x'1);

\node[below] at (v)   {$v$};
\node[below]  at(u) {$u$};

\node[left]at (y)  {$y$};
\node[left] at(x) {$x$};
\node[above]at (z) {$z$};
\node[right]at (y1) {$y$};
\node[right]at (x1) {$x$};
\node[above]at (z1)  {$z$}; 
\node[right]at (y')  {$y$};
\node[right] at(x') {$x$};
\node[above]at (z') {$z$};
\node[left]at (y'1) {$y$};
\node[left]at (x'1) {$x$};
\node[above]at (z'1)  {$z$}; 
\end{tikzpicture}
\caption{An example for   \Cref{step:L23R23} with $z\in F_3\cap (L_{32}\cup R_{32})$, $x\in L_{23}\cup R_{23}   $ and $y\in (L_1\cup R_1)\setminus(L_{11}\cup R_{11}) $.}\label{fig:step6}
\end{center}
\end{figure}
Now, we show that the desired result holds. 
\begin{proposition}
    \label{claimcenter2}
   After \Cref{step:L23R23}, for each $x\in V $, 
 if some edge incident to $x$ has been oriented, then $x$ satisfies (a1) or (a2), otherwise, $x$ satisfies (b1) or (b2).
\end{proposition}
 
\begin{proof}
By \Cref{claimcenter}, every vertex incident to an oriented edge
satisfies (a1) or (a2) before \Cref{step:L23R23}.
We show that this remains true after each operation of that step.

Suppose that the current operation considers $z\in L_3\cap F_3$
which satisfies neither (b1) nor (b2), and chooses
$y\in N(z)\cap L_2$.
If $z$ satisfied (a1) or (a2), the first or last edge of a
dipath attaining the corresponding adjusted distance would give
(b1) or (b2), respectively.
Thus, no edge incident to $z$ is oriented.
If $y$ were incident to an oriented edge, the induction hypothesis
and the unoriented edge $zy$ would again give (b1) or (b2) for $z$.
Hence, all edges incident to $y$ are unoriented.

Since $z\in F_3$, we have $y\in F_2$ and
$N(y)\cap L_1\subseteq F_1$.
For each $x\in N(y)\cap L_1$, the edge $xu$ is already oriented.
If $\nu(x)\le2$, orienting $y\to x$ gives $\nu(y)\le3$.
Otherwise, $u\to x$, so $\mu(x)\le3$, and orienting
$x\to y$ gives $\mu(y)\le4$.
Thus, the operation is well-defined, its only newly incident
vertex $y$ satisfies (a1) or (a2), and the unoriented edge $zy$
gives (b1) or (b2) for $z$.
The case $z\in R_3\cap F_3$ is symmetric.
This proves the assertion for every vertex incident to an oriented edge.

Now, let $x$ have no incident oriented edge after this step.
Since $F_0\cup F_1$ is already contained in the oriented graph,
we have $x\in F_2\cup F_3\cup F_4$.
If $x\in F_2$, a neighbor in $F_1$ gives (b1) or (b2).
If $x\in (L_3\cup R_3)\cap F_3$, the preceding argument and
the rules of \Cref{step:L23R23} give the same conclusion.

It remains to consider $x\in(F_3\cup F_4)\setminus(L_3\cup R_3)$.
By \Cref{claim3}, such a vertex belongs to $X_{31}\cup X_4$.
A vertex of $X_{31}\cap F_3$ has no neighbor in $X_2$,
so shortest paths to $u$ and $v$ force neighbors in both
$L_2$ and $R_2$; it is therefore incident to an edge oriented
in \Cref{step:F0X31X41}.
Vertices of $X_{41}$ are also incident to edges oriented in that step.
Hence, $x\in X_{42}$.
By \Cref{claim1}, choose $y\in N(x)\cap X_3$.
Since $x\in F_3\cup F_4$, we have $N(y)\cap X_2=\emptyset$.
Thus, $y$ has neighbors in both $L_2$ and $R_2$, and
\Cref{step:F0X31X41} gives $\mu(y),\nu(y)\le3$.
The edge $xy$ is unoriented, so $x$ satisfies both (b1) and (b2).
\end{proof}

\begin{remark} 
 \Cref{claimcenter2} indicates that if some vertex $x$  meet neither (a1) nor (a2), then all the edges incident  $x$  are unoriented after \Cref{step:L23R23}.
\end{remark}




\section{Potentials and completion of the orientation}\label{sec:F234}
 
 The preliminary orientation gives each vertex either a small value of at least one distance parameter or an unoriented edge to a vertex with such a bound, as shown in \Cref{claimcenter2}. We now refine the vertex partition, orient the crossing edges, and introduce potentials to guide the final step.  

\subsection{Refined vertex partitions and orienting the crossing edges}\label{sec:refined-partitions}
Now, by \Cref{claimcenter2}, we  partition the vertices into two subsets $A$ and $B$ such that each vertex $x\in A$ has small  $\nu(x)$ or $\mu(x)$ while each vertex $y\in B$ is adjacent to a vertex $x_y$ in $A$ where the edge $x_yy$ is unoriented. Let

     \begin{align*}
A&=\{x\in V : x \text{~meets~} (a1) \text{~or } (a2)\}; \\
A'&=\{x\in V : x \text{~meets~} (a1)  \text{ and } (a2)\}; \\
A_1&=  A'\cap \left(\{x\in R_1:\mu(x)\le 2 \}\cup  \{x\in L_1: \nu(x)\ge 3\}\cup \{x\in (L_2\cup L_3)\setminus F_1: \nu(x)\le 4\}\right);\\
A_2&=  A'\cap (\{x\in L_1:\nu(x)\le 2 \}\cup  \{x\in R_1:  \mu(x)\ge 3\}\cup \{x\in (R_2\cup R_3)\setminus F_1: \mu(x)\le 4\});\\
D_a&=\{x\in A:x \text{~meets~} (a1)\}\setminus   A_1;\\
U_a&=\{x\in A:x \text{~meets~} (a2)\}\setminus A_2;\\
M_a&=D_a\cap  U_a.\label{partitionA}\tag{$\mathcal{A}$} 
\end{align*}

\begin{remark}\label{partitionAB}
All the following hold. 
\begin{itemize}
\item $A_1\cap A_2=\emptyset$ and so $D_a\cup U_a=A$. 
    \item 
If $x\in L_1\cup R_1$, then $x\notin M_a$. 
\end{itemize}
\end{remark}
Then we continue to consider the set $V\setminus A$. Let  
    \begin{align*}
    B&=V \setminus A; &~&\\
 D'_b&=\{x\in B: N(x)\cap D_a\neq\emptyset\}; & U'_b &=\{x\in B: N(x)\cap U_a\neq\emptyset\};\\
 M'_b&=D'_b\cap U'_b; & M_b&=\{x\in M'_b:|N(x)\cap A|\ge 2\};\\
 B_D&=D'_b\setminus M'_b, & B_U&=U'_b\setminus M'_b. \tag{$\mathcal{B}$}\label{eq:partitionB}
 \end{align*}  

It looks like that vertices in $M'_b$ have strong properties, however, considering our next orientation step, some of them are not so easy to handle, specifically, vertices in the set $$C:=\{x\in M'_b:|N(x)\cap A|=1\}=M'_b\setminus M_b. $$ Thus, our target is to define a partition $C_D\cup C_U$ of   $C$ so that each vertex in $C_D$ has a neighbor in $B_U\cup M_b\cup C_U$ and each vertex in $C_U$ has a neighbor in $B_D\cup M_b\cup C_D$. Then orient edges in $[M_b,A]$ so that for each $x\in M_b$, we have $\mu(x)\le 5$ and $\nu(x)\le 5$ since $|N(x)\cap A|\ge 2$. Then for each $x\in C_U$ suppose that $y\in N(x)\cap (B_D\cup M_b\cup C_D) $. Then orient $U_a\to x\to y\to D_a$ when edges in $[y,D_a]$ are unoriented or $ U_a\to x\to y$ when $y\in M_b$. Then $\nu(x)\le 6$ and $\mu(x)\le 5$. The partition $C_U\cup C_D$ is not so natural because of the proofs in the last section. We emphasis here that  $X_{42}$ could be write  as $$X_{42}=X^{L}_{42}\cup X^{R}_{42}\cup X^0_{42},$$
where each vertex in $X^{L}_{42}\subseteq X_{42}$ has a neighbor in $L_{3}$, each vertex in $X^{R}_{42}\subseteq X_{42}$ has a neighbor in $R_{3}$, each vertex in $X^0_{42}\subseteq X_{42}$ has no neighbor in $R_{3}\cup L_3$.

\begin{align*}
M_r&=\{x\in  C\cap     (X^R_{42}\cup X^0_{42})   : N(x)\cap  B_U\neq\emptyset\}\\
&\cup \{x\in  C  \cap    X^L_{42}  :   N(x)\cap  B_D=\emptyset\text{ and } N(x)\cap  B_U\neq\emptyset\}
\\ &\cup \{x\in C\cap R : N(x)\cap A\subseteq R,  N(x)\cap B_D=\emptyset\text{ and }N(x)\cap B_U \neq\emptyset\}
\\&\cup \{x\in C\cap R : N(x)\cap A\subseteq V\setminus R \text{ and }  N(x)\cap B_U\neq\emptyset \}\\&\cup \{x\in C\cap L : N(x)\cap A\subseteq L \text{ and } N(x)\cap B_U\neq\emptyset \} \\&\cup \{x\in C\cap L : N(x)\cap A\subseteq V\setminus L \text{ and } N(x)\cap B_D=\emptyset  \text{ and } N(x)\cap B_U\neq \emptyset\};\\
M_\ell&=\{x\in  C  \cap    X^L_{42}  :   N(x)\cap B_D\neq\emptyset\}\\&\cup \{x\in C\cap     (X^R_{42}\cup X^0_{42})   : N(x)\cap B_U=\emptyset \text{ and }  N(x)\cap B_D\neq\emptyset\} \\ &\cup \{x\in C\cap L : N(x)\cap A\subseteq L, N(x)\cap B_D\neq\emptyset\text{ and } N(x)\cap B_U=\emptyset\}  
\\ &\cup \{x\in C\cap L : N(x)\cap A\subseteq V\setminus L \text{ and } N(x)\cap B_D\neq\emptyset\}\\   &\cup \{x\in C\cap R : N(x)\cap A\subseteq R \text{ and } N(x)\cap B_D\neq \emptyset\}\\&\cup \{x\in  C\cap R : N(x)\cap A\subseteq V\setminus R \text{ and } N(x)\cap  B_U=\emptyset \text{ and }N(x)\cap  B_D\neq \emptyset \}; \\ 
M'_1&=\left\{x\in M'_b\setminus (M_b\cup M_r\cup M_\ell): x \text{ is~isolated~in~} M'_b\setminus (M_b\cup M_r\cup M_\ell)\right\};\\
M'_2&=M'_b\setminus (M_b \cup M_r\cup M_\ell\cup M'_1).\tag{$\mathcal{M}$}\label{partitionM}
\end{align*}  
Every vertex $x\in M'_1$ has a neighbor in $M_b\cup M_r\cup M_\ell$.
Indeed, $B\subseteq L\cup R\cup X_{42}$ by the preceding construction,
so the definitions of $M_r$ and $M_\ell$ give
$N(x)\cap(B_D\cup B_U)=\emptyset$.
Since $|N(x)\cap A|=1$ and $d_G(x)\ge2$,
the vertex $x$ has a neighbor in $M'_b$.
Its isolation in the remaining set forces this neighbor into
$M_b\cup M_r\cup M_\ell$.  
Thus, let $M'_{11}=\{x\in M'_1:N(x)\cap  (M_b\cup M_\ell) \neq\emptyset\}$, and $M'_{12}=M'_1\setminus M'_{11}$. Also, let $M'_{21}\cup M'_{22}$ be a partition of $M'_2$ satisfying that each vertex has a neighbor in the other part since such a partition could be found by considering a spanning forest of $M'_2$. 
Now, we define the final partition of set $B$ as follows:
$$\begin{aligned}
D_b&=(D'_b\setminus M'_b)\cup M_b\cup M_r\cup M'_{11}\cup M'_{21};\\
U_b&=(U'_b\setminus M'_b)\cup M_b\cup M_\ell\cup M'_{12}\cup M'_{22}. 
\end{aligned}$$ 
Observe that $D_b\cup U_b=D'_b\cup U'_b=B$ and $D_b\cap U_b=M_b$.
 Write $D=D_a\cup D_b$ and $U=U_a\cup U_b$. 
\begin{remark}\label{rm:edgeanticross}
   All the following hold.
   \begin{itemize}
       \item By definition, $[U_a\setminus M_a, D_b\setminus M_b]=\emptyset$ and $[D_a\setminus M_a, U_b\setminus M_b]=\emptyset$. 
       
       Suppose not, that there are $x\in U_a\setminus M_a=U_a\setminus D_a$ and $y\in D_b\setminus M_b$ with $xy\in E$. Since $D_b\subseteq D'_b$, by definition of $D'_b$, it holds that $N(y)\cap D_a\neq\emptyset$. Then $|N(y)\cap A|\ge 2$ and so $y\in M_b$, a contradiction. 
       \item If $x\in (M'_b\setminus M_b)\cap U_b$, then $N (x)\cap D_b\neq\emptyset$. If $x\in (M'_b\setminus M_b)\cap D_b$, then $N(x)\cap U_b\neq\emptyset$.

 Since $x\in (M'_b\setminus M_b)\cap U_b$, it holds that $x\in M_\ell\cup M'_{12}\cup M'_{22}$. If $x\in M'_{22}\cup M'_{12}$, then $N(x)\cap D_b\neq\emptyset$. Thus, assume that $x\in M_\ell$. By the definition of $M_\ell$, again, we have that $N(x)\cap D_b\neq\emptyset$. 
   \end{itemize} 
\end{remark}

Now, we deal with unoriented edges related to $[D_b,U_p]$ and $[U_b,D_q]$ where $p,q\in \{a,b\}$.

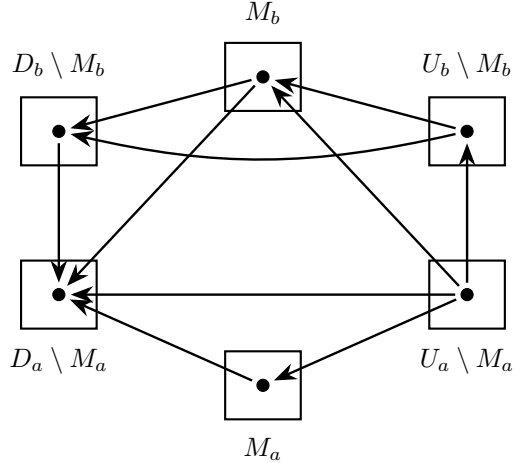
\begin{figure}[htbp]
\begin{center}
\begin{tikzpicture}[
    rect/.style={
        rectangle,
        draw,
        minimum width=1cm,
        minimum height=0.9cm,
        align=center,
        thick,
        inner sep=2pt
    },
    circ/.style={
        circle,
        draw,
        fill=black,
        minimum size=0.16cm,
        inner sep=0pt
    },
    arr/.style={
        -{Stealth[scale=1.2]},
        shorten >=2pt,
        shorten <=2pt,
        line width=0.9pt
    },
    scale=0.6
]

\node[rect] (DbBox) at (-4.5, 1.6) {};
\node[rect] (MbBox) at ( 0.0, 2.8) {};
\node[rect] (UbBox) at ( 4.5, 1.6) {};

\node[rect] (DaBox) at (-4.5,-2.0) {};
\node[rect] (MaBox) at ( 0.0,-4) {};
\node[rect] (UaBox) at ( 4.5,-2.0) {};

\node[circ] (db) at (DbBox.center) {};
\node[circ] (mb) at (MbBox.center) {};
\node[circ] (ub) at (UbBox.center) {};

\node[circ] (da) at (DaBox.center) {};
\node[circ] (ma) at (MaBox.center) {};
\node[circ] (ua) at (UaBox.center) {};


\node[above=3pt] at (DbBox.north)
    {$D_b\setminus M_b$};

\node[above=3pt] at (MbBox.north)
    {$M_b$};

\node[above=3pt] at (UbBox.north)
    {$U_b\setminus M_b$};

\node[below=3pt] at (DaBox.south)
    {$D_a\setminus M_a$};

\node[below=3pt] at (MaBox.south)
    {$M_a$};

\node[below=3pt] at (UaBox.south)
    {$U_a\setminus M_a$};
\draw[arr] (mb) -- (db);
\draw[arr] (ub) -- (mb);

\draw[arr] (ub) to[bend left=13] (db);

\draw[arr] (db) -- (da);

\draw[arr] (mb) -- (da);
\draw[arr] (ua) -- (ub);

\draw[arr] (ua) -- (mb);
\draw[arr] (ua) -- (da);

\draw[arr] (ua) -- (ma);
\draw[arr] (ma) -- (da);

\end{tikzpicture}

\caption{An  illustration of the orientations in
\Cref{step:DUb}.}
\label{fig:stepDUb}
\end{center}
\end{figure}

\begin{step}\label{step:DUb}
The following edges are oriented in order (the final structure see \Cref{fig:stepDUb} for an example). 
\begin{enumerate}[label=(S.\thestep.\arabic*)] 
\item\label{S7-1} Let $x\in M_b$. Orient $x\to D_a\setminus M_a$ and $ U_a\setminus M_a\to x$. When   $N(x)\cap (U_a\setminus M_a)=\emptyset$ and $N(x)\cap (D_a\setminus M_a)=\emptyset$, orient $y\to x\to y'$ for some $y\in N(x)\cap M_a$ and each $y'\in (N(x)\cap M_a)\setminus \{y\}$.  When   $N(x)\cap (U_a\setminus M_a)=\emptyset$ and $N(x)\cap (D_a\setminus M_a)\neq\emptyset$, orient   $M_a\to x$.  When     $N(x)\cap (U_a\setminus M_a)\neq\emptyset$, orient   $  x\to M_a$. 
\item\label{S7-2} For all unoriented edges $xy$ with $x\in D_b\setminus M_b$ and $y\in N(x)\cap (U_b\setminus M_b)$,  orient the unoriented edges in $[U_a,y]\cup \{xy\}\cup [x,D_a]$ in the direction that $U_a\to y\to x\to D_a$.

\item\label{S7-3} For all unoriented edges $xy$ with $x\in  M_b$ and $y\in N(x)\cap (U_b\setminus M_b)$,  orient the unoriented edges in $[U_a,y]\cup \{xy\} $ in the direction that $U_a\to y\to x $.  For all unoriented edges $xy$ with $x\in  M_b$ and $y\in N(x)\cap (D_b\setminus M_b)$,  orient the unoriented edges in $[D_a,y]\cup \{xy\} $ in the direction that $ x\to y\to D_a$.  
    \item\label{S7-4} For all unoreinted edges $xy$ with $x\in U_a$ and $y\in D_a\setminus M_a$   or $x\in U_a\setminus M_a$ and $y\in  M_a$, orient $x\to y$.   
\end{enumerate}
\end{step}
\begin{remark}\label{rm:crossing}
All the following hold. 
\begin{itemize}
    \item   For each   $xy\in [D_b\setminus M_b,D_a]$ oriented in  \Cref{step:DUb}  with $x\in D_b$ and $y\in D_a$, it holds that $xy$ is oriented as $x\to y$.  For each $xy\in [U_b\setminus M_b,U_a]$ oriented in  \Cref{step:DUb} with $x\in U_b$ and $y\in U_a$, it holds that $xy$ is oriented as $y\to x$.
    \item The orientation in \ref{S7-2} and \ref{S7-3} has no conflict.     
    \item If $x\in M'_b$, by \Cref{rm:edgeanticross}, some edge incident to $x$ must be oriented in \Cref{step:DUb}. 
\end{itemize}
\end{remark}

Then the following result holds.
\begin{proposition}\label{claim:xinDyinU}
After \Cref{step:DUb}, all the following hold.
\begin{itemize}
\item[(1)] Every vertex of $M_b$ has an incoming edge from $U_a$
and an outgoing edge to $D_a$, and satisfies $\mu,\nu\le5$.
\item[(2)] If $x\in D_b\setminus M_b$ is incident to an edge
oriented in this step, then all its $D_a$-edges are oriented
from $x$ to $D_a$, and $\nu(x)\le5$, $\mu(x)\le6$.
The symmetric assertion holds for $U_b\setminus M_b$.
\item[(3)] Every vertex of $C$ is incident to an edge oriented in
this step. Every vertex of $A$ incident to an edge oriented
in this step satisfies $\mu,\nu\le7$.
\item[(4)] Every edge in $[D,U]\setminus([M_a,M_a]\cup[M_b,M_b])$
has been oriented. In particular, an unoriented edge with
one endpoint in $D\setminus U$ has both endpoints in
$D\setminus U$; an unoriented edge with one endpoint in
$U\setminus D$ has both endpoints in $U\setminus D$.
\item[(5)] For $x\in D_a\cup D_b$ and $y\in U_a\cup U_b $ with $xy\in E$, if $xy$ is oriented as  $y\to x$ in $G^{O_{\ref{step:DUb}}}$,  then  $\nu(x)\leq 5$, $\nu(y)\leq6$, $\mu(y)\leq 5$ and $\mu(x)\leq6$. 
\end{itemize}
  
\end{proposition}
\begin{proof}
By \Cref{step:DUb} and \Cref{rm:crossing}, (1)--(4) holds. 

  For (5), suppose that $y\to x$ is oriented in \Cref{step:DUb}. Then $\nu(x)\leq5$ and $\mu(y)\leq5$ by \Cref{step:DUb}. Thus,   $\nu(y)\leq1+\nu(x)\leq6$ and $\mu(x)\leq1+\mu(y)\leq6$ since $y\to x$. 

     Suppose that  $y\to x$ is  oriented before \Cref{step:DUb}.  Then $x\in D_a$ and $y\in U_a$  since vertices in $B$ are disjoint with $V(G^{O_{\ref{step:L23R23}}})$ by \Cref{claimcenter2}. By the definition of $D_a$ and $U_a$, we have that $\nu(x)\leq4$ and $\mu(y)\leq4$. Thus, $\mu(x)\leq1+\mu(y)\leq5$ and $\nu(y)\leq1+\nu(x)\leq5$ by $y\to x$.      
\end{proof}

\subsection{Good paths and potential bounds}\label{sec:good-paths}
 Note that all the edges of form $[D_a\cup D_b,U_a\cup U_b]\setminus [M_a,M_a]\cup [M_b,M_b]$ have been  oriented after \Cref{step:DUb}. However, there may still have some edges of $G$ which are not oriented yet. By \Cref{claimcenter2}, if $x\in D_a\cup D_b$  (resp. $x\in U_a\cup U_b$), then $\nu(x)$ (resp. $\mu(x)$) may be small, and so it is enough to consider $\mu(x)$ (resp. $\nu(x)$). In order to reach the goal that each $x\in V $ has both small $\nu(x)$ and $\mu(x)$, we propose the definition of ``good path".

Recall that $M^O_i$ is the resulting  mixed graph  after Step $i$. 
For $x\in D_a\cup D_b$, if there is a $tx$-mixpath $P$ in $M^O_{\ref{step:DUb}}$ where $t\in \{u,v\}$ satisfying that 
\begin{itemize}
 \item  there exists a vertex $w_x\in V(P)$ such that the vertices in sub $w_xx$-path  are in $D_a\cup D_b$,  
 \item edges in $w_xx$-path are either oriented  towards $x$ or unoriented, 
and 
    \item  the subpath, $tw_x$-path, is oriented as a dipath from $t$ to $w_x$ in $G^{O_{\ref{step:DUb}}}$. 
\end{itemize}
then we call $P$ a  {\it good path} of $x$ in $M^O_{\ref{step:DUb}}$.   Denote it by $P_{vx}$ if $t=v$, and $P_{ux}$ if $t=u$. 
Let $\ell(P_{vx})=e(P_{vx})$ and $\ell(P_{ux})=e(P_{ux})+2$. Then define  $$
\U(x)=\min\{\ell(P): P \text{ is a good path of } x
\text{ in } M^O_{\ref{step:DUb}}\}.
$$ 
Note that for some $x$, it is possible that there is no good path in $M^O_{\ref{step:DUb}}$, then let $\U(x)= \infty$.
Similarly, for $y\in U_a\cup U_b$, if there is  a $yt$-mixpath $P$ in $M^O_{\ref{step:DUb}}$ where $t\in \{u,v\}$ satisfying that 
\begin{itemize}
 \item   there exists a vertex $w_y\in V(P)$ such that the vertices in sub $yw_y$-path
 are in $U_a\cup U_b$,  
 \item edges in $yw_y$-path are either towards $w_y$ or unoriented, 
and 
    \item the subpath, $w_yt$-path  is oriented as a dipath from $w_y$ to $t$ in $G^{O_{\ref{step:DUb}}}$.  
   \end{itemize}
then we call $P$ a  good path  of $y$  in $M^O_{\ref{step:DUb}}$.   Denote it by $P_{yv}$ if $t=v$, and $P_{yu}$ if $t=u$. 
Let $\ell(P_{yv})= e(P_{yv}) +2$ and $\ell(P_{yu})= e(P_{yu}) $. 
Then define  $$
\D(y)=\min\{\ell(P): P \text{ is a good path of } y
\text{ in } M^O_{\ref{step:DUb}}\}.
$$
Also, for some $y$, it is possible that there is no good path in $M^O_{\ref{step:DUb}}$, then let $\D(y)=\infty$.

In the above definitions, a subpath is allowed to have length zero.
All good paths and the values of $\U$ and $\D$ are defined in
$M^O_{\ref{step:DUb}}$ and remain fixed in the subsequent
orientation steps. 

\begin{remark}\label{lem:good-path-gluing}
By the definition of good path, all the following hold.
\begin{itemize}
    \item[(1)] For $x\in D_a\cup D_b$ and $y\in U_a\cup U_b$,
    we have that $\U(x)\leq\mu(x)$ and $\D(y)\leq\nu(y)$,
    where $\mu$ and $\nu$ are computed in
    $G^{O_{\ref{step:DUb}}}$.

    \item[(2)] Let $x,y\in D_a\cup D_b$ with $xy\in E$.
    If $xy$ is unoriented or oriented as $x\to y$ in
    $M^O_{\ref{step:DUb}}$ and $\U(x)<\infty$, then
    $\U(y)\leq\U(x)+1$.

    \item[(3)] Let $x,y\in U_a\cup U_b$ with $xy\in E$.
    If $xy$ is unoriented or oriented as $x\to y$ in
    $M^O_{\ref{step:DUb}}$ and $\D(y)<\infty$, then
    $\D(x)\leq\D(y)+1$.
\end{itemize}
Consequently, if $xy$ is unoriented in
$M^O_{\ref{step:DUb}}$ and $x,y\in D_a\cup D_b$, then
$\U(x)$ is finite if and only if $\U(y)$ is finite. 
When both values are finite, it holds that
$  |\U(x)-\U(y)|\leq1.$ 
The same argument holds for $\D$ when $x,y\in U_a\cup U_b$.
\end{remark}

Now, we show that actually in $M^O_{\ref{step:DUb}}$,  $\U(s)$ for $s\in D_a\cup D_b$ and $\D(t)$ for $t\in U_a\cup U_b$  are actually finite, in  \Cref{claimcontrol1}, \Cref{calim:boundmunuV-GO1}, \Cref{claim:X43F3bdd} and \Cref{claim:munuF2} according to the locations of $s$ and $t$.   

\begin{lemma}\label{lem:saturation-stability}
Let $H=G^{O_{\ref{step:sumatmost8}}}$, and suppose that
$\mu_H(s)+\nu_H(s)\le9$ for every $s\in V(H)$, as asserted
in \Cref{claimcenter}.
For every orientation $\overrightarrow{G}$ extending $H$
and every $s\in V(H)$, it holds that
$\mu_{\overrightarrow{G}}(s)=\mu_H(s)$ and
$\nu_{\overrightarrow{G}}(s)=\nu_H(s)$.
Consequently, $\U(s)=\mu_H(s)$ for $s\in V(H)\cap D$,
and $\D(s)=\nu_H(s)$ for $s\in V(H)\cap U$.
\end{lemma}
\begin{proof}
Let $\overrightarrow{G}$ be an orientation extending $H$.
By monotonicity, $\mu_{\overrightarrow{G}}(s)\le\mu_H(s)$
and $\nu_{\overrightarrow{G}}(s)\le\nu_H(s)$ for every
$s\in V(H)$.

Suppose that $\mu_{\overrightarrow{G}}(s)<\mu_H(s)$
for some $s\in V(H)$.
Choose a dipath $P$ from $t\in\{u,v\}$ to $s$.
Put $\varepsilon=0$ if $t=v$ and $\varepsilon=2$ if $t=u$,
so that $e(P)+\varepsilon=\mu_{\overrightarrow{G}}(s)$.
For each $x\in V(P)$, let $P_x$ be the subpath of $P$
from $t$ to $x$. Let $b$ be the first vertex of $H$ on $P$ such that
$e(P_b)+\varepsilon<\mu_H(b)$, 
and let $a$ be the last vertex of $V(H)$ before $b$ on $P$.
Such a vertex exists, since $\mu_H(v)=0$ and $\mu_H(u)\le2$
imply that $b\ne t$.
Let $Q$ be the $ab$-subpath of $P$.
Then $Q$ has no internal vertex in $V(H)$. 
By the choice of $b$, we have that
$\mu_H(a)\le e(P_a)+\varepsilon$ and
$e(P_b)+\varepsilon\le\mu_H(b)-1$.
Consequently, $\mu_H(a)+e(Q)\le\mu_H(b)-1$.

If $Q$ is an arc of $H$, this contradicts
$\mu_H(b)\le\mu_H(a)+1$.
Otherwise, all edges of $Q$ are unoriented in $H$,
since $Q$ has no internal vertex in $H$.
By \Cref{claimcenter}, it follows that
$\mu_H(a)+e(Q)+\nu_H(b)
\le\mu_H(b)+\nu_H(b)-1\le8$,
contrary to the stopping condition of
\Cref{step:sumatmost8}.
Hence, $\mu_{\overrightarrow{G}}(s)=\mu_H(s)$.

The proof for $\nu$ is symmetric.
 
Finally, let $s\in V(H)\cap D$.
By \Cref{lem:good-path-gluing} and monotonicity,
$\U(s)\le\mu_H(s)$.
Choose a good path $P$ of $s$ with $\ell(P)=\U(s)$,
and orient its unoriented edges along $P$.
Extend the resulting partial orientation to an orientation
$\overrightarrow{G}$ of $G$.
By the first assertion,
$\mu_H(s)=\mu_{\overrightarrow{G}}(s)\le\ell(P)=\U(s)$.
Thus, $\U(s)=\mu_H(s)$.
The assertion for $\D$ is symmetric. 
\end{proof}

\begin{proposition}\label{claimcontrol1}
 Let $s\in V(G^{O_{\ref{step:F0X31X41}}})$. Then all of the following hold in $G^{O_{\ref{step:DUb}}}$. 
\begin{itemize}
    \item [(1)]If $s\in L_{11}$, then  $\nu(s)\leq1$ and $\mu(s)\leq 3$. If $s\in R_{11}$, then $\nu(s)\leq 3$ and $\mu(s)\leq1$.
    \item [(2)] If $s\in L_{21}$, then $\nu(s)\leq2$ and $\mu(s)\leq6$. If $s\in R_{21}$, then $\nu(s)\leq6$ and $\mu(s)\leq2$.
    \item [(3)] If $s\in L_{31}$, then $\nu(s)\leq3$ and $\mu(s)\leq5$. If  $s\in R_{31}$, then $\nu(s)\leq5$ and $\mu(s)\leq3$.
    \item [(4)]  If  $s\in X_2$, then  $\nu(s)\leq2$ and $\mu(s)\leq2$.
    \item [(5)] If $s\in X_{31}\cap V(G^{O_{\ref{step:F0X31X41}}})\setminus\tilde{X}_3$, then $\nu(s)\leq3$ and $\mu(s)\leq3$. If  $s\in V(G^{O_{\ref{step:F0X31X41}}})\cap\tilde{X}_3$, then $\nu(s)\leq3$ and $\mu(s)\leq4$ or $\nu(s)\leq4$ and $\mu(s)\leq3$.
    \item [(6)] If  $s\in X_{41}$, then $\nu(s)\leq4$ and $\mu(s)\leq 4$.
\end{itemize}
\end{proposition}
\begin{proof}
    We can obtain the value of $\nu(s)$ and $\mu(s)$ for $s\in V(G^{O_{\ref{step:F0X31X41}}})$ by \Cref{step:F0X31X41}.
\end{proof}

\begin{proposition}\label{calim:boundmunuV-GO1}
    The following hold in $M^{O_{\ref{step:DUb}}}$. 
    \begin{itemize}
        \item[(1)] If $s\in L_{12}\cup L_{13}$, then $\nu(s)\le 5$ and $\mu(s)\leq 7$, furthermore, if $\mu(s)\in \{6,7\}$, then $R_{2}\neq\emptyset$ or $s$ lies on a dicycle which passes   $u$. If $s\in R_{12}\cup R_{13}$, then $\mu(s)\le 5$ and $\nu(s)\leq 7$, furthermore, if $\nu(s)\in \{6,7\}$, then $L_{2}\neq\emptyset$  or $s$ lies on a dicycle which passes   $v$.
\item[(2)] If $s\in L_{22} \cup ( L_{23}\cap F_1)$, then  $\nu(s)\leq6$ and $\mu(s)\leq7$. If $s\in R_{22} \cup (R_{23}\cap F_1)$, then $\nu(s)\leq7$ and $\mu(s)\leq6$. 
        \item[(3)] If $s\in X_{32}$, then $\nu(s)\leq6$ and $\mu(s)\leq6$. 
 \end{itemize}
\end{proposition}
\begin{proof} 
To see (1), by symmetry, suppose that $s\in L_{12}\cup L_{13}$.
The edge $su$ is oriented before \Cref{step:sumatmost8}
and is not oriented in \Cref{step:cycle(a)}.
By \Cref{claimcenter}, its weight is at most eight.
Thus, $\mu(s)+\nu(s)\le8$.
Since $\mu(s)\ge3$ and $\nu(s)\ge1$, it follows that
$\mu(s)\le7$ and $\nu(s)\le5$. 

    Now, suppose that $\mu(s)\in\{6,7\}$. We consider where $su$ is oriented.   If $\orw{su}\in E_{\ref{step:F0X31X41}}$, then $su$ is oriented in \ref{item:step:F0X31X411} and $s\in L_{12}$  with a dipath from $v$ to $s$ of form
$vR_1R_2R_3L_3L_2s$ or
$vR_1R_2R_3(X_4\cup\widetilde X_3)L_3L_2s$.
In either case, $R_2\ne\emptyset$.  Thus, assume that $s\in F_1\setminus V(G^{O_{\ref{step:F0X31X41}}})$.   
If $\orw{su}\in E_{\ref{step:cycle(auv)}}$, then
$su$ is oriented as $s\to u$, since $u\to s$ would give
$\mu(s)\le3$.
By \Cref{prop:cycleauv}, $\mu(s)+1\le7$, and hence
$\mu(s)=6$ and $\omega(s\to u)=7$.
If the operation orients a whole cycle, the required
dicycle is already obtained.
Otherwise, its selected path also has weight seven.
When its associated vertex is $u$, the equality assertion
of \Cref{prop:cycleauv} gives the required dicycle.
When its associated vertex is in $I_2$, a newly oriented
edge incident to $u$ occurs only on the four-edge path
of form (d), and that operation completes its considered
cycle as a dicycle through $u$.
Thus, this case is proved.
  If  $\orw{su}\in  E_{\ref{step:F13}}$, then $\mu(s)\le 5$ by \Cref{prop:E5}, a contradiction.  Thus, assume that $\orw{su}\in E_{\ref{step:F12}}$, where  $s\in F_{11}$ with $f(s)=1$ and there is some $t\in N(s)\cap F_1$ such that $f(t)$ is integer with $f(s)<f(t)$, otherwise, $\mu(s)\le5$.
Indeed, if all neighbors of $s$ in $F_1$ have $f$-value one,
\ref{item:F12224} gives $\mu(s)\le4$.
Otherwise, a neighbor $t$ with nonintegral $f(t)>1$ has
$f(t)\in\{1.5,2.5\}$, and \ref{item:F121} and \ref{item:F123}
give $\mu(s)\le\mu(t)+1\le5$.
Both bounds contradict $\mu(s)\in\{6,7\}$. 
Then $t\notin L_{21}$ since all the edges in $[L_{21},L_1]$ has been oriented in \Cref{step:F0X31X41}. Since $su$ is oriented in \Cref{step:F12}, we know that $ts$ is not oriented in \Cref{item:step:F0X31X414},  and so $t\in F_{11}$.   Then, we may further assume that $t\to I_{f(t)}$, otherwise, $\mu(s)\le 5$, a contradiction.  Since $t\to I_{f(t)}$ and $t\in F_{11}$,  there is $r\in F_1$ such that $f(r)>f(t)$ by \ref{item:F123}. Since the existence of $r$ and $\mu(s)\in \{6,7\}$,  we have $t\in L_{22}$ with   $f(t)=2$ and $f(r)=3$. Since $f(t)<f(r)$, if edges in $[r,X_2]$ are oriented in \Cref{step:F12}, then $X_2\to r$ by \ref{item:F12221}. Consequently, $\mu(s)\le \mu(r)+e(r\to t\to s)\le 5$, a contradiction. Thus, $r\in L_{21}$. Then again $X_2\to r$, a contradiction. 
    

 To see (2), we only prove the case when $s\in L_{22} \cup (L_{23}\cap F_1)$. Let $s\in L_{22} \cup (L_{23}\cap F_1)$, then $\nu(s)\ge 2$, $\mu(s)\ge 3$. Since $\nu(s)+\mu(s)\le 9$ by \Cref{claimcenter}, we are done.

To see (3), for $s\in X_{32}$, by the definition of $X_{3}$ and  \Cref{claim1}, it holds that $\mu(s)\ge 3$ and $\nu(s)\ge 3$. By \Cref{claimcenter}, we have that $\nu(s)\leq6$ and $\mu(s)\leq 6$.
\end{proof}

\begin{corollary}
\label{prop:boundF1}
            If $s\in F_1\cap  L $, then $\nu(s)\leq6$ and $\mu(s)\leq7$. If $s\in F_1\cap  R $, then $\nu(s)\leq7$ and $\mu(s)\leq6$.    
\end{corollary}
\begin{proof}
  It follows  from  \Cref{claim3}, \Cref{prop:xinf0munu}, \Cref{claimcontrol1}, and  \Cref{calim:boundmunuV-GO1}.\end{proof}

Now,   we bound   $\U(s)$  or  $\D(s)$ for $s\in F_3$. For $s\in L_{32}\cap F_3$,   fix a  shortest path between $s$ and $R_{11}$ in $G$ and denote it  by $L_s=s v_1v_2\cdots$.   
Similarly, for $s\in R_{32}\cap F_3$ which is not adjacent to a vertex oriented in \Cref{step:L23R23}, fix a shortest path between $s$ and $L_{11}$ in $G$ and denote it by $R_s=s v_1v_2\cdots$.  

 \begin{proposition}\label{claim:X43F3bdd}
 The following hold in  $M^{O_{\ref{step:DUb}}}$:   
\begin{itemize}
    \item[(1)] If $s\in F_3\cap  (L_3\cup R_3) $, then  $\D(s)\leq8$ when $s\in U$;  $\U(s)\leq8$ when $s\in D$.  
      \item[(2)] If  $s\in (F_3\cup F_4) \cap X_{4}$, then $\D(s)\leq6$ when $s\in U$ and $\U(s)\leq6$ when $s\in D$. 
\end{itemize}
\end{proposition}
\begin{proof}

    For (1), the proof for  the case that $s\in F_3\cap L_3 $ is provided.  Since $s\in F_3\cap L_3$, no edge incident to $s$ is oriented in \Cref{step:L23R23}. Furthermore, $\mu(s),\nu(s)\ge 3$ with $\mu(x)+\nu(s)\le 9$ when $s\in V(G^{O_{\ref{step:sumatmost8}}})$.  Thus,  by \Cref{claimcenter},  assume that  $s\in B$. Moreover,  by   \Cref{claim:xinDyinU},    no edge incident to $s$ oriented in \Cref{step:DUb}.   Then  $s\in L_{32}\cap (U_b\setminus M_b)$ or $s\in L_{32}\cap (D_b\setminus M_b)$.

Suppose that there is some
$t\in N(s)\cap V(G^{O_{\ref{step:sumatmost8}}})$.
Since $s\in F_3$, it holds that
$\mathrm{dist}(t,F_0)\ge2$, and so $\mu(t),\nu(t)\ge2$.
Let $e$ be an oriented edge incident to $t$ in
$G^{O_{\ref{step:sumatmost8}}}$.
By the definition of $\omega$ and \Cref{claimcenter},
we have that $\mu(t)+\nu(t)\leq\omega(e)\leq9$ in
$G^{O_{\ref{step:sumatmost8}}}$.
It follows that $\mu(t)\leq7$ and $\nu(t)\leq7$.
These bounds also hold in $G^{O_{\ref{step:DUb}}}$.

Recall that all the edges incident to $s$ are unoriented
after \Cref{step:DUb}.
If $s\in U_b\setminus M_b$, then $t\in U$, otherwise,
$st$ would have been oriented in \Cref{step:DUb}.
By \Cref{lem:good-path-gluing}, it holds that
$\D(s)\leq\D(t)+1\leq\nu(t)+1\leq8$.
Similarly, if $s\in D_b\setminus M_b$, then $t\in D$,
and $\U(s)\leq\U(t)+1\leq\mu(t)+1\leq8$.
Thus, assume that
$N(s)\cap V(G^{O_{\ref{step:sumatmost8}}})=\emptyset$.

By \Cref{claimcenter2}, $s$ satisfies (b1) or (b2)
after \Cref{step:L23R23}.
Hence, there is some
$t\in N(s)\cap
(V(G^{O_{\ref{step:L23R23}}})\setminus
V(G^{O_{\ref{step:sumatmost8}}}))$.
All the edges oriented in \Cref{step:L23R23} lie in
$[L_2,L_1]\cup[R_2,R_1]$.
Since $s\in L_3$ has no neighbor in
$L_1\cup R_1\cup R_2$, it follows that $t\in L_2$
and there is some $r\in N(t)\cap L_1$ such that
$tr$ is oriented in \Cref{step:L23R23}.
We continue by considering a shortest $(s,R_{11})$-path
in $G$.

     

Let $L_s=sv_1v_2\cdots v_k$ be a shortest
$(s,R_{11})$-path in $G$. Then $v_1\in F_2\cup F_3$ and $k\in\{3,4\}$. Moreover, $v_1\notin V(G^{O_{\ref{step:sumatmost8}}})$ by the above assumption.  If $k=3$, then $L_s$ must be of form
$s(L_{21}\cup X_{31})R_{21}R_{11}$. Thus, $v_2\to v_1$ is oriented in \Cref{step:F0X31X41}, which implies that $v_1\in V(G^{O_{\ref{step:F0X31X41}}})$, a contradiction. Hence, $k=4$.  Since $sv_1$ is unoriented after \Cref{step:DUb}, we have that $v_1\in U$ when $s\in U_b\setminus M_b$, and $v_1\in D$ when $s\in D_b\setminus M_b$. We show that the following claim holds.

     \begin{claim}\label{cl:F3UD}
If  $v_2$ has  $\mu(v_2)\le 6$ and $\nu(v_2)\le 6$ in $G^{O_{\ref{step:DUb}}}$, then the result holds. 
     \end{claim}
     \begin{proofofclaim} 
We first show that $v_1v_2$ could not be oriented in  \Cref{step:L23R23}.  Suppose not that $v_1v_2$ is oriented in  \Cref{step:L23R23}. Then $v_1\in L_{23}\cap F_2$ and $v_2\in L_{12}\cup L_{13}$. However, $N(v_2)\cap N(R_{11})=\emptyset$ by the definition of $L_{12}$ and $R_{11}$, which implies that $\mathrm{dist}(v_2,R_{11})\ge 3$, a contradiction. 

Suppose that $s\in U_b\setminus D_b$. Then $v_1\in U$, otherwise, $\orw{sv_1}\in E(G^{O_{\ref{step:DUb}}})$.   If $v_1v_2$ is unoriented after \Cref{step:DUb}, then $v_2\in U$ since $v_1\in U$.  
By \Cref{lem:good-path-gluing}, it holds that
$\D(s)\le\D(v_1)+1\le\D(v_2)+2\le\nu(v_2)+2\le8$,
where $\nu$ is computed in $G^{O_{\ref{step:DUb}}}$. 

Hence, assume that  $v_1v_2$ is  oriented in \Cref{step:F0X31X41}--\Cref{step:sumatmost8} or \Cref{step:DUb}, then $\nu(v_1)\le 7$ in $G^{O_{\ref{step:DUb}}}$ since $v_1\in (F_2\cup F_3)\cap U$.  
By \Cref{lem:good-path-gluing}, it holds that
$\D(s)\le\D(v_1)+1\le\nu(v_1)+1\le8$. 
 The same argument works for $s\in D_b\setminus U_b$.
\end{proofofclaim}

     By \Cref{claim3} and $v_1\notin V(G^{O_{\ref{step:sumatmost8}}})\cup F_1$, it holds that $v_1\notin L_{21}\cup L_{22}\cup L_{31}\cup X_3\cup X_{41}$ and so    
     $v_1\in (L_{32}\cup X_{42}\cup L_{23})\cap (F_2\cup F_3)$.  We proceed by considering the exact location of $v_1$.

First, suppose that  $v_1\in L_{32}\cap F_3$. Then $v_2\in F_2$ and $\mathrm{dist}(v_2,R_{11})=2$. Then $v_2\notin X_{32}\cup L_{22}\cup L_3\cup X_4\cup (L_{23}\cap F_2)$ since $X_{32}\cup L_{22}\subseteq F_1$ and $\mathrm{dist}(X_4\cup L_{3}\cup (L_{23}\cap F_2),R_{11})\ge 3$. Thus, $v_2\in  L_{21}\cup X_{31}$ and $v_3\in R_{21}\cap F_1$, which implies that $v_3\to v_2$ by \Cref{step:F0X31X41}. Then $\nu(v_2)\le 6$ and $\mu(v_2)\le 6$ with $v_1v_2$ not oriented in \Cref{step:L23R23} by \Cref{claimcenter}. Thus, by \Cref{cl:F3UD}, we are done.

Then, suppose that $v_1\in  L_{32}\cap F_2$. 
Since   $\mathrm{dist}(v_2,R_{11})=2$ and $v_1v_2\in E$, it holds that $v_2\notin L_3\cup R_3 \cup X_4$, which implies that $v_2\in L_2\cup X_3$. If  $v_2\in L_{21}\cup X_3$, then $\nu(v_2)\le 6$ and $\mu(v_2)\le 6$ by \Cref{prop:xinf0munu} and \Cref{claimcenter}.   


Thus, assume that $v_2\in L_2\setminus L_{21}$. Then $v_3\in L_{11}\cup X_2$, and so $v_2v_3$ is oriented in  \Cref{step:F0X31X41}--\Cref{step:F13}. In the following argument, the values of $\mu$, $\nu$ and $\omega$ are computed in
$G^{O_{\ref{step:sumatmost8}}}$.  Let $e$ denote the oriented edge corresponding to $v_2v_3$. By the definition of $\omega$ and \Cref{claimcenter}, it holds that $\mu(v_2)+\nu(v_2)\leq\omega(e)\leq9$. Since $v_2\in L_2$, we have that $\mu(v_2),\nu(v_2)\ge2$. If $\mu(v_2)+\nu(v_2)\leq8$, then $\mu(v_2)\leq6$ and $\nu(v_2)\leq6$, and we are done by \Cref{cl:F3UD}.   
Hence, assume that $\mu(v_2)+\nu(v_2)=9$. Then $\omega(e)=9$. By \Cref{claimcenter}, there is a dicycle $C$ of length 5 containing $e$ in  $G^{O_{\ref{step:sumatmost8}}}$.  We claim that $\mu(v_3),\nu(v_3)\leq2$. If $v_3\in X_2$, this follows from \Cref{obs:munu}. If $v_3\in L_{11}$, since $v_4\in R_{11}$ and $v_3v_4\in E$, the path $v\to v_4\to v_3\to u$ is a dipath in $G^{O_{\ref{step:F0X31X41}}}$. Thus, $\mu(v_3)\leq2$ and $\nu(v_3)\leq1$. The claim follows.   Since $C$ contains $v_2v_3$, the two directed subpaths  of $C$ between $v_2$ and $v_3$ have lengths 1 and 4,  in some order. It follows that $\mu(v_2)\leq\mu(v_3)+4\leq6$ and  $\nu(v_2)\leq4+\nu(v_3)\leq6$. These bounds also hold in  $G^{O_{\ref{step:DUb}}}$. Therefore, by \Cref{cl:F3UD}, we are done.

Now, suppose that $v_1\in X_{42}$.  Since $s\in L$ and $sv_1\in E$, it follows that  $v_2\in X_{31}$ and $v_3\in X_2\cup R_{21}$ or $v_2\in X_{32}$ and $v_3\in X_2$. By \Cref{claimcenter}, it holds that $\mu(v_2)\le 6$ and $\nu(v_2)\le 6$, again, we are done.

Finally, suppose that  $v_1\in F_2\cap   L_{23} $. 
Since   $\mathrm{dist}(v_2,R_{11})=2$,  $v_1v_2\in E$ and the definition of $L_{11}, L_{12}, L_{13}$, it holds that $v_2\notin L_1\cup  L_3\cup R_3\cup X_3\cup X_4$, which implies that $v_2\in L_2$. By the same argument as in the case 
$v_1\in F_2\cap L_{32}$, we are done.

Therefore, (1) holds.

\medskip 

  For (2), let $s\in (F_3\cup F_4)\cap X_4$. It is enough to assume that $s\in X_{42}$ and $N(s)\cap X_{31}\neq\emptyset$. Let $t\in N(s)\cap X_{31}$. Then  $t\notin F_1$. Thus,  $N(t)\cap L_{21}\neq\emptyset$ and $N(t)\cap R_{21}\neq\emptyset$, which implies that $\nu(t),\mu(t)\le 3$. Since $st$ could not be oriented in \Cref{step:L23R23}, it holds that either $\mu(s)+\nu(s)\le 9$ or $s\in D_b$ or $s\in U_b$. Hence, if $st$ is oriented in \Cref{step:DUb}, then $\mu(s)\le 6$ when $s\in D_b$ and $\nu(s)\le 6$ when $s\in U_b$. If $st$ is unoriented after  \Cref{step:DUb}, then $\U(s)\le \mu(t)+e(st)\le 4$  when $s\in D_b$  and $\D(s)\le  \nu(t)+e(st)\le 4$, we are done. 
\end{proof} 
\begin{proposition}
\label{claim:munuF2}
  The following hold in  $M^{O_{\ref{step:DUb}}}$:   
\begin{itemize}
    \item[(1)] Suppose $s\in F_2\cap (L_2\cup L_3)$. Then  $\D(s)\leq 9$ when $s\in U$ with equality  possible only when $s$ satisfies the following \begin{itemize}
        \item[$\bullet$]  $s\in F_2\cap L_2\cap (U_a\setminus D_a)$,  
        \item[$\bullet$]  there is an edge $st$ with $t\to s$ being oriented in \Cref{step:L23R23}, and
        \item[$\bullet$] 
   there is $r\in N(s)\cap L_3\cap F_3\cap (U_b\setminus D_b)$ with $\D(r)=8$ and $rs$ is unoriented after \Cref{step:DUb}.   
    \end{itemize} 
Also,  $\U(s)\leq9$ when $s\in D$ with equality  possible only when $s$ satisfies the following \begin{itemize}
        \item[$\bullet$]  $s\in F_2\cap L_2\cap (D_a\setminus U_a)$,  
        \item[$\bullet$]  there is an edge $st$ with $s\to t$ being oriented in \Cref{step:L23R23}, and
        \item[$\bullet$] 
   there is $r\in N(s)\cap L_3\cap F_3\cap (D_b\setminus U_b)$ with $\U(r)=8$  and $rs$ is unoriented after \Cref{step:DUb}.      
    \end{itemize} 
By repalcing $L_2$ with $R_2$ and $L_3$ with $R_3$, the same result holds. 
    
     \item[(2)] If  $s\in (F_2\cup F_3) \cap X_3 $, then  $\nu(s)\leq3$ and $\mu(s)\leq3$.
      \item[(3)] If  $s\in F_2 \cap  X_{4} $, then  $\D(s)\leq7$ when $s\in U$ and $\U(s)\leq7$ when $s\in D$. 
\end{itemize}
\end{proposition}
\begin{proof}
For (1), we only prove the case for $s\in F_2\cap (L_2\cup L_3)$.  If some edge incident to  $s$ is oriented in \Cref{step:F0X31X41}--\Cref{step:sumatmost8} or \Cref{step:DUb}, then $\nu(s)\le 7$ and $\mu(s)\le 7$. Thus, assume that no such edge exists. 

Now, suppose that some edge incident to  $s$ is oriented in \Cref{step:L23R23}. Then   $N(s)\cap L_3\cap F_3\neq \emptyset$ by \Cref{step:L23R23}. Moreover,   $\mu(s)\le 4$ or $\nu(s)\le 3$. If both of them are at most 4, we are done. Hence,   assume that   $s\in U_a\setminus D_a$ or $s\in D_a\setminus U_a$. Let $r\in N(s)\cap L_3\cap F_3$. 
If $rs$ is oriented in \Cref{step:DUb}, then $\mu(s)\le 7$ and $\nu(s)\le 4$ or $\nu(s)\le 7$ and $\mu(s)\le 4$. Thus, assume that $rs$ is unoriented after \Cref{step:DUb}. Then $r\in U$ when $s\in U_a\setminus D_a$ and $r\in D$ when $s\in D_a\setminus U_a$.   By \Cref{claim:X43F3bdd}, we know that $\D(s)\le \D(r)+1\le 9$ when $s\in U_a$ or   $\U(s)\le \U(r)+1\le 9$ when $s\in D_a$. 
Equality nine requires the corresponding potential of $r$ to be eight.
Since $r\in F_3$ is untouched by \Cref{step:L23R23},
if $r\in A$, then $r\in V(G^{O_{\ref{step:sumatmost8}}})$
and $\mu(r),\nu(r)\le6$ by \Cref{claimcenter}.
Also, $r\in M_b$ gives $\mu(r),\nu(r)\le5$
by \Cref{claim:xinDyinU}.
Thus, $r$ lies in the stated exclusive part of $B$.
Moreover, $s\in L_2$, and its incident edge from
\Cref{step:L23R23} points into $s$ in the $U$-case
and out of $s$ in the $D$-case; the opposite direction
would give the corresponding distance parameter at most four. 

Therefore, assume that $st$ is unoriented after \Cref{step:DUb} for each $t\in N(s)\cap F_1$.   Let $t\in N(s)\cap F_1$. 
By \Cref{prop:boundF1}, and by \Cref{claimcenter} when $t\in X_3$,
we have $\mu(t),\nu(t)\le7$ in $G^{O_{\ref{step:DUb}}}$.
Since $st$ is unoriented, \Cref{claim:xinDyinU}(4)
and \Cref{lem:good-path-gluing} give the required potential at most eight.  Hence, the result holds.   

For (2), since $s\notin F_1$, it holds that  $N(s)\cap L_{21}\neq\emptyset$ and $N(s)\cap R_{21}\neq\emptyset$ and so $\mu(s),\nu(s)\le 3$ by \Cref{step:F0X31X41}. 

 For (3),   we may assume that $s\in X_{42}$ since  $\mu(s),\nu(s)\le4$ for $s\in X_{41}$. If  some edge incident to  $s$ is oriented in \Cref{step:cycle(a)} or \Cref{step:sumatmost8} or \Cref{step:DUb}, then  $\nu(s),\mu(s)\le 7$ and so we are done. Therefore, assume that all the edges incident to  $s$ are unoriented  after \Cref{step:DUb}.  Let $t\in N(s)\cap F_1$. Then  $t\in X_3\cap F_1$.  By \Cref{claimcenter}, it holds that $\mu(t)\le 6$ and $\nu(t)\le 6$ since $t\in X_3$.   Since $st$ is unoirented after \Cref{step:DUb}, there is a desired good  path $P$ for $s$ with $\ell(P)\le 7$, and so we are done. 
\end{proof}

 \begin{corollary}
  For each $s\in U_a\cup U_b$, it holds that $\D(s)\le 9$ is finite.  For each $s\in  D_a\cup D_b$, it holds that $\U(s)\le 9$ is finite. Equality nine is possible only at vertices of $A$, with
the corresponding equality conditions in
\Cref{claim:munuF2}(1).
 \end{corollary}
\begin{proof}
    It follows from \Cref{claimcontrol1}, \Cref{calim:boundmunuV-GO1}, \Cref{claim:X43F3bdd} and \Cref{claim:munuF2}. 
\end{proof}

\subsection{The final orientation step}\label{sec:final-orientation}

Now, we state the final construction step which  deals with all the remaining unoriented edges. Take the edges in $[D_b,D_a]\cup [D_a,D_a]\cup [D_b,D_b]$ for example.  The target is that after the final orientation, each  $x\in D_b$ has $\nu(x)\le 6$ and each   $x\in D_a\cup D_b$ has $\mu(x)\le9$.  To have $\mu(x)\le 9$, we always want  with priority  \begin{itemize}
\item $\mu(x)\le \U(x) $;
\item if $\mu(x)\le \U(x) $  fails at some vertices, for those we want that 
$\mu(x)\le \U(x)+1$ and there is a characterization of all such vertices. 
\end{itemize}
This makes the edges $xy$ with $x\in D_b$ and $y\in D_a$ hard to orient. 
We emphasis here all such  $xy$ will be oriented according to the values of $\U(x)$ in decreasing order; that is, first orient all such $xy$ with largest value of $\U(x)$, then those corresponding to the second largest value, and so on. We now explain how the orientation works and  will not care about the proof $\U(x)\le 9$. 

Now, if $\U(x)\le \U(y)$, by orienting $x\to y$, directly $\nu(x)\le 6$ and if we have  $\mu(x)=\U(x)$ then $\mu(y)\le \U(y)+1$. The hard case is that $\U(x)> \U(y)$. To make $\mu(x)=\U(x)$, the edge $xy$ should be oriented as $y\to x$, but it may lead to that $\nu(x)>6$ for which  $x\to y$ is desired.  If now    $x$ has a neighbor $x'$ in $D_b$ with $\U(x)>\U(x')$, then let $x'\to x$ to guarantee that $\mu(x)=\U(x)$ and so let $x\to y$ such that $\nu(x)\le 6$. 
Thus, assume that $\U(x')\ge \U(x)$ for each $x'\in N(x)\cap D_b$. If now there is $y'\in N(x)\cap D_a$ with $x\to y'$ which ensures  $\nu(x)\le 6$, then again let $y\to x$ to ensure $\mu(x)=\U(x)$.  
Note that now if there is some $x'\in N(x)\cap D_b$ with $\U(x')>\U(x)$, by the assumption, edges in $[x',D_a]$ are oriented earlier than $[x,D_a]$, which means that $x\to x'$ is enough since $x'\to y'$ for some $y'\in D_a$ since it ensure $\nu(x)\le 6$. 
Thus, letting  $y\to x$ makes $\mu(x)=\U(x)$ possible. Then we will  encounter the   scenario for which we have $\mu(x)=\U(x)+1$, that $\U(x)=\U(x')$ for each $x'\in N(x)\cap D_b$. Suppose that there is some $z\in N(x)\cap D_b$ such that all the edges in  $[N[z]\cap D_b,D_a]\setminus\{xy\}$ have been oriented as $ D_a\to N[z]\cap D_b$, then to let $\nu(z)\le 6$, we must have $z\to x\to y$ which would lead to $\mu(x)=\U(x)+1$. The positive aspect is that the existence of $z$ could be seen as a characterization of $x$ for which  $\mu(x)=\U(x)+1$. 

The detailed orientation and related proofs are listed \Cref{step:complicate},   \Cref{cl:complicateDmu} and \Cref{claimcenter4}.

\begin{step}\label{step:complicate}
 
\begin{enumerate}[label=(S\thestep.\arabic*)] 
The following edges are oriented in order.   

\item\label{S8-1} Let $E_k=\{xy:x\in D_b, y\in D_a, \U(x)=k\text{ and } xy \text{ unoriented}\}$.  For each finite $k$, the following edges related to $E_k$ are oriented in order. For each fixed $k$, first apply
\ref{S8-11}--\ref{S8-14} in order.
Then repeatedly apply \ref{S8-15} and
\ref{S8-16}, with priority given to \ref{S8-15}.
After each application, update the current mixed
graph and check \ref{S8-15} again for the same $k$.
Proceed to $E_{k-1}$ only when $E_k=\emptyset$.

The same procedure is used for
\ref{S8-25} and \ref{S8-26}, after applying
\ref{S8-21}--\ref{S8-24} in order.
Throughout this procedure, $E_k$ denotes the
remaining unoriented edges in the current layer.

For all $xy\in E_k$ with $x\in D_b$ and $y\in D_a$,   
\begin{enumerate}[label=(S10.1.\arabic*)] 
  
    \item\label{S8-11}  suppose $\U(x)\le \U(y)$, then let $x\to y$;   
    \item\label{S8-12} suppose $\U(x)>\U(y)$, and  \begin{itemize}
        \item $|N(x)\cap D_a|\ge 2$, or 
        \item $N(x)\cap D_b\neq \emptyset$  with $\U(x')<\U(x)$ for some $x'\in N(x)\cap D_b$,  
    \end{itemize}  
 moreover, there is no $y'\in N(x)\cap D_a$ such that $x\to y'$, then let $x\to y$; 
    \item\label{S8-13} suppose $\U(x)>\U(y)$, 
     and  \begin{itemize}
        \item $|N(x)\cap D_a|\ge 2$, or 
        \item $N(x)\cap D_b\neq \emptyset$  with $\U(x')<\U(x)$ for some $x'\in N(x)\cap D_b$,  
       \end{itemize}       
  moreover, there is $y'\in N(x)\cap D_a$ such that $x\to y'$ has been oriented, then let $y\to x$; 
    \item\label{S8-14} suppose $\U(x)>\U(y)$, $\U(x')\ge \U(x)$ for each  $x'\in N(x)\cap D_b$ and  $\U(x')>\U(x)$ for some such $x'$, then let $y\to x$; 

  \item\label{S8-15} suppose $\U(x)>\U(y)$,
$\U(x')=\U(x)$ for all $x'\in N(x)\cap D_b$, and there is some
$z_x\in N[x]\cap D_b$ such that all the edges in
$[N[z_x]\cap D_b,D_a]\setminus\{xy\}$ have been oriented as
$D_a\to N[z_x]\cap D_b$.
For every such witness distinct from $x$, orient $z_x\to x$.
If $x$ itself is a witness, choose some $z\in N(x)\cap D_b$ and
orient $z\to x$ as well.
Then orient $x\to y$.
The existence of this choice and the fact that the prescribed edges
are unoriented are proved in \Cref{lem:exceptional-vertices}.
For each such $x$, fix one of the chosen neighbors and denote it by
$\hat{x}$.
   \item\label{S8-16} Suppose that $E_k\neq\emptyset$ and no edge in
$E_k$ satisfies the conditions in \ref{S8-15}.
Choose one unoriented edge $xy\in E_k$ and
let $y\to x$.
\end{enumerate}

\item\label{S8-2} Let $E_k=\{xy:x\in U_b, y\in U_a, \D(x)=k\text{ and } xy \text{ unoriented}\}$.  For each finite $k$, the following edges related to $E_k$ are oriented in order. For all $xy\in E_k$ with $x\in U_b$ and $y\in U_a$,   
\begin{enumerate}[label=(S10.2.\arabic*)] 
  \item\label{S8-21}  suppose $\D(x)\le \D(y)$, then let $ y\to x$;   
    \item\label{S8-22} suppose $\D(x)>\D(y)$ and 
    \begin{itemize}
        \item  $|N(x)\cap U_a|\ge 2$, or 
        \item $N(x)\cap U_b\neq \emptyset$ and there is some $x'\in N(x)\cap U_b$ with $\D(x')<\D(x)$,  
    \end{itemize}  
moreover, there is no $y'\in N(x)\cap U_a$ such that $ y'\to x$,   then let $ y\to x$; 
    \item\label{S8-23} suppose $\D(x)>\D(y)$ and  \begin{itemize}
        \item  $|N(x)\cap U_a|\ge 2$, or 
        \item $N(x)\cap U_b\neq \emptyset$ and there is some $x'\in N(x)\cap U_b$ with $\D(x')<\D(x)$,  
    \end{itemize}  
moreover, there is  $y'\in N(x)\cap U_a$ such that $ y'\to x$ has been oriented, then let $ x\to y$; 

    \item\label{S8-24} suppose $\D(x)>\D(y)$, $\D(x')\ge \D(x)$ for each  $x'\in N(x)\cap U_b$ and  $\D(x')>\D(x)$ for some such $x'$, then let $  x\to y$; 

  \item\label{S8-25} suppose  $\D(x)>\D(y)$,
$\D(x')=\D(x)$ for all $x'\in N(x)\cap U_b$, and there is some
$z_x\in N[x]\cap U_b$ such that all the edges in
$[N[z_x]\cap U_b,U_a]\setminus\{xy\}$ have been oriented as
$N[z_x]\cap U_b\to U_a$.
For every such witness distinct from $x$, orient $x\to z_x$.
If $x$ itself is a witness, choose some $z\in N(x)\cap U_b$ and
orient $x\to z$ as well.
Then orient $y\to x$.
For each such $x$, fix one of the chosen neighbors and denote it by
$\hat{x}$.

   \item\label{S8-26}  Suppose that $E_k\neq\emptyset$ and no edge in
$E_k$ satisfies the conditions in \ref{S8-25}.
Choose one unoriented edge $xy\in E_k$ and
let $x\to y$. 
\end{enumerate}
\item\label{S8-3} The following edges are oriented in order.\begin{enumerate}[label=(S10.3.\arabic*)]
    \item\label{S8-31} For  unoriented $x_1x_2\in [D_a,D_a]\cup [D_b,D_b]$ with $\U(x_1)\neq  \U(x_2)$, if $\U(x_1)> \U(x_2)$, let  $x_2\to x_1$; if $\U(x_1)<\U(x_2)$, let  $x_1\to x_2$.
    \item\label{S8-32}  For unoriented $\hat{p}z\in [D_b,D_b]$ with $\U(z)=\U(\hat{p})$ and $z\neq \hat{q}$ for any $p$ and $q$  playing the role of $x$ in \ref{S8-15}, let $z\to \hat{p}$. 
     \item\label{S8-33} For unoriented $x_1x_2\in   [D_b,D_b]$ satisfying  $\U(x_1)=\U(x_2)$ and there is $y\in N(x_2)\cap D_a$ such that $x_2\to y$, let $x_1\to x_2$.  
      \item\label{S8-34} For unoriented $x_1x_2\in [D_a,D_a]\cup [D_b,D_b]$ and $\U(x_1)=\U(x_2)$, let $x_1\to x_2$. 
\end{enumerate}

\item\label{S8-4} The following edges are oriented in order.
\begin{enumerate}[label=(S10.4.\arabic*)]
    \item\label{S8-41} For   unoriented  $x_1x_2\in E(U_a)\cup E(U_b)$  with $\D(x_1)\neq  \D(x_2)$, if $\D(x_1)> \D(x_2)$, let  $x_1\to x_2$; if $\D(x_1)< \D(x_2)$, let $x_2\to x_1$.
    \item\label{S8-42}  For unoriented $\hat{p}z\in [U_b,U_b]$ with $\D(z)=\D(\hat{p})$ and $z\neq \hat{q}$ for  for any $p$ and $q$  playing the role of $x$ in \ref{S8-25}, let $\hat{p}\to z$. 
     \item\label{S8-43} For unoriented $x_1x_2\in   [U_b,U_b]$ satisfying  $\D(x_1)=\D(x_2)$ and there is $y\in N(x_2)\cap U_a$ such that $  y\to x_2$, let $x_2\to x_1$.  
    \item\label{S8-44} For unoriented $x_1x_2\in [U_a,U_a]\cup [U_b,U_b]$ and $\D(x_1)=\D(x_2)$, let $x_1\to x_2$. 
\end{enumerate}
\end{enumerate}
\end{step}

\begin{remark}\label{rm:step834}
    The orientation  on an edge $x_1x_2$ in \ref{S8-34} and \ref{S8-44} means that both $x_1\to x_2$ and $x_2\to x_1$ work. We just arbitrarily choose one. 
\end{remark}
Let $T_D$ and $T_U$ be the sets of vertices playing the role of $x$
in \ref{S8-15} and \ref{S8-25}, respectively.  Write $H_D=\{\hat p:p\in T_D\}$ and $H_U=\{\hat p:p\in T_U\}$.
The exclusion in \ref{S8-32} means $z\notin H_D$, and that in
\ref{S8-42} means $z\notin H_U$.

\begin{lemma}\label{lem:exceptional-vertices}
For each $x\in T_D$ with $\U(x)=k$, all the following hold.
\begin{itemize}
\item $x\in D_b\setminus U_b$, no edge incident to $x$ is oriented
in $G^{O_{\ref{step:DUb}}}$, and $N(x)\cap U=\emptyset$.
\item There is a unique vertex $y\in N(x)\cap D_a$.
Moreover, $\U(y)=k-1$, and $\U(r)=k$ for every
$r\in N(x)\cap D_b$.
\item $\hat{x}\in D_b\setminus U_b$, $\U(\hat{x})=k$ and
$\hat{x}\notin T_D$.
All its $D_a$-edges are oriented towards $\hat{x}$.
There is some $w\in N(\hat{x})\cap D_a$ with $\U(w)=k-1$
such that $w\to\hat{x}\to x\to y$ in the final orientation.
\end{itemize}
Consequently, if $r,s\in D$ with $rs\in E$ and
$\U(r)<\U(s)$, then $r\notin T_D$.
The symmetric assertions hold for $T_U$.
\end{lemma}
\begin{proof}
     
We prove the assertions for $T_D$.
Let $H=G^{O_{\ref{step:DUb}}}$, and consider the moment
when $x$ is selected in \ref{S8-15} for an unoriented edge $xy$.
By \ref{S7-1}, all edges in $[M_b,A]$ are oriented in $H$,
so $x\notin M_b$.
If some edge incident to $x$ were oriented in
\Cref{step:DUb}, then all its $D_a$-edges would already
be oriented by \Cref{claim:xinDyinU}(2), a contradiction.
Since no edge incident to a vertex of $B$ is oriented before
that step, no edge incident to $x$ is oriented in $H$.
Thus, $x\in D_b\setminus U_b$, and
\Cref{claim:xinDyinU}(4) gives $N(x)\cap U=\emptyset$.

The witness condition requires all edges in
$[x,D_a]\setminus\{xy\}$ to be oriented towards $x$.
If some $y'\in N(x)\cap D_a$ satisfies $\U(y')\ge k$,
then \ref{S8-11} would orient $x\to y'$.
For $y'=y$, this contradicts the fact that $xy$ is unoriented;
otherwise, it contradicts the witness condition.
Thus, $\U(y')<k$ for every $y'\in N(x)\cap D_a$.
If $|N(x)\cap D_a|\ge2$, then \ref{S8-12} would give
an outgoing $D_a$-edge before \ref{S8-15}, again a contradiction.
Hence, $N(x)\cap D_a=\{y\}$.
Since $xy$ is unoriented in $H$, by
\Cref{lem:good-path-gluing}, $\U(y)=k-1$.
The equality $\U(r)=k$ for every $r\in N(x)\cap D_b$
is a condition of \ref{S8-15}. 

For each witness $z_x\ne x$, all its $D_a$-edges
are already oriented towards $z_x$.
If $x$ itself is a witness, then
$N(x)\cap D_a=\{y\}$, $N(x)\cap U=\emptyset$ and
$d_G(x)\ge2$ give $N(x)\cap D_b\ne\emptyset$.
Choose $z\in N(x)\cap D_b$.
Since $[z,D_a]\subseteq
[N[x]\cap D_b,D_a]\setminus\{xy\}$,
all the edges in $[z,D_a]$ are already oriented
towards $z$.

Let $z$ be any of the neighbors chosen in \ref{S8-15}.
Then $\U(z)=k$.
Since $N(z)\cap D_a\ne\emptyset$ and all these edges
are oriented towards $z$, \Cref{claim:xinDyinU} gives
$z\in D_b\setminus U_b$ and no edge incident to $z$
is oriented in $H$.
Thus, $N(z)\cap U=\emptyset$.
Choose $w\in N(z)\cap D_a$.
If $\U(w)\ge k$, then \ref{S8-11} would give $z\to w$,
a contradiction.
Since $zw$ is unoriented in $H$,
\Cref{lem:good-path-gluing} gives $\U(w)=k-1$.

Neither $x$ nor $z$ has been selected earlier in
\ref{S8-15}: such a selection would have oriented
$xy$ or given an edge from $z$ to $D_a$, respectively.
Before the present application, only \ref{S8-15}
can have oriented an edge in $[D_b,D_b]$ after
\Cref{step:DUb}, and each such edge is incident to
the vertex selected in that application.
Thus, $zx$ is still unoriented, and the rule gives
$z\to x\to y$.
Moreover, $z$ cannot be selected later, since all
its $D_a$-edges are already oriented.
Hence, $z\notin T_D$.
Taking $z=\hat{x}$ proves the third assertion.

Every neighbor in $D$ of a vertex of $T_D$ has
$\U$-value at most that of the vertex.
This proves the last assertion.

For $T_U$, the same argument first shows that the
selected vertex and its chosen neighbors have no
neighbor in $D$.
Thus, their incident edges are not affected by
\ref{S8-1}, and the proof with $\D$ and reversed
directions is symmetric.
\end{proof}
\begin{lemma}\label{lem:overlap-potentials}
For each $x\in M_a\cup M_b$, it holds that
$\mu(x)=\U(x)$ and $\nu(x)=\D(x)$ in
$G^{O_{\ref{step:DUb}}}$.
\end{lemma}
\begin{proof}
Let $H=G^{O_{\ref{step:sumatmost8}}}$, and compute $\mu,\nu$
after \Cref{step:DUb}.
By \Cref{step:L23R23} and (\ref{partitionA}),
$(A\setminus V(H))\cap A'\subseteq A_1\cup A_2$.
Thus, $M_a\subseteq V(H)$, and the assertion for $M_a$
follows from \Cref{lem:saturation-stability}.

Let $x\in M_b$.
By \Cref{claim:xinDyinU} and \Cref{lem:good-path-gluing},
$\U(x)\le\mu(x)\le5$.
Suppose that $\U(x)\le3$, and choose a good path $P$
of $x$ with $\ell(P)\le3$.
Let $x\to q$ be an arc with $q\in D_a$.

If $q\in V(H)$, $W=xq$ and $p=q$.  Then $\nu_H(q)\le4$ and $\mu_H(q)\ge2$:
otherwise, since $q\notin F_0$, we would have
$q\in R_1\cap A_1$, contrary to $q\in D_a$. 
If $q\notin V(H)$, then $q\in(L_2\cup R_2)\cap F_2$.
Since $\nu(q)$ was finite when $D_a$ was defined,
\Cref{step:L23R23} gives an arc $q\to c$ with
$c\in(L_1\cup R_1)\cap F_1$.
If $c\in L_1$, then $\nu_H(c)\le2$ and
$c\notin L_{11}$ gives $\mu_H(c)\ge3$.
If $c\in R_1$, the rule gives $\mu_H(c)\ge3$,
so $c\to v$ and $\nu_H(c)\le3$. 
Put $W=xqc$ and $p=c$. 
In either case, $e(W)+\nu_H(p)\le5$.

Since $x\to q$ is already oriented, $q$ cannot immediately
precede $x$ on $P$.
An earlier occurrence of $q$ would give
$\mu_H(q)\le\ell(P)-2\le1$ in the first case,
or contradict $q\in F_2$ in the second.
Similarly, an occurrence of $c$ would give
$\mu_H(c)\le\ell(P)-1\le2$.
Here we use \Cref{lem:saturation-stability}.
Thus, $PW$ is a path.
Let $h$ be its last vertex in $V(H)$ before $p$, and let
$Q$ be the $hp$-subpath.
Such a vertex $h$ exists, since $P$ starts at a vertex
of $\{u,v\}\subseteq V(H)$.
Since $V(W)\cap V(H)=\{p\}$ and $x\notin V(H)$,
the vertex $x$ is an internal vertex of $Q$.
By the choice of $h$, no internal vertex of $Q$ belongs
to $V(H)$.
Thus, every edge of $Q$ has an endpoint outside $V(H)$,
and hence is unoriented in $H$.

Since $P$ is a good path and $W$ is a dipath from $x$
to $p$, we may orient the unoriented edges of $PW$
along the path and extend the resulting partial
orientation to an orientation $\overrightarrow{G}$ of $G$.
By \Cref{lem:saturation-stability}, it holds that
$\mu_H(h)=\mu_{\overrightarrow{G}}(h)
\le\ell(P)+e(W)-e(Q)$.
Consequently,
$\mu_H(h)+e(Q)+\nu_H(p)
\le\ell(P)+e(W)+\nu_H(p)\le8$.
Thus, $Q$ could still be oriented from $h$ to $p$
with weight at most eight at the end of
\Cref{step:sumatmost8}, a contradiction.
Hence, $\U(x)\ge4$.

 If $\U(x)=5$, then $\U(x)\le\mu(x)\le5$ gives equality.
Thus, assume that $\U(x)=4$. 
Let $P$ be a good path of $x$ with $\ell(P)=4$.
If $P$ is a dipath, then $\mu(x)\le4$.
Otherwise, its predecessor $r$ belongs to $D$ and
$\U(r)\le\ell(P-x)=3$, so $r\notin M_b$.
By \Cref{claim:xinDyinU}, $rx$ is already oriented
as $r\to x$.
The rules in \Cref{step:DUb} then give $r\in M_a$.
Thus, $\mu(x)\le\mu(r)+1=\U(r)+1\le4$,
and consequently $\mu(x)=\U(x)$. 

The assertion for $\nu$ is symmetric.
\end{proof}
 The following two results state that the target could be achieved by the above  orientation procedure.

\begin{proposition}
    \label{cl:complicateDmu}
 In $G^{O_{\ref{step:complicate}}}$, 
  $\nu(s)\le 6$ for each $s\in D_a\cup D_b$ while
     $\mu(s)\le 6$ for each $ s\in U_a\cup  U_b$. 
\end{proposition}
\begin{proof}
We prove that $\nu(s)\le6$ for $s\in D$.
By the definition of $D_a$ and \Cref{step:DUb},
$\nu(s)\le5$ for
$s\in D_a\cup(D_b\cap V(G^{O_{\ref{step:DUb}}}))$.
Thus, assume that
$s\in D_b\setminus V(G^{O_{\ref{step:DUb}}})$
and there is no $t\in N(s)\cap D_a$ with $s\to t$.
It suffices to find a dipath of length two from $s$ to $D_a$.

By \ref{S8-11} and \ref{S8-12}, we have that
$N(s)\cap D_a=\{t\}$, $\U(t)<\U(s)$, and
$\U(r)\ge\U(s)$ for every $r\in N(s)\cap D_b$.
Since $N(s)\cap U=\emptyset$ and $d_G(s)\ge2$,
it follows that $N(s)\cap D_b\neq\emptyset$.
If $s=\hat{x}$ for some $x\in T_D$, then
\Cref{lem:exceptional-vertices} gives $s\to x\to y$
for some $y\in D_a$, and we are done.
Thus, assume that $s\neq\hat{x}$ for every $x\in T_D$.

Suppose that some $r\in N(s)\cap D_b$ satisfies
$\U(r)>\U(s)$.
By \Cref{step:DUb}, \ref{S8-11} and \ref{S8-12},
there is some $r'\in N(r)\cap D_a$ with $r\to r'$,
since $s$ is a $D_b$-neighbor of $r$ with $\U(s)<\U(r)$.
Moreover, $s\to r$ by \ref{S8-31}, since \ref{S8-15}
only orients edges between $D_b$-vertices with equal
$\U$-values.
Thus, $s\to r\to r'$, and we are done.

Hence, $\U(r)=k=\U(s)$ for every $r\in N(s)\cap D_b$.
We show that some such $r$ has an edge $r\to r'$
with $r'\in D_a$.
Otherwise, every edge in
$E_s=[N[s]\cap D_b,D_a]$ is oriented towards $D_b$.
Let $xy$, with $x\in D_b$ and $y\in D_a$, be the last
edge of $E_s$ to be oriented.
Since all the $D_b$-ends of these edges have $\U$-value $k$,
and $st$ is not oriented in \ref{S8-11}--\ref{S8-14},
the edge $xy$ is oriented in \ref{S8-16}.

Since no edge is oriented from $x$ to $D_a$,
\ref{S8-11} and \ref{S8-12} imply that $\U(y)<k$
and $\U(z)\ge k$ for every $z\in N(x)\cap D_b$.
Since $xy$ is not oriented in \ref{S8-14},
we have $\U(z)=k$ for every such $z$.
Immediately before $xy$ is oriented, all the edges in
$E_s\setminus\{xy\}$ are oriented towards $D_b$.
Thus, \ref{S8-15} applies with $z_x=s$ and gives
$x\to y$ before \ref{S8-16}, a contradiction.

Choose $r\in N(s)\cap D_b$ and $r'\in D_a$ with $r\to r'$.
Since $s$ has no edge directed from $s$ to $D_a$ and
$s\neq\hat{x}$ for every $x\in T_D$,
none of \ref{S8-15}, \ref{S8-32} and \ref{S8-33}
can orient $sr$ as $r\to s$.
Thus, $s\to r$ already holds, or follows from \ref{S8-33}.
Consequently, $\nu(s)\le2+\nu(r')\le6$.

For $s\in U_b\setminus V(G^{O_{\ref{step:DUb}}})$,
we have $N(s)\cap D=\emptyset$.
Thus, \ref{S8-1} and \ref{S8-3} do not orient any edge
incident to $s$, and the same proof with reversed
directions gives $\mu(s)\le6$ for $s\in U$.
\end{proof}

\begin{proposition}\label{claimcenter4}
In $G^{O_{\ref{step:complicate}}}$, all the following hold.
\begin{itemize}
\item For each $s\in D$, it holds that $\mu(s)\le\U(s)$ if
$s\notin T_D$, and $\mu(s)\le\U(s)+1$ if $s\in T_D$.
\item For each $s\in U$, it holds that $\nu(s)\le\D(s)$ if
$s\notin T_U$, and $\nu(s)\le\D(s)+1$ if $s\in T_U$.
\end{itemize}
\end{proposition}
\begin{proof}
We prove the first assertion by induction on $\U(s)$.
For $s\in M_a\cup M_b$, we have $s\notin T_D$, and the
assertion follows from \Cref{lem:overlap-potentials}.
The case $\U(s)=0$ is clear.
Thus, assume that $s\in D\setminus U$ and $\U(s)=k\ge1$.

Suppose that $s\in T_D$.
By \Cref{lem:exceptional-vertices}, there is some $w\in D_a$
with $\U(w)=k-1$ such that $w\to\hat{s}\to s$.
Since $w\notin T_D$, the induction hypothesis gives
$\mu(s)\le\mu(w)+2\le k+1$.
Thus, assume that $s\notin T_D$.

Let $P$ be a good path of $s$ with $\ell(P)=k$.
If $P$ is a dipath in $G^{O_{\ref{step:DUb}}}$, then
$\mu(s)\le k$.
Otherwise, let $r$ be the vertex immediately preceding
$s$ on $P$.
Since $P-s$ is a good path of $r$, by
\Cref{lem:good-path-gluing}, it holds that
$k-1\le\U(r)\le\ell(P)-1=k-1$.
By \Cref{lem:exceptional-vertices}, $r\notin T_D$.
Thus, the induction hypothesis gives $\mu(r)\le k-1$.
If $r\to s$, then $\mu(s)\le\mu(r)+1\le k$.

Thus, assume that $s\to r$.
Since $P$ is a good path, $rs$ is unoriented after
\Cref{step:DUb}, and hence $r,s\in D\setminus U$.
By \ref{S8-11}, \ref{S8-31} and
\Cref{lem:exceptional-vertices}, we have that
$r\in D_a$ and $s\in D_b$.
Since $s\notin T_D$, the edge $sr$ is oriented in
\ref{S8-12}.

We show that there is some $r'\in N(s)\cap D$ with
$\U(r')<k$ and $r'\to s$.
If some $r'\in N(s)\cap D_b$ satisfies $\U(r')<k$,
then $r'\to s$ by \ref{S8-31}, since \ref{S8-15}
only orients edges between $D_b$-vertices with equal
$\U$-values.
Otherwise, $|N(s)\cap D_a|\ge2$ by \ref{S8-12}.
All the edges in $[s,D_a]$ are unoriented after
\Cref{step:DUb}, since otherwise that step would also
have oriented $sr$.
Moreover, $\U(y)<k$ for every $y\in N(s)\cap D_a$;
otherwise, \ref{S8-11} would give $s\to y$, preventing
the application of \ref{S8-12}.
Thus, after $s\to r$, \ref{S8-13} gives $r'\to s$
for some $r'\in(N(s)\cap D_a)\setminus\{r\}$. 
In either case, $r'\notin T_D$ by
\Cref{lem:exceptional-vertices}.
By the induction hypothesis, it follows that
$\mu(s)\le\mu(r')+1\le\U(r')+1\le k$.

For the second assertion, vertices in $M_a\cup M_b$
are covered by \Cref{lem:overlap-potentials}.
By \Cref{step:DUb}, every edge incident to $U\setminus D$
that remains unoriented has both ends in $U\setminus D$.
Thus, these edges are not affected by \ref{S8-1} or
\ref{S8-3}, and the same proof with $\D$ and reversed
directions gives the assertion.
\end{proof}

\section{Proof of the diameter bound}\label{sec:finalmain}

We now show that the completed orientation $G^{O_{\ref{step:complicate}}}$ has directed diameter at most $16$, as required by \Cref{lemless5}. The argument reduces the problem to noncritical endpoints and then examines a shortest path between them, using the distance bounds already established.

Throughout this section, $\mu_{\ref{step:DUb}}$ and $\nu_{\ref{step:DUb}}$ denote the values in $G^{O_{\ref{step:DUb}}}$, while $\mu_{\ref{step:L23R23}}$ and $\nu_{\ref{step:L23R23}}$ denote the values in $G^{O_{\ref{step:L23R23}}}$ when this distinction is needed. The values of $\U$ and $\D$ remain fixed. Unless a different stage is specified, $\mu$, $\nu$ and $\partial$ are computed in $G^{O_{\ref{step:complicate}}}$. Adding arcs cannot increase a directed distance. In particular, an upper bound in an earlier stage also holds in a later stage.

\subsection{Edge types and special vertices}\label{sec:edge-types}
\begin{definition}

For $e\in E(G^{O_{\ref{step:DUb}}})$, if $e$ is oriented in \Cref{step:F0X31X41}--\Cref{step:sumatmost8}, or \Cref{step:DUb}, then call $e$   a  \emph{type-I} edge; if $e$ is oriented in \Cref{step:L23R23}, then call  $e$  a \emph{type-II} edge. 
\end{definition}

\begin{proposition}\label{prop:t1edge}
  If $s\to t$ is  a \emph{type-I} edge with $s,t\notin F_0$, then $\mu(s)\le 6$, $\nu(t)\le 6$, and $\mu(s)+\nu(t)\le 10$ in $G^{O_{\ref{step:DUb}}}$. 
 \end{proposition}
\begin{proof} Throughout this proof, $\mu$, $\nu$ and $\omega$ are computed in $G^{O_{\ref{step:DUb}}}$. The bounds then hold in every extension by monotonicity.
Suppose that $st$ is oriented in   \Cref{step:F0X31X41}--\Cref{step:sumatmost8}. Then  $\mu(s)+\nu(t)+1\le \omega(s\to t)\le 9$ by \Cref{claimcenter}. If $\omega(s\to t)\le8$, we are done. Thus, assume that $\omega(s\to t)=9$. Then $st$ lies on a dicycle of length 5 that intersects $I_2\cup I_3\cup I_4$. Since $s,t\notin F_0$, the result holds.

Suppose  $s\to t$  is oriented in \Cref{step:DUb}.  Then $\mu(s)+\nu(t)+1\le \omega(s\to t)\le 4+4+3\le 11$, $\mu(s)\le 6$ and $\nu(t)\le 6$ by \Cref{claim:xinDyinU}.
\end{proof}

Recall that by \Cref{claimcenter4}, for $s\in U$  it holds that $\nu(s)\le \D(s)+1$; for  $s\in D$  it holds that $\mu(s)\le \U(s)+1$. We propose the following definitions. 
\begin{definition}
  For $s\in U\setminus D$ with $\nu(s)=\D(s)+1$, or $s\in D\setminus U$ with $\mu(s)=\U(s)+1$, we call $s$ a {\it critical-I} vertex. Here the equalities refer to the final orientation. By \Cref{claimcenter4}, every such vertex belongs to $T_U$ or $T_D$, respectively. For each such $s$, fix the corresponding vertex $\hat{s}$ in \ref{S8-25} or \ref{S8-15}, and call it an {\it adherent-I} vertex.
\end{definition}

\begin{definition}
 For $s\in V$ such that the equality  in \Cref{claim:munuF2}(1) holds,   we call $s$ a  {\it critical-II} vertex and call the corresponding vertex $r$ in \Cref{claim:munuF2}(1) an   {\it adherent-II} vertex. 
\end{definition}

We list the following properties about  critical and adherent  vertices. 
\begin{proposition}\label{prop:criticaland adherent} All the following hold  in $G^{O_{\ref{step:complicate}}}$. 
     \begin{enumerate}[label=(CA\arabic*)]    \item\label{prop:critical-4}
If  $s\in D\setminus U$ has $\U(s)=9$ or  
  $s\in U\setminus D$ has $\D(s)=9$,  then $s$ is critical-II.    
       \item\label{prop:critical-1}  If  $s\in V$ is an  adherent vertex, then 
$s\in B\setminus \{t:t\text{ is adjacent to some edge oriented in  \Cref{step:DUb}}\}$ with   
$\D(s)\le 8$ when $s\in U_b\setminus D_b$ and $\U(s)\le 8$ when $s\in D_b\setminus U_b$. 
\item\label{prop:critical-3}
If $s$ is a critical vertex and $t$ is the corresponding adherent vertex, then $s\to t$ when $s\in U$ and $t\to s$ when $s\in D$. 
 \item\label{prop:critical-2} If $s \in D_b \setminus  U_b $ is adherent-I, then 
$\U(t)\ge \U(s)$  for every
$t\in N(s)\cap D_b$, and  $\U(s)>\U(t)$ where 
$\{t\}=N(s)\cap D_a$. If $s \in U_b \setminus  D_b $ is  adherent-I, then 
$\D(t)\ge \D(s)$  for every
$t \in N(s)\cap U_b$, and  $\D(s)>\D(t)$ where 
$\{t\}=N(s)\cap U_a$. 
\item \label{lem:gold-adherent-edge}
Let $s$ be adherent and $t\in N(s)$.
If $s\in U_b\setminus D_b$, $t\in U$ and $\D(t)<\D(s)$,
then $s\to t$ in the final orientation.
If $s\in D_b\setminus U_b$, $t\in D$ and $\U(t)<\U(s)$,
then $t\to s$ in the final orientation.

  \item\label{prop:cadisjoint} The set of critical vertices and the set of  adherent vertices are disjoint. 
    \end{enumerate}
\end{proposition}
\begin{proof}

To see \ref{prop:critical-4}, suppose that $s\in D\setminus U$ and $\U(s)=9$. By \Cref{prop:boundF1},   \Cref{claim:X43F3bdd} and \Cref{claim:munuF2},   the result holds.
 
To see \ref{prop:critical-1}, suppose first that $s$ is adherent-I.
By \Cref{lem:exceptional-vertices}, $s$ belongs to the same
exclusive $B$-side as its critical vertex.
By symmetry, assume that $s\in U_b\setminus D_b$.
Then no edge incident to $s$ is
 oriented in $G^{O_{\ref{step:DUb}}}$, by
\Cref{lem:exceptional-vertices}.
Moreover, $\D(s)\le8$, since the preceding potential bounds
allow equality nine only at vertices in $A$.

If $s$ is adherent-II, then by \Cref{claim:munuF2}(1)
and symmetry, assume that $s\in U_b\setminus D_b$
and $\D(s)=8$.
If an edge incident to $s$ were oriented in
\Cref{step:DUb}, then \Cref{claim:xinDyinU}and \Cref{lem:good-path-gluing}
would give $\D(s)\le6$, a contradiction.
No edge incident to a vertex of $B$ is oriented before
that step. Thus, the assertion follows.

To see \ref{prop:critical-3}, suppose $s\in  D_b\setminus U_b$ is critical-I. Then $t\to s$ by  \ref{S8-15}.
 Suppose $s\in D_a\setminus M_a$ is critical-II. Then $\U(s)=9$ and $\U(t)=8$ with $t\in D_b\setminus M_b$. By \ref{S8-11}, it holds that $t\to s$. The other way is similar. 
 
To see \ref{prop:critical-2}, by symmetry,
assume that $s\in D_b\setminus U_b$ is adherent-I.
By \Cref{lem:exceptional-vertices}, all its $D_a$-edges
are oriented towards $s$, and all were unoriented in
$G^{O_{\ref{step:DUb}}}$.
Thus, $\U(t)<\U(s)$ for every $t\in N(s)\cap D_a$,
as otherwise \ref{S8-11} would give $s\to t$.
If $|N(s)\cap D_a|\ge2$, or if some
$t\in N(s)\cap D_b$ satisfies $\U(t)<\U(s)$,
then \ref{S8-12} would give an outgoing $D_a$-edge,
a contradiction.
Since $N(s)\cap D_a\ne\emptyset$, the assertion follows. 

To see \ref{lem:gold-adherent-edge}, by symmetry,
assume that $s\in U_b\setminus D_b$. 
By \ref{prop:critical-1} and \Cref{claim:xinDyinU}(4),
all edges incident to $s$ are unoriented after
\Cref{step:DUb}, and all its neighbors belong to $U\setminus D$.
Thus, \ref{S8-1} and \ref{S8-3} do not orient these edges. 
If $t\in U_b$, then $s\to t$ by \ref{S8-41},
since $\D(s)>\D(t)$ and \ref{S8-25} only orients
edges between $U_b$-vertices with equal $\D$-values.

Thus, assume that $t\in U_a$.
If $s$ is adherent-I, then $s\to t$ by
\Cref{lem:exceptional-vertices}.
Otherwise, by \Cref{claim:munuF2}(1), there is some
$r\in N(s)\cap U_a$ with $\D(r)=9$ and $\D(s)=8$.
Since $\D(t)<8$, we have $r\ne t$.
By \ref{S8-21}, $r\to s$, so \ref{S8-22}
does not apply at $s$.
Then \ref{S8-23} gives $s\to t$, as required.

To see \ref{prop:cadisjoint}, since adherent vertices  are in $B$ by \ref{prop:critical-1} and critical-II vertices are in $A$, they are disjoint. It is sufficient to consider the case that $s\in D_b\setminus U_b$ is critical-I.  Then there is some $t\in N(s)\cap D_a$ with $\U(s)>\U(t)$ and $s\to t$ by \ref{S8-15} and so $s$ is not adherent-I.  Now, suppose that $s$ is adherent-II.  Then there is $r\in N(s)\cap D_a$ with $\U(r)>\U(s)$ by the definition of  adherent-II. Thus, $s\to r$ by \ref{S8-11}, which implies that $s$  does not meet the conditions in \ref{S8-15}.     
\end{proof}

\begin{corollary}\label{cor:noncritical-bounds}
    Suppose that $s\in V$ is not critical. Then $\mu(s)\le \U(s)\le 8$ and $\nu(s)\le 6$ when $s\in D$ and $\nu(s)\le \D(s)\le 8$ and $\mu(s)\le 6$ when $s\in U$ in $G^{O_{\ref{step:complicate}}}$. 
\end{corollary}
\begin{proof}
    It follows from \Cref{cl:complicateDmu}, \Cref{claimcenter4} and \Cref{prop:criticaland adherent}. 
\end{proof}
 
\subsection{Reductions and path classification}\label{sec:path-reductions}
 
The bounds in \Cref{cor:noncritical-bounds} lead to the following reduction: it suffices to control noncritical pairs on the two exclusive sides, with one unit saved for each adherent endpoint. We first justify this reduction, then dispose of paths of length at most three and classify the remaining four-edge paths according to their intersections with $F_0$. 

\begin{theorem}\label{prop:adhernettocritical}
Let $s\in U\setminus D$ and $t\in D\setminus U$ where $s,t$ are not critical vertices. Then  $\partial(s,t)\le 16-\alpha$ in $G^{O_{\ref{step:complicate}}}$, where $\alpha$ is the number of adherent vertices in $\{s,t\}$.
\end{theorem}
Assuming \Cref{prop:adhernettocritical} is true,  the proof of \Cref{lemless5} based on the construction of orientation $G^{O_{\ref{step:complicate}}}$ is as follows. 
\begin{proof} [Proof of \Cref{lemless5}]Let $s,t\in G^{O_{\ref{step:complicate}}}$. If $s\in D$ or $t\in U$, then $\nu(s)\le 6$ or $\mu(t)\le 6$, respectively, in $G^{O_{\ref{step:complicate}}}$ by \Cref{cl:complicateDmu}.  Note that by \ref{prop:critical-4} and \Cref{claimcenter4}, we have $\nu(s),\mu(s)\le 9$ for each $s\in V$. Then $\partial(s,t)\le 16$ by \Cref{lem:Gxy16}. Therefore, we may assume that $s\in U\setminus D$ and $t\in D\setminus U$. If neither $s$  nor  $t$ is critical, then we are done by \Cref{prop:adhernettocritical}. Thus, assume that $s$ and $t$ are critical (the proof for one of them being critical is similar). Let $s'$ and $t'$ be the corresponding adherent vertices, respectively. By  \ref{prop:cadisjoint}, $s'$ and $t'$ are not critical.  Then $s\to s'$ and $t'\to t$ by \ref{prop:critical-3} with $\partial(s',t')\le 16-2\le 14$ by \Cref{prop:adhernettocritical}. Thus, $\partial(s,t)\le \partial(s, s')+\partial(s',t')+\partial(t',t)\le 16$.
\end{proof}

Therefore,  it remains to prove  \Cref{prop:adhernettocritical}.  
\begin{proof}[Proof of \Cref{prop:adhernettocritical}]
We fix 
\begin{itemize}
    \item 
 $s\in U\setminus D$   and   $t\in D\setminus U$  where $s,t$ are not critical vertices, with $\alpha$ being the number of adherent vertices in $\{s,t\}$.
 \end{itemize}  
By \Cref{cor:noncritical-bounds},
$\nu(s)\le\D(s)\le8$ and $\mu(t)\le\U(t)\le8$.
If $s\in F_0$, then $\nu(s)\le4$ by \Cref{obs:munu},
and $s$ is not adherent by \ref{prop:critical-1}.
Thus, $\partial(s,t)\le4+1+8=13\le16-\alpha$.
The case $t\in F_0$ is symmetric.
Hence, we may assume throughout the remaining argument that
$s,t\notin F_0$.
We proceed by considering a shortest $st$-path in $G$. 

\begin{claim}\label{cl:root-threshold}
The following holds in $G^{O_{\ref{step:sumatmost8}}}$
and every extension.
If $x\in R_1$ and $\mu(x)\ge3$, then
$x\in F_{12}$, $N(x)\cap F_0=\{v\}$,
$N(x)\cap F_1=\emptyset$ and $x\to v$.
Moreover, $\mu(x)\ge3$ after any further orientation.
The symmetric assertions hold for $x\in L_1$
with $\nu(x)\ge3$.
\end{claim}
\begin{proofofclaim}
By symmetry, assume that $x\in R_1$.
By \Cref{step:F0X31X41} and monotonicity, every vertex in
$R_1\cap V(G^{O_{\ref{step:F0X31X41}}})$ has $\mu$-value
at most two.
Thus, $x\notin V(G^{O_{\ref{step:F0X31X41}}})$ and
$f(x)\in\{4.5,5\}$.
If $f(x)=4.5$, then $v\to x$ by \ref{item:F121},
a contradiction.
Hence, $f(x)=5$ and $N(x)\cap F_0=\{v\}$.

Suppose that $x\in F_{11}$.
If some $z\in N(x)\cap F_1$ satisfies $f(z)<5$,
then $v\to x$ by \ref{item:F12221}.
Otherwise, $f(z)=5$ for every $z\in N(x)\cap F_1$,
and \ref{item:F12224} gives either $v\to x$ or
$v\to z\to x$ for some such $z$.
In both cases, $\mu(x)\le2$, a contradiction.
Thus, $x\in F_{12}$ and $N(x)\cap F_1=\emptyset$.

Since $vx$ is oriented by the end of \Cref{step:F13}
and $\mu(x)\ge3$, its direction is $x\to v$.
Finally, $x\in V(G^{O_{\ref{step:sumatmost8}}})$,
so the assertion concerning further orientations follows
from \Cref{lem:saturation-stability}.
\end{proofofclaim}

\begin{claim}\label{lem:short-main-path}
Let $s\in U\setminus D$ and $t\in D\setminus U$ be
noncritical vertices with $\mathrm{dist}_G(s,t)\le3$,
and let $\alpha$ be the number of adherent vertices
in $\{s,t\}$.
Then $\partial(s,t)\le16-\alpha$.
Moreover, if some shortest $st$-path is disjoint from $F_0$,
then $\partial(s,t)\le\min\{15,16-\alpha\}$.
\end{claim}
\begin{proofofclaim}
Let $H=G^{O_{\ref{step:DUb}}}$, and let
$L_{st}=v_0v_1\cdots v_m$ be a shortest $st$-path,
where $v_0=s$, $v_m=t$ and $m\le3$.
Choose $L_{st}$ disjoint from $F_0$ whenever possible.
Call $v_iv_{i+1}$ an \emph{anti-edge} if $v_{i+1}\to v_i$ in $H$.
Recall that $\nu(v_0)\le\D(v_0)\le8$,
$\mu(v_m)\le\U(v_m)\le8$, and
$\partial(v_0,v_m)\le\nu(v_0)+1+\mu(v_m)$.

Suppose that $V(L_{st})\cap F_0\neq\emptyset$.
If an endpoint belongs to $F_0$, then it is not adherent
and the corresponding parameter is at most four,
so $\partial(v_0,v_m)\le13$.
Otherwise, the bounds for $F_1$ and the fact that
$F_1\subseteq A$ give the assertion unless, up to symmetry,
$m=3$, $v_1\in F_0$, $\nu(v_0)=7$, and $v_3$ is adherent
with $\mu(v_3)=\U(v_3)=8$.
Here $v_0\in R$ by \Cref{prop:boundF1}, since
$v_0\in X_3$ would give $\nu(v_0)\le6$ by \Cref{claimcenter}.
Thus, $v_1\in\{v\}\cup R_{11}\cup X_2$, and
$\mu_H(v_1)\le2$, $\nu_H(v_1)\ge2$.

Since $v_3$ is adherent, $v_2v_3$ is unoriented in $H$
and $v_2\in F_1\cap(D\setminus U)$.
The edge $v_1v_2$ is oriented before \Cref{step:L23R23}.
If $v_1\to v_2$, then $\mu_H(v_2)\le3$.
Otherwise, \Cref{claimcenter} and monotonicity give
$\mu_H(v_2)\le8-\nu_H(v_1)\le6$.
By \Cref{lem:good-path-gluing},
$\U(v_3)\le\mu_H(v_2)+1\le7$, a contradiction.

Thus, assume that $V(L_{st})\cap F_0=\emptyset$.
If $\nu_H(v_0)\le6$ or $\mu_H(v_m)\le6$, the corresponding
endpoint belongs to $V(H)$ and is not adherent.
Hence, $\alpha\le1$ and
$\partial(v_0,v_m)\le15\le16-\alpha$.
Therefore, assume that $\nu_H(v_0)>6$ and $\mu_H(v_m)>6$.

If $v_0v_1$ is not an anti-edge and $v_1\in D$,
then $v_0\to v_1$ in $H$ and $\nu_H(v_0)\le6$ by the
definition of $D_a$ and \Cref{claim:xinDyinU},
a contradiction.
The other end is symmetric.
Thus, $v_1\in U\setminus D$ whenever $v_0v_1$ is not
an anti-edge, and $v_{m-1}\in D\setminus U$ whenever
$v_{m-1}v_m$ is not an anti-edge.
Moreover, every anti-edge incident to an endpoint is
type-II by \Cref{prop:t1edge}.

Suppose that there is no anti-edge.
Then $m=3$, $v_1\in U\setminus D$, $v_2\in D\setminus U$,
and $v_1\to v_2$ in $H$.
By \Cref{claim:xinDyinU,lem:good-path-gluing},
$\D(v_1),\U(v_2)\le6$ and $\D(v_0),\U(v_3)\le7$.
The assertion follows unless $\alpha=2$ and
$\D(v_0)=\U(v_3)=7$.
In this case, by  \ref{lem:gold-adherent-edge}, it holds that 
$v_0\to v_1$ and $v_2\to v_3$, so
$\partial(v_0,v_3)\le3$.

Suppose that there is exactly one anti-edge and it is type-I.
Then $m=3$ and this edge is $v_2\to v_1$.
By \Cref{prop:t1edge,lem:good-path-gluing},
$\D(v_0)+\U(v_3)\le\nu_H(v_1)+\mu_H(v_2)+2\le12$,
and we are done.

By \Cref{step:L23R23} and \Cref{cl:root-threshold},
every type-II arc $q\to p$ has both ends in $A$ and
satisfies $\nu_H(p)\le3$ or $\mu_H(q)\le3$.
These rules also imply that a type-II anti-edge incident
to an endpoint of $L_{st}$ has its other endpoint in
$L_1\cup R_1$, since $\nu_H(v_0),\mu_H(v_m)>6$.     

Suppose that $v_{i+1}\to v_i$ is the only anti-edge
and it is type-II.
By symmetry, assume that $\nu_H(v_i)\le3$.
Since $i\le2$, we have $v_0,\ldots,v_{i-1}\in U$.
If $v_i\in U$, then $\D(v_0)\le\nu_H(v_i)+i\le5$.
Otherwise, $i\ge1$ and $v_i\in D_a$.
The edge $v_{i-1}v_i$ is a crossing arc
$v_{i-1}\to v_i$ in $H$, so $\nu_H(v_{i-1})\le4$.
Again, $\D(v_0)\le\nu_H(v_{i-1})+i-1\le5$.
The other case gives $\U(v_m)\le5$ symmetrically.
Thus, $\partial(v_0,v_m)\le14$.

It remains to consider at least two anti-edges.
Two consecutive type-II anti-edges are impossible,  since $m\le3$ and  a vertex in $L_1\cup R_1$ cannot be both
the head and the tail of type-II edges by \Cref{step:L23R23}. 

Suppose that both end edges are type-II anti-edges.
Then $m=3$ and $v_1,v_2\in L_1\cup R_1$.
If they belong to different sides, their adjacency gives
$v_1,v_2\in F_0$, a contradiction.
If $v_1,v_2\in R_1$, the edge $v_3\to v_2$ requires
$\mu(v_2)\ge3$ when it is oriented.
By \Cref{cl:root-threshold}, $N(v_2)\cap F_1=\emptyset$,
contrary to $v_1\in N(v_2)\cap F_1$.
If $v_1,v_2\in L_1$, the edge $v_1\to v_0$ requires
$\nu(v_1)\ge3$, and the same claim gives a contradiction.

Thus, up to symmetry, the only remaining case is $m=3$,
where $v_1\to v_0$ is type-II, $v_2\to v_1$ is type-I,
and $v_2v_3$ is not an anti-edge.
By \Cref{step:L23R23}, \Cref{cl:root-threshold} and
(\ref{partitionA}), we have $v_1\in U_a\setminus D_a$
and $\nu_H(v_1)\ge3$.
Here, when $v_1\in R_1$, the latter inequality follows
from $v_1\notin F_0$.

Since \Cref{step:DUb} orients no edge towards
$U_a\setminus D_a$, the edge $v_2\to v_1$ is oriented
before \Cref{step:L23R23}.
Thus, \Cref{claimcenter} and monotonicity give
$\mu_H(v_2)+1+\nu_H(v_1)\le9$, so $\mu_H(v_2)\le5$.
By \Cref{lem:good-path-gluing},
$\U(v_3)\le\mu_H(v_2)+1\le6$.
Since $v_0$ is not adherent, $\alpha\le1$, and hence
$\partial(v_0,v_3)\le8+1+6=15\le16-\alpha$.

These cases exhaust all possibilities and also give
$\partial(v_0,v_m)\le15$ whenever
$V(L_{st})\cap F_0=\emptyset$.
The assertion follows.
\end{proofofclaim}
By \Cref{lem:short-main-path}, it remains to consider $\mathrm{dist}_G(s,t)=4$. Let $L_{st}=v_0v_1v_2v_3v_4 $ be a shortest   $st$-path with $s=v_0$ and $t=v_4$.  Recall, if $v_{i}v_{i+1}$ is oriented as $v_{i+1}\to v_i$ 
in $G^{O_{\ref{step:DUb}}}$, then call it an  anti-edge. Among all the anti-edges in $L_{st}$ the anti-edge $v_iv_{i+1}$ with the smallest $i$ is called the \emph{first} and the one with the largest $i$ is called the \emph{last}.   
Let $H=G^{O_{\ref{step:DUb}}}$.
For each $i$ with $\nu_H(v_i)<\infty$, choose a dipath
$Q_{v_i}$ in $H$ from $v_i$ to $u$ or $v$ with adjusted
length $\nu_H(v_i)$, and let $Q_i=v_0\cdots v_{i-1}Q_{v_i}$.
For each $i$ with $\mu_H(v_i)<\infty$, choose a dipath
$P_{v_i}$ in $H$ from $u$ or $v$ to $v_i$ with adjusted
length $\mu_H(v_i)$, and let $P_i=P_{v_i}v_{i+1}\cdots v_4$. 

Note that $s,t\notin F_0$. Then either $L_{st}\cap F_0=\emptyset$ or  $ \{v_1,v_2,v_3\}\cap F_0\neq \emptyset$. Thus, it is sufficient to consider the following cases (\ref{SF-1}--\ref{SF-3}).
\begin{enumerate}[label=(F\arabic*)]
    \item\label{SF-1} $V(L_{st})\cap F_0=\emptyset$. Then at least one of the following holds. 
    \begin{enumerate}[label=(F1.\arabic*)] 
       \item\label{Final-1} There is no anti-edge, proved in \Cref{cl:f-1}. 
        \item\label{Final-2} Exactly one anti-edge and it is  type-I, proved in \Cref{cl:f-2}. 
         \item\label{Final-3}  Exactly one anti-edge and it is  type-II, proved in \Cref{cl:f-3}.
         \item\label{Final-4}  The first anti-edge or the last anti-edge is  type-II, except for the case that $v_0v_1$ and $v_3v_4$ are anti-edges and they are type-II,  proved in \Cref{cl:f-4}.
         \item\label{Final-7}  $v_0v_1$ and $v_3v_4$ are anti-edges and they are type-II, proved in \Cref{cl:f-7}.
          \item\label{Final-6}  $v_0v_1$ or  $v_3v_4$ is a  type-I anti-edge, proved in \Cref{cl:f-6}. 
         \item\label{Final-5} Only $v_1v_2$ and $v_2v_3$ are anti-edges and they are type-I, proved in \Cref{cl:f-5}.
    \end{enumerate} 

    By \ref{Final-1}--\ref{Final-7}, assume that there are at least two anti-edges with the first anti-edge and the last anti-edge being type-I.  Then either encounters \ref{Final-6} or \ref{Final-5}. 
   \item\label{SF-2} $v_1\in F_0$ or $v_3\in F_0$. 

Without loss of generality, assume that $v_1\in F_0$. Then it is enough to consider the following cases.
 \begin{enumerate}[label=(F2.\arabic*)] 
       \item\label{Final2-1} $v_1\in F_0$,  $v_0\in R_{12}\cap F_1$ with $\nu(v_0)=7$ and $\mu(v_4)=8$ with  $v_4$ being adherent, proved in \Cref{cl:f2-1}. 
       \item\label{Final2-2} $\nu(v_0)=7$, $v_1\in F_0$, $\mu(v_4)=8$ with  $v_4$ being adherent,  and there is $v_0'\in N(v_0)\cap F_0$ such that $v_0'\in R_{11}$, $v_0'\to v_0$ is oriented in \Cref{step:cycle(a)} and  $\omega(v_0'\to v_0)=9$, proved in \Cref{cl:f2-2}.
\end{enumerate}
Since $v_0\in F_1\subseteq A$, it is not adherent.
By \Cref{prop:boundF1}, $\nu(v_0)\le6$ when $v_0\in L$
and $\nu(v_0)\le7$ when $v_0\in R$.
If $v_0\in X_3$, then $\mu(v_0)\ge3$, and
\Cref{claimcenter} gives $\nu(v_0)\le6$.
If $v_4$ is not adherent, then $\alpha=0$ and
$\partial(v_0,v_4)\le7+1+8=16$.
If $\nu(v_0)\le6$ or $\mu(v_4)\le7$, then
$\partial(v_0,v_4)\le15\le16-\alpha$.
Thus, assume that $v_0\in R\cap F_1$ with
$\nu(v_0)=7$, $\mu(v_4)=8$, and $v_4$ is adherent.

If $v_0\in V(G^{O_{\ref{step:F0X31X41}}})$, then
$v_0\in R_{12}$ by \Cref{step:F0X31X41}
and \Cref{claimcontrol1}.
Thus, assume that
$v_0\notin V(G^{O_{\ref{step:F0X31X41}}})$.
Since $\mu(v_0)\ge1$ and $\nu(v_0)=7$, no edge
incident to $v_0$ can have weight at most seven. 
By \Cref{prop:F118,prop:cycleauv,prop:F12224,prop:E5},
none of the edges in $[v_0,F_0]$ is oriented in
\Cref{step:F12}, \Cref{step:cycle(auv)},
\Cref{step:F12224} or \Cref{step:F13}.
Thus, all these edges are oriented in \Cref{step:cycle(a)}.
By \ref{item:F121}, $f(v_0)$ is an integer.
If $f(v_0)=5$, then $N(v_0)\cap F_0=\{v\}$,
and the considered cycle containing $v_0v$ belongs to
\Cref{step:cycle(auv)}, a contradiction.
Hence, $f(v_0)=4$.

Suppose that $v_0\in F_{13}$, and consider the first
operation incident to $v_0$.
By the proof of \Cref{lem:cycle-star-choice}, it orients
a whole cycle or a four-edge path of form (d).
In forms (b) and (d), the completed dicycle gives a
dipath from $v_0$ to $F_0$ of length at most three,
and hence $\nu(v_0)\le6$, a contradiction.
In form (a), both vertices of $F_1$ belong to $R_{22}$.
A neighbor of either one in $F_2$ cannot lie in $L_2$
by the definition of $R_{21}$, or in $X_3$ by \Cref{claim1}.
It cannot lie in $L_3\cup X_4$, since its distance from
$v$ is at most three.
Thus, both $F_2$-vertices belong to $R$, and the cycle
lies in $F_0\cup R$.
The choice in \ref{item:cycle5-whole} gives
$\nu(v_0)\le4$, since the direction from $v_0$ to
$I_4$ is available, again a contradiction.
Therefore, $v_0\notin F_{13}$, and none of its incident
edges is an additional edge in \ref{item:cycle5-short}. 
Choose $v_0'\in N(v_0)\cap R_{11}$.
If $v_0\to v_0'$, then
$\nu(v_0)\le1+\nu(v_0')\le4$, a contradiction.
Hence, $v_0'\to v_0$.
Since $\mu(v_0')=1$ and $\nu(v_0)=7$, by
\Cref{claimcenter}, it holds that
$\omega(v_0'\to v_0)=9$.

    \item\label{SF-3} $v_2\in F_0$ with $ v_1 ,v_3\notin F_0$. 

  We may further assume that $v_0,v_4\in F_2$.
Indeed, since $v_2\in F_0$ and $v_0,v_4\notin F_0$,
we have that $v_0,v_4\in F_1\cup F_2$.
Suppose that $v_0\in F_1$.
As in \ref{SF-2}, it suffices to consider
$\nu(v_0)=7$ and $\mu(v_4)=\U(v_4)=8$,
where $v_4$ is adherent.

Since $v_1\in F_1$, it holds that
$v_0\in V(G^{O_{\ref{step:F0X31X41}}})\cup F_{11}$.
If $v_0\in F_{11}$, then
$\mu(v_0)+\nu(v_0)\le7$ by
\Cref{prop:F118,prop:cycleauv,prop:F12224}.
Since $\mu(v_0)\ge1$, we obtain $\nu(v_0)\le6$,
a contradiction.
Thus, $v_0\in V(G^{O_{\ref{step:F0X31X41}}})$.
By \Cref{step:F0X31X41} and \Cref{claimcontrol1},
$\nu(v_0)=7$ implies that $v_0\in R_{12}$.
The result follows from \Cref{cl:f2-1},
whose proof does not use $v_1\in F_0$.
The case $v_4\in F_1$ is symmetric.
  
  Since $v_1v_2,v_2v_3\in [F_0,F_1]\subseteq E(G^{O_{\ref{step:F13}}})$,   
    we focus on the orientations on $v_0v_1$ and $v_3v_4$. Then it suffices to consider the following cases. 
\begin{enumerate}[label=(F3.\arabic*)] 
\item\label{Final3-1} $\orw{v_0v_1}\in E(G^{O_{\ref{step:DUb}}})$ and is  type-I or  $\orw{v_3v_4}\in E(G^{O_{\ref{step:DUb}}})$ and is  type-I, proved in \Cref{cl:f3-1}.
\item\label{Final3-2} Both  $\orw{v_0v_1},\orw{v_3v_4}\in E(G^{O_{\ref{step:DUb}}})$  and they are type-II, proved in \Cref{cl:f3-2}
\item\label{Final3-3} Both   $v_0v_1$ and $v_3v_4$ are unoriented after \Cref{step:DUb}, proved in \Cref{cl:f3-3}.
\item\label{Final3-4} One of $v_0v_1$ and $v_3v_4$ is type-II and the other one is unoriented in $G^{O_{\ref{step:DUb}}}$, proved in \Cref{cl:f3-4}.
\end{enumerate}
    
\end{enumerate}
Therefore, the result holds. 
\end{proof}

The proofs of the mentioned lemmas are presented in order.  Recall $\nu(s)\le \D(s)\le 8$ and $\mu(t)\le \U(t)\le 8$  by \ref{prop:critical-4} and \ref{prop:cadisjoint}. 
Moreover,  by the argument in the proof of \Cref{lem:Gxy16}, if $\D(v_0)+\U(v_4)\le k$, then $$\nu(v_0) +\mu(v_4)\le k  \text{ and so }  \pt(v_0,v_4)\le k+1.$$ 

 We first show the following statements hold, which are helpful later. 
\begin{proposition}
  \label{cl:v1v4}
  Suppose that $v_0v_1$ is not an anti-edge and $v_1\in D$, or that $v_3v_4$ is not an anti-edge and $v_3\in U$. Then $\pt(v_0,v_4)\le16-\alpha$.
\end{proposition}
\begin{proof}
 By symmetry, assume that $v_0v_1$ is not an anti-edge and $v_1\in  D$. Then $v_0\to v_1\in E(G^{O_{\ref{step:DUb}}})$. If $v_1\in D_a$, then $\nu(v_0)\le 5$ and so $\nu(v_0)+\mu(v_4)\le 13$.  Thus, assume that $v_1\in D_b$. Then  $v_0\to v_1$ is oriented in \Cref{step:DUb}. It follows that  $\nu(v_0)\le 6$ and $v_0$ is not adherent. Thus, $\nu(v_0)+\mu(v_4)\le 14$.  
\end{proof}

Therefore, if we have $v_0v_1$ or $v_3v_4$ is not an anti-edge. Then it suffices to assume that $v_1\in U\setminus D$ or $v_3\in D\setminus U$ by \Cref{cl:v1v4}, respectively. 

\begin{proposition}
    \label{cl:forv4}
  Assume $\mu(v_4)=8$ and $v_4$ is adherent. Then   $v_4\notin F_1$. Moreover, if $v_4\in R$, then $v_4\in F_3$.  
\end{proposition}
\begin{proof}
By \ref{prop:critical-1}, $v_4\in B$, so $v_4\notin F_1$
and all its incident edges are unoriented after
\Cref{step:DUb}.
Suppose that $v_4\in F_2\cap R$, and choose
$r\in N(v_4)\cap F_1$.
Since $v_4\notin R_{21}$, we have $r\in R\cup X_3$.

Let $H=G^{O_{\ref{step:sumatmost8}}}$.
If $r\notin R_1$, then $\mathrm{dist}_G(r,u)\ge3$ and
$\mathrm{dist}_G(r,v)\ge2$, so $\nu_H(r)\ge3$.
If $r\in R_1$, then $r\in F_1$ implies $r\notin R_{11}$,
and hence $N(r)\cap L_1=\emptyset$.
Together with $u\to v$, this gives $\partial_H(r,u)\ge3$.
Since $\partial_H(r,v)+2\ge3$, we again have $\nu_H(r)\ge3$.
Thus, $\mu_H(r)\le6$ by \Cref{claimcenter}.

As $rv_4$ is unoriented after \Cref{step:DUb},
\Cref{claim:xinDyinU}(4) gives $r\in D$.
By \Cref{cor:noncritical-bounds} and \Cref{lem:good-path-gluing}, $
8=\mu(v_4)\le\U(v_4)\le\U(r)+1
\le\mu_H(r)+1\le7,
$ 
a contradiction.
Hence, $v_4\in R$ implies $v_4\in F_3$ by \Cref{claim3}.
\end{proof}

\subsection{\texorpdfstring{Paths avoiding  $F_0$}{Paths avoiding F0}}\label{sec:paths-outside-core}
Now, we start the proof of the mentioned lemmas, beginning with lemmas involved in \ref{SF-1}.  Recall that  a common assumption for \Cref{cl:f-6}-\Cref{cl:f-7} is that $V(L_{st})\cap F_0=\emptyset$.

\begin{lemma}
    \label{cl:f-6}
Let $v_0v_1$ or $v_3v_4$ be an anti-edge and suppose that it is  type-I,  i.e., \ref{Final-6}.  Then  $\pt(v_0,v_4)\le 16-\alpha$. Moreover, $\nu(v_0)\le 6$ and $v_0$ is non-adherent or   $\mu(v_4)\le 6$ and $v_4$ is non-adherent, respectively. 
\end{lemma}
\begin{proof}
By symmetry, assume that $v_0v_1$ is an anti-edge of type-I. Then $\nu(v_0)\le6$ by \Cref{prop:t1edge},
and $v_0$ is not adherent by \ref{prop:critical-1}. Since $\mu(v_4)\le \U(v_4)\le 8$, we are done. 
\end{proof}

\begin{lemma}\label{cl:f-5}
Let $v_1v_2$  and  $v_2v_3$ be the only   anti-edges and they are type-I,  i.e., \ref{Final-5}.  Then $\pt(v_0,v_4)\le 16-\alpha$. Moreover,  $\nu(v_0)\le 7$ and   $\mu(v_4)\le 7$. 
\end{lemma}
\begin{proof} 
Let $H=G^{O_{\ref{step:DUb}}}$.
By the proof of \Cref{prop:t1edge}, we have
$\nu_H(v_1)\le6$ and $\mu_H(v_3)\le6$.
By \Cref{lem:good-path-gluing},
$\D(v_0)\le\nu_H(v_1)+1$ if $v_1\in U$.
If $v_1\notin U$, then $v_0\to v_1$ in $H$ by
\Cref{claim:xinDyinU}, and the same inequality follows
from $\D(v_0)\le\nu_H(v_0)\le\nu_H(v_1)+1$.
Similarly, $\U(v_4)\le\mu_H(v_3)+1$.
Thus, $\nu(v_0)\le\D(v_0)\le7$ and
$\mu(v_4)\le\U(v_4)\le7$.

By \Cref{lem:Gxy16}, it remains to consider
$\nu(v_0)+\mu(v_4)=14$ and $\alpha=2$.
Then $\D(v_0)=\U(v_4)=7$ and
$\nu_H(v_1)=\mu_H(v_3)=6$.
Both endpoints are adherent-I.
Their incident edges are unoriented in $H$, so
$v_1\in U\setminus D$ and $v_3\in D\setminus U$ by
\Cref{claim:xinDyinU}.
Since $\D(v_1)\le6<\D(v_0)$, by
\ref{prop:critical-2}, we have $v_1\in U_a\setminus D_a$.
Similarly, $v_3\in D_a\setminus U_a$.
Thus, neither $v_2\to v_1$ nor $v_3\to v_2$ is oriented
in \Cref{step:DUb}.
Both are type-I, so they are oriented before
\Cref{step:L23R23}.
By \Cref{claimcenter} and monotonicity,
$\mu_H(v_2)\le8-\nu_H(v_1)=2$ and
$\nu_H(v_2)\le8-\mu_H(v_3)=2$.
By \Cref{obs:munu}, it follows that $v_2\in F_0$,
a contradiction. 
\end{proof}

\begin{lemma}
   \label{cl:f-2}
Let $v_{i}v_{i+1}$ be the only anti-edge in $L_{st}$  and   it is  type-I,  i.e., \ref{Final-2}.  Then $\pt(v_0,v_4)\le 15-\alpha$. 
\end{lemma}
\begin{proof}
Let $H=G^{O_{\ref{step:DUb}}}$.
By the proof of \Cref{prop:t1edge}, we have
$\mu_H(v_{i+1}),\nu_H(v_i)\le6$ and
$\mu_H(v_{i+1})+\nu_H(v_i)\le10$.
All other edges of $L_{st}$ are unoriented or oriented
from $v_0$ towards $v_4$.
We use \Cref{claim:xinDyinU,lem:good-path-gluing}
for the following estimates.
By symmetry, it suffices to consider $i=0$ and $i=1$.

Suppose first that $i=0$.
Then $\nu(v_0)\le\nu_H(v_0)\le6$ and $v_0$ is not
adherent, so $\alpha\le1$.
If $v_3\in U$, then $v_3\to v_4$ in $H$.
Thus, $\mu_H(v_4)\le6$, $v_4$ is not adherent, and
$\pt(v_0,v_4)\le6+1+6=13$. 
Hence, assume that $v_3\in D\setminus U$.
If $v_2\in U$, then $v_2\to v_3$ in $H$, and
$\U(v_4)\le\mu_H(v_3)+1\le7$.
It follows that $\pt(v_0,v_4)\le14\le15-\alpha$.
Thus, assume that $v_2\in D\setminus U$.
If $v_1\in D$, then
$\U(v_4)\le\mu_H(v_1)+3$.
Otherwise, $v_1\to v_2$ in $H$, and
$\U(v_4)\le\mu_H(v_2)+2\le\mu_H(v_1)+3$.
Consequently,
$\nu(v_0)+\mu(v_4)
\le\nu_H(v_0)+\mu_H(v_1)+3\le13$,
and the result follows.

Now, suppose that $i=1$.
If $v_1\in U$, then
$\D(v_0)\le\nu_H(v_1)+1$.
Otherwise, $v_0\to v_1$ in $H$, and
$\nu_H(v_0)\le\nu_H(v_1)+1$.
In either case,
$\nu(v_0)\le\nu_H(v_1)+1\le7$. If $v_3\in U$, then $v_3\to v_4$ in $H$,
$\mu_H(v_4)\le6$, and $v_4$ is not adherent.
Thus, $\alpha\le1$ and
$\pt(v_0,v_4)\le7+1+6=14\le15-\alpha$.
Hence, assume that $v_3\in D\setminus U$.
If $v_2\in D$, then
$\U(v_4)\le\mu_H(v_2)+2$.
Otherwise, $v_2\to v_3$ in $H$, and
$\U(v_4)\le\mu_H(v_3)+1\le\mu_H(v_2)+2$.
Therefore,
$\nu(v_0)+\mu(v_4)
\le\nu_H(v_1)+\mu_H(v_2)+3\le13$.
If $\alpha\le1$, we are done. 
Thus, assume that both endpoints are adherent.
If $v_2\to v_1$ is oriented before
\Cref{step:L23R23}, then \Cref{claimcenter} gives
$\nu_H(v_1)+\mu_H(v_2)\le8$.
Hence, $\pt(v_0,v_4)\le8+3+1=12<13$. 
Otherwise, $v_2\to v_1$ is oriented in \Cref{step:DUb}.
Since $v_0$ is adherent, $v_0v_1$ is unoriented in $H$,
and so $v_1\in U\setminus D$ by
\Cref{claim:xinDyinU}(4).
Since \Cref{step:DUb} orients no edge towards
$U_a\setminus M_a$, we have $v_1\in U_b\setminus D_b$.
Moreover, $\D(v_0)\le\nu_H(v_1)+1\le7$,
so $v_0$ is adherent-I.
By \ref{prop:critical-2}, it follows that
$\D(v_0)\le\D(v_1)\le\nu_H(v_1)$.
Consequently,
$\nu(v_0)+\mu(v_4)
\le\D(v_0)+\U(v_4)
\le\nu_H(v_1)+\mu_H(v_2)+2\le12$.
By \Cref{lem:Gxy16}, we obtain
$\pt(v_0,v_4)\le13=15-\alpha$.
\end{proof}

\begin{lemma}\label{cl:f-3}
Let $v_iv_{i+1}$ be the only anti-edge in $L_{st}$
and suppose that it is type-II, i.e., \ref{Final-3}.
Then $\pt(v_0,v_4)\le16-\alpha$.
Furthermore,  either $\min\{\nu(v_0),\mu(v_4)\}\le6$,
or $\min\{\nu(v_0),\mu(v_4)\}=7$ and neither endpoint
is adherent.
\end{lemma}
\begin{proof}
Let $H=G^{O_{\ref{step:DUb}}}$.
By the proof of \Cref{cl:v1v4}, both assertions hold
if $i\neq0$ and $v_1\in D$, or $i\neq3$ and $v_3\in U$;
these cases give $\nu(v_0)\le6$ or $\mu(v_4)\le6$,
respectively.
Thus, assume that $v_1\in U\setminus D$ when $i\neq0$,
and $v_3\in D\setminus U$ when $i\neq3$.

By \Cref{step:L23R23} and \Cref{cl:root-threshold}, we have
$\nu_H(v_i)\le3$ or $\mu_H(v_{i+1})\le3$.
Suppose first that $\nu_H(v_i)\le3$.

If $i=0$, then $\D(v_0)\le\nu_H(v_0)\le3$.
For $i\in\{1,2\}$, the vertices
$v_0,\ldots,v_{i-1}$ belong to $U$.
If $v_i\in U$, then \Cref{lem:good-path-gluing} gives
$\D(v_0)\le\nu_H(v_i)+i\le5$.
Otherwise, $v_{i-1}\to v_i$ in $H$, and
$\nu_H(v_{i-1})\le\nu_H(v_i)+1$ gives the same bound.
Thus, when $i\le2$, we have $\nu(v_0)\le5$ and
$\pt(v_0,v_4)\le5+1+8=14\le16-\alpha$.

Hence, assume that $i=3$.
Since $v_4$ is incident to a type-II edge, it belongs
to $A$ and is not adherent.
If $\D(v_0)\le6$, then $\nu(v_0)\le6$ and
$\pt(v_0,v_4)\le6+1+8=15\le16-\alpha$.
Thus, assume that $\D(v_0)\ge7$. 
Suppose that $v_2\in U$.
If $v_3\in U$, gluing along $v_0v_1v_2v_3$ gives
$\D(v_0)\le\nu_H(v_3)+3\le6$, a contradiction.
Otherwise, $v_2\to v_3$ in $H$, and
$\D(v_0)\le\nu_H(v_2)+2\le\nu_H(v_3)+3\le6$,
again a contradiction.
Thus, $v_2\in D\setminus U$.
If $v_2\in D_a$, then $v_1\to v_2$ in $H$, and
$\D(v_0)\le\nu_H(v_1)+1\le\nu_H(v_2)+2\le6$,
a contradiction. 
Therefore, $v_2\in D_b\setminus U_b$.
By \Cref{rm:edgeanticross}, $v_1\in U_b\setminus D_b$.
Since $v_1\to v_2$ in $H$, \Cref{claim:xinDyinU}
gives $\nu_H(v_1)\le6$.
It follows that
$\D(v_0)\le\D(v_1)+1\le\nu_H(v_1)+1\le7$,
and hence $\D(v_0)=7$.
If $v_0$ is adherent, then it is adherent-I,
since an adherent-II vertex in $U_b\setminus D_b$
has $\D$-value eight.
By \ref{prop:critical-2}, we obtain
$\D(v_0)\le\D(v_1)\le\nu_H(v_1)\le6$,
a contradiction.
Thus, neither endpoint is adherent, and
$\pt(v_0,v_4)\le7+1+8=16$. 
In the case considered, we have $\nu(v_0)\le6$,
or $\nu(v_0)\le7$ with neither endpoint adherent.
When $\mu_H(v_{i+1})\le3$, the same argument with
reversed directions gives the distance bound and
$\mu(v_4)\le7$.
Thus, the additional assertion also holds.
\end{proof}

\begin{lemma}\label{cl:f-1}
Suppose there is no anti-edge, i.e., \ref{Final-1}.  Then $\pt(v_0,v_4)\le 16-\alpha$. 
\end{lemma}
\begin{proof}
By \Cref{cl:v1v4}, assume that
$v_0,v_1\in U\setminus D$ and
$v_3,v_4\in D\setminus U$.
By symmetry, assume that $v_2\in U$.
In this proof, $\mu$ and $\nu$ are computed in
$G^{O_{\ref{step:DUb}}}$.

Since there is no anti-edge, $v_2\to v_3$.
By \Cref{claim:xinDyinU}(5), we have
$\nu(v_2),\mu(v_3)\le6$ and $\nu(v_3)\le5$.
By \Cref{lem:good-path-gluing}, it holds that
$\D(v_0)\le\D(v_1)+1\le\D(v_2)+2\le\nu(v_2)+2\le8$
and
$\U(v_4)\le\U(v_3)+1\le\mu(v_3)+1\le7$.
Recall that
$\pt(v_0,v_4)\le\D(v_0)+\U(v_4)+1$.
Thus, it suffices to consider $\D(v_0)=8$, or
$\D(v_0)=\U(v_4)=7$ with both endpoints adherent.

Suppose first that $\D(v_0)=8$.
Then $\D(v_1)=7$ and $\D(v_2)=\nu(v_2)=6$.
Since $v_2\to v_3$, we have $\nu(v_3)=5$,
and so $v_3\in D_b\setminus U_b$.
The bounds for $M_a\cup M_b$ and
\Cref{rm:edgeanticross} give
$v_2\in U_b\setminus D_b$.
Thus, $v_2\to v_3$ is oriented in \Cref{step:DUb}. If $v_4$ is adherent, then it is adherent-I, and
\ref{prop:critical-2} gives
$\U(v_4)\le\U(v_3)\le6$.
Hence, the remaining cases have $v_0$ adherent, and
either $\U(v_4)=7$, or $\U(v_4)=6$ with $v_4$ adherent.
We show that $L_{st}$ is a dipath in the final orientation. 

By \ref{lem:gold-adherent-edge}, we have $v_0\to v_1$.
If $v_1\in U_a$, then $v_1\to v_2$ by
\Cref{claim:xinDyinU}(2).
If $v_1\in U_b$, then $\D(v_1)>\D(v_2)$, and
\ref{S8-41} gives the same direction if this edge
is still unoriented.
Any edge already oriented after \Cref{step:DUb}
has the required direction, since there is no anti-edge.

If $\U(v_4)=7$, then $\U(v_3)=6$.
Thus, $v_3\to v_4$ follows from \ref{S8-11} or
\ref{S8-31} if this edge is still unoriented. 
Suppose that $\U(v_4)=6$ and $v_4$ is adherent.
Then $\U(v_3)=6$ by \ref{prop:critical-2}.
By \Cref{claim:xinDyinU}(2), all the edges in
$[v_3,D_a]$ are oriented from $v_3$ to $D_a$.
Hence, $v_3\neq\hat{x}$ for every $x\in T_D$ by
\Cref{lem:exceptional-vertices}.
That lemma also gives $v_3,v_4\notin T_D$.
Since $v_4$ is adherent, $v_3v_4$ is unoriented after
\Cref{step:DUb}.
Neither \ref{S8-15} nor \ref{S8-31} orients it,
and so \ref{S8-32} gives $v_3\to v_4$.
Therefore, $L_{st}$ is a dipath, and we are done.

Now, suppose that $\D(v_0)=\U(v_4)=7$ and both
endpoints are adherent.
Then both are adherent-I.
Since $\U(v_3)\le6$, by \ref{prop:critical-2},
we have $v_3\in D_a\setminus U_a$.
Also, $\mu(v_3)=6$, and so $v_2\notin U_a$,
as otherwise $\mu(v_3)\le\mu(v_2)+1\le5$.
By \Cref{rm:edgeanticross}, it follows that $v_2\in M_b$. 
Now, $\D(v_1)\le\nu(v_2)+1\le6$, and
\ref{prop:critical-2} gives $v_1\in U_a\setminus D_a$.
Thus, \ref{S7-1} gives $v_1\to v_2\to v_3$,
while \ref{lem:gold-adherent-edge} gives
$v_0\to v_1$ and $v_3\to v_4$.
Again, $L_{st}$ is a dipath, giving
$\pt(v_0,v_4)\le4\le16-\alpha$.
\end{proof}

\begin{lemma}\label{cl:f-4}
Suppose that  $v_{i}v_{i+1}$ is the first anti-edge and it is  type-II except for the case that $v_0v_1$ and $v_3v_4$ are anti-edges and they are type-II, i.e., \ref{Final-4}. Then $\pt(v_0,v_4)\le 16-\alpha$. 
\end{lemma} 
\begin{proof}
If there is only one anti-edge, the result follows from
\Cref{cl:f-3}.
Thus, assume that there are at least two anti-edges.
Let $v_jv_{j+1}$ be the last anti-edge.
Then $i<j$ and $i\le2$.
In this proof, $\mu$ and $\nu$ are computed in
$G^{O_{\ref{step:DUb}}}$.

By \Cref{cl:v1v4}, assume that $v_1\in U\setminus D$
when $i\ge1$, and $v_3\in D\setminus U$ when $j\le2$.
If $\nu(v_i)\le3$, then
$\D(v_0)\le\nu(v_i)+i\le5$ by
\Cref{lem:good-path-gluing}.
Here, if $i=2$ and $v_2\notin U$, we use the crossing
arc $v_1\to v_2$ and
$\nu(v_1)\le\nu(v_2)+1$.
Thus, $\pt(v_0,v_4)\le5+8+1=14\le16-\alpha$.

Hence, assume that $\nu(v_i)\ge4$.
By \ref{s6-1}, \ref{s6-2} and \Cref{cl:root-threshold},
either $v_{i+1}\in R_1$ with $\mu(v_{i+1})\le2$,
or $v_{i+1}\in L_1$ with $\nu(v_{i+1})\ge3$.
In both cases, (\ref{partitionA}) gives
$v_{i+1}\in U_a\setminus D_a$, and
$\nu(v_{i+1})\ge3$.

\medskip
\noindent{\bf Case 1.} $j=i+1$.

The edge $v_{i+2}\to v_{i+1}$ is not type-II,
since its direction would contradict the corresponding
threshold in \ref{s6-1} or \ref{s6-2}.
Nor is it oriented in \Cref{step:DUb}, since
$v_{i+1}\in U_a\setminus D_a$.
Thus, by \Cref{claimcenter},
$\mu(v_{i+2})+\nu(v_{i+1})\le8$,
and so $\mu(v_{i+2})\le5$.

Suppose that $\mu(v_{i+2})=5$.
Then $\nu(v_{i+1})=3$, and the weight of
$v_{i+2}\to v_{i+1}$ in
$G^{O_{\ref{step:sumatmost8}}}$ is nine.
By \Cref{claimcenter}, there is a dicycle $C$ of
length five containing this edge and some
$y\in I_k$, where $k\in\{2,3,4\}$.
Moreover, $C$ avoids $u$ and $v$, since otherwise
this edge would have weight at most seven. 
Suppose that $v_{i+1}\in R_1$.
Since $v_{i+1}\notin R_{11}$, the definition of
$R_{11}$ gives
$\mathrm{dist}_G(v_{i+1},I_k)\ge5-k$.
Thus, the directed subpath of $C$ from $v_{i+1}$
to $y$ has length at least $5-k$.
The directed subpath from $y$ to $v_{i+2}$ therefore
has length at most $k-1$.
Since $\mu(y)\le5-k$, we obtain
$\mu(v_{i+2})\le(5-k)+(k-1)=4$,
a contradiction. 
Suppose that $v_{i+1}\in L_1$.
By \Cref{cl:root-threshold},
$N(v_{i+1})\cap F_0=\{u\}$ and
$N(v_{i+1})\cap F_1=\emptyset$.
Since $u\notin V(C)$, both subpaths of $C$ between
$v_{i+1}$ and $y$ have length at least three,
a contradiction.

Therefore, $\mu(v_{i+2})\le4$.
By \Cref{lem:good-path-gluing}, it follows that
$\U(v_4)\le\mu(v_{i+2})+2-i\le6-i$. 
Here, if $i=0$ and $v_2\notin D$, use the crossing arc
$v_2\to v_3$ in $G^{O_{\ref{step:DUb}}}$ to obtain
$\U(v_4)\le\mu(v_3)+1\le\mu(v_2)+2$. 
When $i=0$, the vertex $v_0$ is not adherent,
so $\alpha\le1$ and $\pt(v_0,v_4)\le15\le16-\alpha$.
When $i\ge1$, we have $\pt(v_0,v_4)\le14\le16-\alpha$.

\medskip
\noindent{\bf Case 2.} $j=i+2$.

Suppose first that $v_{i+3}\to v_{i+2}$ is type-II.
If $\mu(v_{i+3})\le3$, then
$\U(v_4)\le\mu(v_{i+3})+1-i\le4$,
and we are done.
Thus, assume that $\mu(v_{i+3})\ge4$.
By \ref{s6-1} and \ref{s6-2},
$v_{i+2}\in L_1\cup R_1$. 
If $v_{i+1}$ and $v_{i+2}$ belong to different
sets $L_1$ and $R_1$, their adjacency puts them
in $F_0$, a contradiction.
If both belong to $R_1$, then
$\mu(v_{i+2})\ge3$ by \ref{s6-2}.
By \Cref{cl:root-threshold},
$v_{i+2}$ has no neighbor in $F_1$,
contrary to $v_{i+1}v_{i+2}\in E$.
If both belong to $L_1$, the same contradiction
follows from $\nu(v_{i+1})\ge3$.
Thus, the last anti-edge is type-I.

If $i=1$, the result follows from \Cref{cl:f-6}.
Hence, assume that $i=0$.
Then $v_0$ is not adherent, and
$\U(v_4)\le\mu(v_3)+1\le7$ by
\Cref{prop:t1edge} and \Cref{lem:good-path-gluing}.
It remains to consider $v_4$ adherent with
$\U(v_4)=7$. 
Then $v_4$ is adherent-I.
Since $\U(v_3)\le6$, by \ref{prop:critical-2},
we have $v_3\in D_a\setminus U_a$.
Moreover, $\mu(v_3)=6$.
The type-I edge $v_3\to v_2$ cannot be oriented
in \Cref{step:DUb}, since its tail belongs to
$D_a\setminus U_a$.
Thus, it is oriented before \Cref{step:L23R23}. 
If $v_1\in R_1$, then
$v_2\in R_1\cup R_2$ and $\nu(v_2)\ge3$,
since $V(L_{st})\cap F_0=\emptyset$.
By \Cref{claimcenter}, we obtain
$\mu(v_3)\le8-\nu(v_2)\le5$,
a contradiction. 
If $v_1\in L_1$, then $\nu(v_1)\ge3$.
By \Cref{cl:root-threshold}, $v_2\notin F_1$,
and so $\mathrm{dist}_G(v_2,F_0)\ge2$.
Applying \Cref{lem:gold-deep-layer} to
$v_3\to v_2$ gives
$\mu(v_3)\le7-\mathrm{dist}_G(v_2,F_0)\le5$,
again a contradiction.

\medskip
\noindent{\bf Case 3.} $j=i+3$.

Then $i=0$ and $j=3$.  
By assumption, the anti-edge $v_4\to v_3$ is type-I. 
The result follows from \Cref{cl:f-6}.
\end{proof}

\begin{lemma}
\label{cl:f-7} 
 Suppose   $v_0v_1$ and $v_3v_4$ are anti-edges and they are type-II, i.e., \ref{Final-7}.  Then $\pt(v_0,v_4)\le 16-\alpha$.     
\end{lemma}
\begin{proof}
Let $H_0=G^{O_{\ref{step:sumatmost8}}}$ and
$H=G^{O_{\ref{step:DUb}}}$.
Since both endpoints are incident to type-II edges,
they belong to $A$ and are not adherent. Thus, $\alpha=0$.
By \Cref{lem:Gxy16}, it suffices to consider
$\nu(v_0)=\D(v_0)=\mu(v_4)=\U(v_4)=8$.

By \ref{s6-1} and \ref{s6-2}, we have
$v_0,v_4\in L_2\cup R_2$ and $v_1,v_3\in L_1\cup R_1$.  
If $v_1$ and $v_3$ belong to different parts, the path
$v_1v_2v_3$ implies that
$v_1,v_3\in L_{11}\cup R_{11}\subseteq F_0$,
a contradiction.
By symmetry, assume that $v_1,v_3\in R_1$.
Then $v_0,v_4\in R_2\cap F_2$,
$v_0\in U_a\setminus D_a$ and
$v_4\in D_a\setminus U_a$.

Let $z\in N(v_0)\cap R_3\cap F_3$ be the vertex
considered when $v_0$ is chosen in \ref{s6-2}.
We first show that $z\in B$.
If $z\in V(H_0)$, then
$\mu_{H_0}(z)+\nu_{H_0}(z)\le9$ by \Cref{claimcenter}.
A shortest corresponding dipath implies that $z$
satisfies (b1) or (b2), contrary to its choice.
Since \Cref{step:L23R23} orients no edge incident to $F_3$,
it follows that $z\in B$.
Moreover, $z\in U_b$ by \Cref{rm:edgeanticross}.

If an edge incident to $z$ is oriented in $H$, then
\ref{S7-1} and \Cref{claim:xinDyinU} give
$v_0\to z$ and $\nu_H(z)\le6$.
Thus, $\nu(v_0)\le7$, a contradiction.
Hence, $z\in U_b\setminus D_b$ and all its incident
edges are unoriented in $H$.
By \Cref{claim:X43F3bdd},
$\D(z)\le8=\D(v_0)$, and so $v_0\to z$ by \ref{S8-21}.
The neighbor $v_0\in U_a$ with $\D(v_0)\ge\D(z)$
excludes $z\in T_U$ by \Cref{lem:exceptional-vertices}.
Thus, $z$ is noncritical and
$7\le\nu(z)\le\D(z)\le8$.
It remains to show that $\pt(z,v_4)\le15$.

Let $L_{zv_4}=u_0u_1\cdots u_m$ be a shortest
$zv_4$-path, where $u_0=z$ and $u_m=v_4$.
Since $z\in F_3$ and $v_4\in F_2$, every path between
them through $F_0$ has length at least five.
Thus, $V(L_{zv_4})\cap F_0=\emptyset$.
If $m\le3$, the result follows from
\Cref{lem:short-main-path}.
Hence, assume that $m=4$.

The edge $u_0u_1$ is unoriented in $H$, and
 $u_1\in(U\setminus D)\cap(F_2\cup F_3\cup F_4)$. 
By \Cref{lem:good-path-gluing},
$7\le\D(u_0)\le\D(u_1)+1\le\nu_H(u_1)+1$.
Consequently, $\nu_H(u_1)\ge6$.  
Moreover, when $u_3u_4$ is not an anti-edge, we have
$u_3\in D\setminus U$.
Indeed, suppose instead that $u_3\in U$.
Then $u_3\to u_4$ in $H$. 
Since $u_4\in D_a\setminus U_a$,
\Cref{rm:edgeanticross} gives $u_3\in U_a\cup M_b$.
Thus, $\mu_H(u_3)\le5$, and so
$\mu(u_4)\le\mu_H(u_4)\le\mu_H(u_3)+1\le6$,
a contradiction.

Suppose that there is an anti-edge, and let
$u_{j+1}\to u_j$ be the last one.
Then $j\in\{1,2,3\}$.

\medskip
\noindent{\bf Case 1.} $j=1$.

If $u_2\to u_1$ is type-I, the bounds proved in
\Cref{prop:t1edge} give
$\mu_H(u_2)\le10-\nu_H(u_1)\le4$.
If it is type-II, then $\nu_H(u_1)\ge6$ and
\ref{s6-1}--\ref{s6-2} give $\mu_H(u_2)\le3$.
Since $u_3\in D\setminus U$, in either case
\Cref{lem:good-path-gluing} gives
$\U(u_4)\le\mu_H(u_2)+2\le6$,
using the crossing arc $u_2\to u_3$ when $u_2\notin D$.
This contradicts $\U(u_4)=8$.

\medskip
\noindent{\bf Case 2.} $j=2$.

Since $u_3\in D\setminus U$, we have
$8=\U(u_4)\le\mu_H(u_3)+1$.
Thus, $\mu_H(u_3)\ge7$, and $u_3\to u_2$ is type-II
by \Cref{prop:t1edge}.
The rules \ref{s6-1}--\ref{s6-2} then give
$u_3\in L_2\cup R_2$ and $u_2\in L_1\cup R_1$.
If $u_2\in L_1$, the path $u_0u_1u_2u$ contradicts
$\mathrm{dist}_G(u_0,u)=4$.
Hence, $u_2\in R_1$ and $u_3\in R_2$.
By \Cref{cl:root-threshold} and (\ref{partitionA}),
$u_2\in D_a\setminus U_a$,
$\mu_H(u_2)\ge3$ and $\nu_H(u_2)\le3$.

The edge $u_1u_2$ is oriented in $H$.
If $u_1\to u_2$, then $\nu_H(u_1)\le4$, a contradiction.
Thus, $u_2\to u_1$.
This arc is neither type-II, by $\mu_H(u_2)\ge3$,
nor oriented in \Cref{step:DUb}, since
$u_2\in D_a\setminus U_a$.
Therefore, it belongs to $H_0$, and \Cref{claimcenter}
gives $\nu_H(u_1)\le8-\mu_H(u_2)\le5$,
again a contradiction.

\medskip
\noindent{\bf Case 3.} $j=3$.

Since $\mu(u_4)=8$, the arc $u_4\to u_3$ is type-II.
As $u_4\in R_2$, \Cref{cl:root-threshold} gives
$u_3\in R_1\cap(D_a\setminus U_a)$,
$N(u_3)\cap F_0=\{v\}$, $N(u_3)\cap F_1=\emptyset$
and $u_3\to v$.
In particular,
$\mu_{H_0}(u_3)\ge3$, $\nu_{H_0}(u_3)\le3$,
and $u_2\in R_2\cap F_2$.

We claim that $u_2\notin U_a$.
Otherwise, (\ref{partitionA}) gives
$\mu_{\ref{step:L23R23}}(u_2)\le4$ and
$\nu_{\ref{step:L23R23}}(u_2)\ge5$.
If $u_2\notin V(H_0)$, then \ref{s6-2} gives
$u_2\to u_3\to v$, a contradiction.
Thus, $u_2\in V(H_0)$.
By monotonicity and \Cref{claimcenter},
$\nu_{H_0}(u_2)\ge5$ and $\mu_{H_0}(u_2)\le4$.

If $u_2u_3$ is unoriented in $H_0$, then
$\mu_{H_0}(u_2)+1+\nu_{H_0}(u_3)\le8$,
contrary to \Cref{step:sumatmost8}.
Its direction cannot be $u_2\to u_3$, since this gives
$\nu_{H_0}(u_2)\le4$.
Hence, $u_3\to u_2$ belongs to $H_0$ and has weight
at least $3+1+5=9$.
By \Cref{claimcenter}, its weight is nine, and it belongs
to a dicycle $C=u_3u_2w_1w_2w_3u_3$ of length five
meeting $F_0$.
Since $w_3\to u_3$ and $u_3\to v$, we have $w_3\ne v$.
The properties of $u_3$ therefore give $w_3\in F_2$.
As $u_2,w_3\in F_2$, neither $w_1$ nor $w_2$ belongs
to $F_0$, a contradiction.
This proves the claim.

By \Cref{rm:edgeanticross}, it follows that $u_2\in D$.
Thus, $u_1u_2$ is oriented in $H$.
Suppose that $u_2\to u_1$.
This arc is not type-II, since neither endpoint belongs
to $L_1\cup R_1$.
It is not oriented in \Cref{step:DUb} either, since
$u_1\in U\setminus D$ and $u_2\notin U_a$.
Thus, it belongs to $H_0$, and
\Cref{lem:gold-deep-layer} gives
$\nu_H(u_1)\le7-\mathrm{dist}_G(u_2,F_0)=5$,
a contradiction.

Hence, $u_1\to u_2$.
If $u_2\in D_a$, then $\nu_H(u_2)\le4$.
If $u_2\in D_b$, then \ref{S7-1} and
\Cref{claim:xinDyinU}(2) give $u_2\to u_3$, and again
$\nu_H(u_2)\le1+\nu_H(u_3)\le4$.
Consequently, $\nu_H(u_1)\le5$, a contradiction.

\medskip

Therefore, $L_{zv_4}$ has no anti-edge.
Recall that $u_3\in D\setminus U$.
If $u_2\in U$, the  arc $u_2\to u_3$ gives
$\mu_H(u_3)\le6$ and hence $\U(u_4)\le7$.
Thus, $u_2\in D\setminus U$.
If $u_1\in U_a$, then
$\U(u_4)\le\mu_H(u_1)+3\le7$, again a contradiction.
Hence, $u_1\in U_b\setminus D_b$, and
\Cref{rm:edgeanticross} gives $u_2\in D_b\setminus U_b$. 
By \Cref{claim:xinDyinU}, the  arc $u_1\to u_2$
gives $\nu_H(u_1),\mu_H(u_2)\le6$.
Thus,
$7\le\D(u_0)\le\D(u_1)+1\le\nu_H(u_1)+1\le7$
and
$8=\U(u_4)\le\U(u_3)+1\le\U(u_2)+2
\le\mu_H(u_2)+2\le8$.
It follows that
$\D(u_0)=7$, $\D(u_1)=6$, $\U(u_2)=6$ and $\U(u_3)=7$.
By \ref{S8-41}, we have $u_0\to u_1$.
Since $u_2\in D_b$ and $u_4\in D_a$,
\ref{S8-11} and \ref{S8-31} give
$u_2\to u_3\to u_4$, according to whether
$u_3\in D_a$ or $u_3\in D_b$.
Here any edge already oriented in $H$ has the required
direction, since there is no anti-edge.
Therefore, $u_0\to u_1\to u_2\to u_3\to u_4$ is a dipath.

In all cases, $\pt(z,v_4)\le15$.
Since $v_0\to z$, we obtain $\pt(v_0,v_4)\le16$,
as required.
\end{proof}
\begin{remark}\label{rm:f-7}
The conclusion of \Cref{cl:f-7} also holds when
$L_{st}$ meets $F_0$.
Indeed, the proof establishes $v_1,v_3\notin F_0$
from the type-II rules, without this assumption.
The auxiliary shortest path from $z\in F_3$ to
$v_4\in F_2$ is disjoint from $F_0$, since a path
between them through $F_0$ has length at least five.
Thus, the rest of the proof applies without change.
\end{remark}

If $V(L_{st})\cap F_0=\emptyset$, then $\pt(s,t)\le 16-\alpha$ by the above lemmas. 

\subsection{\texorpdfstring{Paths  intersects $F_0$}{Paths intersecting F0}}\label{sec:core-near-endpoint}
\noindent
Having settled the paths that avoid $F_0$, we turn to the case $v_1\in F_0$ or $v_3\in F_0$ first, taking $v_1\in F_0$ by symmetry.

\begin{lemma}\label{cl:f2-1}
Suppose that  $v_0\in R_{12}\cap F_1$ with $\nu(v_0)=7$ and $\mu(v_4)=\U(v_4)=8$ with  $v_4$ being adherent, i.e., \ref{Final2-1}. Then $\pt(v_0,v_4)\le 16-\alpha$.  
\end{lemma} 
\begin{proof}
Let $H_0=G^{O_{\ref{step:sumatmost8}}}$ and
$H=G^{O_{\ref{step:DUb}}}$.
Since $v_0\in R_{12}$, choose $z\in N(v_0)\cap R_{21}$.
By \Cref{step:F0X31X41}, we have $v_0\to z$,
$\mu_{H_0}(z)\le2$ and $\nu_{H_0}(z)\le6$.
Since $\nu(v_0)=7$, it follows that
$\nu(z)=\nu_H(z)=\nu_{H_0}(z)=6$.
The initial rules then give
$N(z)\cap(L_2\cup X_2\cup X_3)=\emptyset$,
and hence $N(z)\subseteq R$.
By (\ref{partitionA}) and \Cref{lem:saturation-stability},
we have $z\in U_a\setminus D_a$ and $\D(z)=6$.

We first observe that every arc $w\to z$ of $H$
belongs to $H_0$.
Indeed, it cannot be oriented in \Cref{step:DUb},
since $z\in U_a\setminus D_a$.
It cannot be type-II either, since every edge joining
$z\in R_{21}$ to $R_1$ is oriented initially.
By \Cref{claimcenter} and \Cref{lem:gold-deep-layer}, such an arc
satisfies $\mu_{H_0}(w)\le2$ and
$\mathrm{dist}_G(w,F_0)\le1$.

Let $L_{zv_4}=u_0u_1\cdots u_m$ be a shortest
$zv_4$-path, where $u_0=z$ and $u_m=v_4$.
Since $v_4$ is adherent, all its incident edges are
unoriented in $H$.
By \Cref{claim:xinDyinU}, its neighbors belong to
$D\setminus U$, so $m\ge2$.
Put $r=u_{m-1}$.
By \Cref{lem:good-path-gluing},
$8=\U(v_4)\le\U(r)+1$, and hence
$\mu_H(r)\ge\U(r)\ge7$.

\medskip
\noindent{\bf Case 1.}
$V(L_{zv_4})\cap F_0\ne\emptyset$.

Let $q$ be the last vertex of $L_{zv_4}$ in $F_0$.
Since $v_4\notin F_0\cup F_1$, at least two edges
follow $q$.

Suppose first that $q$ immediately precedes $r$.
Then $\mathrm{dist}_G(z,q)\le2$ and $q\ne u$.
We have $\mu_H(q)\le2$ unless $q\in L_{11}$.
In the latter case, the initial subpath has form $zwq$.
Since $N(z)\subseteq R$, the edge $wq$ gives
$w\in R_{11}$.
Thus, $v\to w\to q$ also gives $\mu_H(q)\le2$.

The edge $qr$ is oriented in $H_0$.
If $q\to r$, then $\mu_H(r)\le3$, a contradiction.
Suppose that $r\to q$.
If its weight in $H_0$ is at most eight, then
$\nu_{H_0}(q)\ge1$ gives $\mu_H(r)\le6$.
Otherwise, its dicycle of length five gives
$\mu_H(r)\le\mu_H(q)+4\le6$.
Both contradict $\mu_H(r)\ge7$.

Thus, $L_{zv_4}=zqart$, where $t=v_4$,
$q\in R_{11}$ and $a\in F_1\cap(R_1\cup R_2)$.
We have $\mu_H(a)\le5$.
Indeed, $q\to a$ gives $\mu_H(a)\le2$.
If $a\to q$ has weight at most eight in $H_0$,
then $\nu_{H_0}(q)\ge2$ gives $\mu_H(a)\le5$.
If its weight is nine, its dicycle of length five
gives the same bound.

If $ar$ is not an anti-edge, then $a\in D$ gives
$\U(r)\le\mu_H(a)+1\le6$.
If $a\notin D$, the crossing arc $a\to r$ gives
$\mu_H(r)\le6$.
Thus, $r\to a$. 
If this arc belongs to $H_0$, then
$\nu_{H_0}(a)\ge2$ gives $\mu_H(r)\le6$.
If it is oriented in \Cref{step:DUb}, its tail
$r\in D\setminus U$ belongs to $D_b$, and
\Cref{claim:xinDyinU} again gives $\mu_H(r)\le6$.
Hence, $r\to a$ is type-II.
Since $a\in F_1$, we have $a\in R_1$.
By \Cref{cl:root-threshold},
$N(a)\cap F_0=\{v\}$, contrary to $aq\in E$
with $q\in R_{11}$.

\medskip

Therefore, assume that $V(L_{zv_4})\cap F_0=\emptyset$.

\medskip
\noindent{\bf Case 2.}
$u_{m-2}r$ is an anti-edge.

Since $\mu_H(r)\ge7$, this anti-edge is type-II
by \Cref{prop:t1edge}.
The rules \ref{s6-1}--\ref{s6-2} imply that
$r\in L_2\cup R_2$ and $u_{m-2}\in L_1\cup R_1$.

If $m=2$, this contradicts $u_0=z\in R_2$.
If $m=3$, then $u_1\in N(z)\cap R_1= N(z)\cap R_{12}$,
since the path avoids $F_0$.
Thus, $\mu_{H_0}(u_1)\le1$, contrary to
the condition for $r\to u_1$ in \ref{s6-2}.
Hence, $m=4$.
Since $u_1\in N(z)\subseteq R$, the possibility
$u_2\in L_1$ would give $u_1\in R_{11}$,
a contradiction.
Thus, $u_2\in R_1$.
By \Cref{cl:root-threshold} and (\ref{partitionA}),
we have $u_2\in D_a\setminus U_a$,
$\mu_{H_0}(u_2)\ge3$, $u_2\to v$,
$N(u_2)\cap F_0=\{v\}$ and
$N(u_2)\cap F_1=\emptyset$.
In particular, $u_1\in R_2\cap F_2$. 
By the observation that  every arc $w\to z$ of $H$
belongs to $H_0$, $u_1\to z$ is impossible.
Thus, $zu_1$ is not an anti-edge.
If $u_1\notin U$, then $u_1\in D\setminus U$,
and \Cref{rm:edgeanticross} gives $u_1\in D_a$.
The crossing arc $z\to u_1$ then gives
$\nu_H(z)\le5$, a contradiction.
Hence, $u_1\in U$. 

The edge $u_1u_2$ is oriented in $H$.
If $u_1\to u_2$, then $\nu_H(u_1)\le4$.
Suppose that $u_2\to u_1$.
This arc is neither type-II nor oriented in
\Cref{step:DUb}, and hence belongs to $H_0$.
Its weight cannot be nine.
Indeed, in its specified dicycle of length five,
the outneighbor of $u_2$ is $u_1\in F_2$,
and its inneighbor also belongs to $F_2$,
since $u_2\to v$.
The whole cycle would therefore avoid $F_0$,
contrary to \Cref{claimcenter}.
Thus, its weight is at most eight, and
$\mu_{H_0}(u_2)\ge3$ again gives $\nu_H(u_1)\le4$.
In either case, \Cref{lem:good-path-gluing} gives
$\D(z)\le\nu_H(u_1)+1\le5$, a contradiction.

\medskip

Thus, $u_{m-2}r$ is not an anti-edge.
Since $\mu_H(r)\ge7$,  by
\Cref{claim:xinDyinU}, we know that  $u_{m-2}\in D\setminus U$.
In particular, $m\ne2$.
If $m=3$, the edge $zu_1$ is oriented in $H$.
The direction $z\to u_1$ gives $\mu_H(u_1)\le3$,
while $u_1\to z$ gives $\mu_H(u_1)\le2$ since    every arc $w\to z$ of $H$
belongs to $H_0$ and $\nu(z)= 6$.
Thus, $\U(v_4)\le\mu_H(u_1)+2\le5$,
a contradiction.
Hence, $m=4$ and $u_2,u_3\in D\setminus U$.

Suppose that $u_1u_2$ is not an anti-edge.
If $u_1\in D$, the same argument gives $\mu_H(u_1)\le3$.
If $u_1\in U_a$, then $\mu_H(u_1)\le4$.
In either case, \Cref{lem:good-path-gluing} gives
$\U(u_4)\le\mu_H(u_1)+3\le7$.
Otherwise, $u_1\in U_b\setminus D_b$.
By \Cref{rm:edgeanticross},
$u_2\in D_b\setminus U_b$.
Then \ref{S7-2} gives $z\to u_1\to u_2$,
so $\U(u_4)\le\mu_H(z)+4\le6$.
Both are contradictions.

Therefore, $u_2\to u_1$ is an anti-edge. Note that 
$8=\U(u_4)\le\mu_H(u_2)+2$, so $\mu_H(u_2)\ge6$.
If this arc is type-II, its head must be an
$R_1$-neighbor of $z$, and the same argument as in
Case 2 gives a contradiction.
Thus, it is type-I.
By \Cref{prop:t1edge},
$\mu_H(u_2)=6$ and $\nu_H(u_1)\le4$.

If $zu_1$ is not an anti-edge and $u_1\in U$,
then $\D(z)\le\nu_H(u_1)+1\le5$.
If $u_1\notin U$, then
\Cref{rm:edgeanticross} gives $u_1\in D_a$,
and $z\to u_1$ gives $\nu_H(z)\le5$.
Thus, $u_1\to z$.

By the observation above, $u_1\to z$ belongs to $H_0$
and $\mu_{H_0}(u_1)\le2$.
By \Cref{lem:gold-deep-layer}, it holds that
$6=\nu_{H_0}(z)\le7-\mathrm{dist}_G(u_1,F_0)$.
Since $u_1\notin F_0$, we obtain $u_1\in F_1$. 
Since $u_1\in N(z)\subseteq R$,
we have $u_1\in R_1\cup R_2$.
If $u_2\to u_1$ belongs to $H_0$, its weight gives
$\nu_{H_0}(u_1)\le2$.
Together with $\mu_{H_0}(u_1)\le2$,
this contradicts $u_1\notin F_0$ by \Cref{obs:munu}.
Hence, $u_2\to u_1$ is oriented in \Cref{step:DUb}.
Its rules give $u_2\in D_b\setminus M_b$ and
$u_1\in D_a$.

If $u_1\in R_1$, then
$\mu_{\ref{step:L23R23}}(u_1)\le2$ implies
$u_1\in A_1$, contrary to $u_1\in D_a$.
Thus, $u_1\in R_2\cap F_1$, and
(\ref{partitionA}) gives $u_1\in M_a$.
Consequently, $u_2\in C=M'_b\setminus M_b$ and
$N(u_2)\cap A=\{u_1\}$.
In particular, $u_3\in D_b\setminus U_b$.

Since $\U(u_3)\ge7$, \Cref{claim:xinDyinU}
implies that no edge incident to $u_3$ is oriented in $H$.
By \Cref{rm:crossing}, $u_3\notin M'_b$,
and hence $u_3\in B_D$.
Moreover, a neighbor of $u_1\in R_2$ outside $R$
belongs to $L_{21}\cup X_2\cup X_3\subseteq A$.
Since $u_2\in B$, it follows that $u_2\in R$.
Now $N(u_2)\cap A=\{u_1\}\subseteq R$ and
$u_3\in N(u_2)\cap B_D$.
By (\ref{partitionM}), we obtain
$u_2\in M_\ell\subseteq U_b$,
contrary to $u_2\in D_b\setminus M_b$.
This completes the proof.
\end{proof}

\begin{lemma}\label{cl:f2-2}
 Suppose that  $\nu(v_0)=7$, $v_1\in F_0$, and $\mu(v_4)=8$ with  $v_4$ being adherent.  Moreover,   there is $v_0'\in N(v_0)\cap F_0$ such that $v_0'\in R_{11}$, $v_0'v_0$ is oriented in \Cref{step:cycle(a)} and  $\omega(v_0'\to v_0)=9$, i.e., \ref{Final2-2}. Then $\pt(v_0,v_4)\le 16-\alpha$. 
\end{lemma}
\begin{proof}
Let $H_0=G^{O_{\ref{step:sumatmost8}}}$ and
$H=G^{O_{\ref{step:DUb}}}$.
All edge weights below are computed in $H_0$.

Since $v_0'\in R_{11}$, a noninteger value of $f(v_0)$
would give $\nu(v_0)\le4$ by \ref{item:F121}.
Thus, $f(v_0)=4$, and so $v_1\in R_{11}$.

Since $v_4$ is adherent and $\mu(v_4)=8$,
\Cref{claimcenter4} and \ref{prop:critical-1} give
$\U(v_4)=8$, and all its incident edges are unoriented in $H$.
By \Cref{claim:xinDyinU} and \Cref{lem:good-path-gluing},
we have $v_3\in D\setminus U$ and
$\mu_H(v_3)\ge\U(v_3)\ge7$.

Suppose first that $v_2\in F_0$.
Since $v_1\in R_{11}$ and $v_1v_2\in E$,
the initial rules give $\mu_H(v_2)\le2$.
Also, $v_2\ne u$, so $\nu_{H_0}(v_2)\ge1$.
The edge $v_2v_3$ belongs to $H_0$.
If $v_2\to v_3$, then $\mu_H(v_3)\le3$, a contradiction.
Thus, $v_3\to v_2$.
If its weight is at most eight, then $\mu_H(v_3)\le6$.
If its weight is nine, \Cref{claimcenter} gives a
dicycle of length five containing $v_3\to v_2$.
Its directed subpath from $v_2$ to $v_3$ has length four,
so $\mu_H(v_3)\le\mu_H(v_2)+4\le6$.
Both are contradictions.

Hence, $v_2\in F_1\cap(R_1\cup R_2)$.
We first show that $\mu_H(v_2)\le5$.
The edge $v_1v_2$ belongs to $H_0$.
If $v_1\to v_2$, then $\mu_H(v_2)\le2$.
If $v_2\to v_1$ has weight at most eight, then
$\nu_{H_0}(v_1)\ge2$ gives $\mu_H(v_2)\le5$.
If its weight is nine, its dicycle of length five gives
$\mu_H(v_2)\le\mu_H(v_1)+4\le5$.

Suppose that $v_2v_3$ is not an anti-edge.
If $v_2\in D$, then \Cref{lem:good-path-gluing} gives
$\U(v_3)\le\mu_H(v_2)+1\le6$.
Otherwise, the crossing arc $v_2\to v_3$ gives
$\mu_H(v_3)\le\mu_H(v_2)+1\le6$.
Thus, $v_3\to v_2$ is an anti-edge.

If it is type-I, then $\mu_H(v_3)\le6$ by
\Cref{prop:t1edge}, a contradiction.
Hence, it is type-II.
Since $v_2\in R_1\cup R_2$ and $\mu_H(v_3)\ge7$,
\ref{s6-2} forces $v_2\in R_1$ and $v_3\in R_2$.
But $v_1\in R_{11}$ and $v_1v_2\in E$ give
$f(v_2)=4.5$.
Thus, \ref{item:F121} already gives $v\to v_2$,
so \ref{s6-2} cannot orient $v_3\to v_2$.
This contradiction completes the proof.
\end{proof}

 In below, we show the lemmas related to the case that    $v_2\in F_0$ with 
  $v_1\notin F_0$ and $v_3\notin F_0$. Note that  $v_0,v_4\in F_2$  and $\orw{v_1v_2},\orw{v_2v_3}\in E(G^{O_{\ref{step:F13}}})$ since they are of form $[F_0,F_1]$.  

\begin{lemma}\label{cl:f3-1}
Suppose that $v_2\in F_0$ with 
 $v_0\in F_2$, $v_1\notin F_0$, $v_3\notin F_0$ and $v_4\in F_2$. If $\orw{v_0v_1}\in E(G^{O_{\ref{step:DUb}}})$ and it is  type-I or  $\orw{v_3v_4}\in E(G^{O_{\ref{step:DUb}}})$ and it is  type-I, i.e., \ref{Final3-1}. Then $\pt(v_0,v_4)\le 16-\alpha$. 
\end{lemma}
\begin{proof} 
By symmetry, assume that $\orw{v_0v_1}$ is type-I.
Then $v_0$ is not adherent by \ref{prop:critical-1}.
If this edge is oriented before \Cref{step:L23R23},
then $v_0\in F_2$ and \Cref{lem:gold-deep-layer} give
$\nu(v_0)\le6$.

Thus, assume that it is oriented in \Cref{step:DUb}.
If $v_1\to v_0$, then $\nu(v_0)\le6$ by
\Cref{prop:t1edge}.
If $v_0\to v_1$, then $v_1\in F_1\subseteq A$.
Since \Cref{step:DUb} orients no edge with head in
$U_a\setminus D_a$, it follows that $v_1\in D_a$.
Thus, $\nu(v_0)\le1+\nu(v_1)\le5$.
In every case, $\nu(v_0)\le6$ and $\alpha\le1$.
Since $\mu(v_4)\le8$, by \Cref{lem:Gxy16}, we have
$\pt(v_0,v_4)\le6+1+8=15\le16-\alpha$.
 \end{proof}

\begin{lemma}\label{cl:f3-2}
 Suppose that $v_2\in F_0$ with 
 $v_0\in F_2$, $v_1\notin F_0$, $v_3\notin F_0$ and $v_4\in F_2$. If both  $\orw{v_0v_1},\orw{v_3v_4}\in E(G^{O_{\ref{step:DUb}}})$  and they are type-II, i.e., \ref{Final3-2}. Then $\pt(v_0,v_4)\le 16-\alpha$. 
\end{lemma} 
\begin{proof}
Since $v_0,v_4\in A$, neither is adherent, and so $\alpha=0$.

If $v_0\to v_1$ or $v_3\to v_4$, then
$\nu(v_0)\le4$ or $\mu(v_4)\le4$, respectively, by
\Cref{step:L23R23} and \Cref{cl:root-threshold}.
Since $\nu(v_0),\mu(v_4)\le8$, it follows that
$\pt(v_0,v_4)\le13$ by \Cref{lem:Gxy16}. 
Thus, assume that $v_1\to v_0$ and $v_4\to v_3$.
Then the first and last edges of $L_{st}$ are type-II
anti-edges.
By \Cref{rm:f-7}, we may apply \Cref{cl:f-7} to $L_{st}$.
Hence, $\pt(v_0,v_4)\le16=16-\alpha$.
\end{proof}

\begin{lemma}\label{cl:f3-3}
Suppose that $v_2\in F_0$ with 
 $v_0\in F_2$, $v_1\notin F_0$, $v_3\notin F_0$ and $v_4\in F_2$. If  both   $v_0v_1$ and $v_3v_4$ are unoriented after \Cref{step:DUb}, i.e., \ref{Final3-3}. Then $\pt(v_0,v_4)\le 16-\alpha$. 
\end{lemma}
\begin{proof}
Let $H=G^{O_{\ref{step:DUb}}}$.
Since $v_1,v_3\in F_1$ and $v_0v_1,v_3v_4$ are
unoriented in $H$, by \Cref{claim:xinDyinU}, we have
$v_1\in U_a\setminus D_a$ and $v_3\in D_a\setminus U_a$.
By \Cref{lem:good-path-gluing}, it holds that
$\D(v_0)\le\nu_H(v_1)+1$ and
$\U(v_4)\le\mu_H(v_3)+1$.

Recall that $\D(v_0),\U(v_4)\le8$.
If $\D(v_0)+\U(v_4)\le13$, then
$\pt(v_0,v_4)\le14\le16-\alpha$ by \Cref{lem:Gxy16}.
Thus, assume that $\D(v_0)+\U(v_4)\ge14$.
It follows that $\nu_H(v_1),\mu_H(v_3)\ge5$.
Since $\mu_H(v_2),\nu_H(v_2)\le3$ and both
$v_1v_2,v_2v_3$ are already oriented, we have
$v_2\to v_1$ and $v_3\to v_2$.

Suppose that $v_2=u$.
Then $v_1\in L_1$, and \Cref{cl:root-threshold} gives
$v_1\in F_{12}$.
Thus, $uv_1$ is oriented in \Cref{step:cycle(auv)}.
By \Cref{prop:cycleauv} and monotonicity, we obtain
$2+1+\nu_H(v_1)\le7$, a contradiction.
The case $v_2=v$ is symmetric.
Hence, $v_2\in L_{11}\cup X_2\cup R_{11}$, and so
$3\le\mu_H(v_2)+\nu_H(v_2)\le4$.

Let $w_1=\omega_H(v_2\to v_1)$ and
$w_2=\omega_H(v_3\to v_2)$.
By \Cref{claimcenter} and monotonicity, $w_1,w_2\le9$,
with equality only if the corresponding arc lies on
a dicycle of length five.
If $w_1,w_2\le8$, then
$\nu_H(v_1)+\mu_H(v_3)
=w_1+w_2-2-\mu_H(v_2)-\nu_H(v_2)\le11$,
contrary to $\D(v_0)+\U(v_4)\ge14$.
Thus, at least one of $w_1,w_2$ equals nine.
The other four edges of its dicycle give
$9\le\mu_H(v_2)+\nu_H(v_2)+5\le9$.
Consequently, $\mu_H(v_2)+\nu_H(v_2)=4$, and
$\nu_H(v_1)+\mu_H(v_3)=w_1+w_2-6\le12$.

It follows that $\D(v_0)+\U(v_4)=14$,
$w_1=w_2=9$,
$\D(v_0)=\nu_H(v_1)+1$ and
$\U(v_4)=\mu_H(v_3)+1$.
If $\alpha\le1$, then
$\pt(v_0,v_4)\le15\le16-\alpha$.
Thus, assume that both endpoints are adherent.
Since $\D(v_1)\le\nu_H(v_1)<\D(v_0)$ and
$\U(v_3)\le\mu_H(v_3)<\U(v_4)$, by \ref{lem:gold-adherent-edge}, we know that 
$v_0\to v_1$ and $v_3\to v_4$. 
The dicycles containing $v_2\to v_1$ and $v_3\to v_2$
give a $(v_1,v_2)$-dipath and a $(v_2,v_3)$-dipath,
each of length four.
Together with $v_0\to v_1$ and $v_3\to v_4$,
they form a $(v_0,v_4)$-diwalk of length ten.
Therefore, $\pt(v_0,v_4)\le10<16-\alpha$.
\end{proof}

 \begin{lemma}\label{cl:f3-4}
Suppose that $v_2\in F_0$ with 
 $v_0\in F_2$, $v_1\notin F_0$, $v_3\notin F_0$ and $v_4\in F_2$. If one of $v_0v_1$ and $v_3v_4$ is type-II and the other one is unoriented in $G^{O_{\ref{step:DUb}}}$, i.e., \ref{Final3-4}. Then $\pt(v_0,v_4)\le 16-\alpha$. 
  \end{lemma}
 \begin{proof}
Let $H_0=G^{O_{\ref{step:sumatmost8}}}$ and
$H=G^{O_{\ref{step:DUb}}}$.
The parameters without subscripts are computed in
$G^{O_{\ref{step:complicate}}}$.

By symmetry, assume that $v_0v_1$ is unoriented in $H$
and $v_3v_4$ is type-II.
Then $v_4\in A$ is not adherent, and so $\alpha\le1$.
By \Cref{lem:Gxy16}, we are done if $\D(v_0)\le6$,
or if $\D(v_0)\le7$ and $v_0$ is not adherent.

If $v_3\to v_4$, then $\mu_H(v_4)\le4$ by
\Cref{step:L23R23}, and the result follows.
Thus, assume that $v_4\to v_3$.
Since $v_0v_1$ is unoriented in $H$, we have
$v_1\in U\setminus D$ and
$\D(v_0)\le\nu_H(v_1)+1$.

Suppose that $v_4\in L_2$.
Then $v_3\in L_1\setminus L_{11}$ and
$v_2\in\{u\}\cup L_{11}$.
If $v_1\to v_2$, then $\nu_H(v_1)\le2$.
If $v_2\to v_1$ has weight at most eight in $H_0$,
then $\nu_{H_0}(v_1)\le5$, since $\mu_{H_0}(v_2)\ge2$.
For weight nine, the dicycle supplied by
\Cref{claimcenter} gives
$\nu_{H_0}(v_1)\le\nu_{H_0}(v_2)+4\le5$.
Thus, $\D(v_0)\le6$.

Hence, assume that $v_4\in R_2$.
By \Cref{step:L23R23} and \Cref{cl:root-threshold}, 
$v_3\in R_1$, $N(v_3)\cap F_0=\{v\}$ and $v_3\to v$.
Consequently, $v_2=v$ and $v_1\in R_1\setminus R_{11}$.

We may assume that $\nu_H(v_1)\ge6$.
By the proof of \Cref{calim:boundmunuV-GO1}(1),
either $v_1$ is incident to an edge oriented in
\Cref{step:F0X31X41}, or $v_1\in F_{12}$,
$\nu_H(v_1)=6$, and $v\to v_1$ lies on a dicycle
of length five through $v$.
In the latter case, $\D(v_0)\le7$.
Thus, we only need to consider an adherent $v_0$
with $\D(v_0)=7$.
Since $\D(v_1)\le6$, by \ref{lem:gold-adherent-edge}, it holds that $v_0\to v_1$.
Moreover, $\partial(v_1,v)\le4$.
Since $\partial(v,u)\le4$ and $\mu(v_4)\le8$,
the definition of $\mu$ gives $\partial(v,v_4)\le10$.
Therefore, $\pt(v_0,v_4)\le1+4+10=15$.

Thus, assume that $v_1$ is incident to an edge oriented
in \Cref{step:F0X31X41}.
Then $v\to v_1$ is oriented in that step.
Since $\nu_H(v_1)\ge6$, the initial rules give
$L_{21}\ne\emptyset$.
Also, $v_0\in R_{23}\cap F_2$: a vertex of $R_{22}$
belongs to $F_1$, while all its $R_1$-edges would already
be oriented if $v_0\in R_{21}$.

If $v_0\in V(H_0)$, then $\nu_{H_0}(v_0)\le6$ by
\Cref{lem:gold-deep-layer}, and we are done.
If $v_0$ were first incident to an oriented edge in
\Cref{step:L23R23}, that step would also orient $v_0v_1$.
Hence, $v_0\in B$.
If an edge incident to $v_0$ is oriented in
\Cref{step:DUb}, then $\nu_H(v_0)\le6$ by
\Cref{claim:xinDyinU}.
Therefore, every edge incident to $v_0$ is unoriented in $H$.
By \Cref{rm:crossing}, it follows that $v_0\in B_U$.

Let $u_0u_1\cdots u_m$ be a shortest
$(v_0,L_{21})$-path, where $u_0=v_0$.
Then $2\le m\le4$ and $u_1\in U\setminus D$.
Since $v_0\in R_{23}\cap F_2$, we also have $u_1\in R$.
If $m=2$, then $u_1\in R_{21}$ and
$\nu_{H_0}(u_1)\le3$, giving $\D(v_0)\le4$.
Thus, $m\in\{3,4\}$.

\medskip
\noindent{\bf Case 1.} $u_1\in R_1$.

We show that $\nu_H(u_1)\le5$.
If $m=3$, then $u_2\in R_{21}$ and the initial rules
give $u_1\to u_2\to u_3$, so $\nu_H(u_1)\le4$.
Thus, assume that $m=4$.

We have $u_2\ne v$.
If $u_2\in R_{11}$, then
$v\to u_1\to u_2$ by \ref{item:F121}, giving
$\nu_H(u_1)\le4$.
Also, $u_2\notin L_1\cup X_2$, since otherwise
$u_1\in R_{11}$.
Hence, $u_2\in(R_1\cup R_2)\setminus R_{11}$.

If $u_2\in R_1$, then $u_3\in R_{21}$ and
$u_2\to u_3\to u_4$ is oriented initially.
If $u_1$ has a neighbor in
$R_{11}\cup(R_{21}\cap N(L_2))$, then
$\nu_H(u_1)\le4$.
Otherwise, \ref{item:step:F0X31X415} gives
$u_1\to u_2$, and so $\nu_H(u_1)\le5$.

Suppose that $u_2\in R_2$.
If $u_2$ has a neighbor in $L_2\cup X_3$, then
$u_2\in R_{21}$, $u_1\to u_2$ is oriented initially,
and $\nu_H(u_2)\le4$.
If $u_2$ has a neighbor in $X_2$, then
$u_1\to u_2\to X_2$, giving $\nu_H(u_1)\le4$.
Indeed, an initially processed such vertex belongs to
$R_{21}$.
Otherwise, it has either $f(u_2)=3.5$, or $f(u_2)=3$
and a neighbor $u_1\in F_1$ with larger $f$-value.
In the latter case, a neighbor with smaller $f$-value
would put $u_2$ in $R_{21}$.
Thus, the asserted directions follow from
\ref{item:F121}, \ref{item:F12222} and \ref{item:F123}.

Finally, if
$N(u_2)\cap(L_2\cup X_2\cup X_3)=\emptyset$,
then $u_3\in R_{21}$, and
\ref{item:step:F0X31X414} gives
$u_1\to u_2\to u_3\to u_4$.
Again, $\nu_H(u_1)\le5$.
Consequently, $\D(v_0)\le6$.

\medskip
\noindent{\bf Case 2.} $u_1\in R_2\cup R_3$.

It suffices to prove $\D(u_1)\le6$.
Indeed, this gives $\D(v_0)\le7$.
If $v_0$ were adherent with $\D(v_0)=7$, it would
be adherent-I.
By \ref{prop:critical-2}, the inequality
$\D(u_1)<\D(v_0)$ would force $u_1\in U_a$.
But $v_1\in U_a$ is another neighbor of $v_0$,
contrary to the uniqueness of its $U_a$-neighbor.

If $u_1\in M'_b$, then $\nu_H(u_1)\le6$ by
\Cref{rm:crossing} and \Cref{claim:xinDyinU}.
Thus, assume that $u_1\in U_a\cup B_U$.

We first consider $u_1\in V(H_0)$.
If $u_1\notin F_1$, then
$\nu_{H_0}(u_1)\le6$ by \Cref{lem:gold-deep-layer}.
Hence, assume that $u_1\in R_2\cap F_1$.
The initial rules and
\Cref{prop:F118,prop:cycleauv,prop:F12224}
give $\nu_{H_0}(u_1)\le6$ when $u_1$ is initially
processed or belongs to $F_{11}$.
The same bound holds if $u_1$ has a neighbor in $X_2$,
by the argument in Case 1.

Thus, $N(u_1)\cap F_0\subseteq R_{11}$.
If $u_1\in F_{12}$ and
$N(u_1)\cap F_0=\{r\}$, then
$vv_1v_0u_1rv$ is a cycle of length five through $v$.
The choice of the considered cycle and
\Cref{prop:cycleauv} give $\nu_{H_0}(u_1)\le5$.

Suppose that $u_1\in F_{13}$.
Then $u_1\in R_{22}$ and $f(u_1)=4$.
Consider the first operation orienting an edge in
$[u_1,R_{11}]$.
No edge incident to $u_1$ is oriented before this operation:
an earlier operation involving $u_1$ would already
require or orient an edge in $[u_1,R_{11}]$. If the first operation is in \Cref{step:F13}, then it gives
$u_1\to r$ for some $r\in R_{11}$. Hence,
$\nu_{H_0}(u_1)\le1+\nu_{H_0}(r)\le4$.
Thus, assume that the first operation is before \Cref{step:F13}.

An operation in \Cref{step:cycle(auv)} gives
$\nu_{H_0}(u_1)\le5$.
A whole cycle of length at most four gives
$\nu_{H_0}(u_1)\le6$.
A four-edge path of form (d) gives a dipath of length
at most three from $u_1$ to $F_0$, and hence the same bound.
A shorter path, a four-edge path of form (a), or an
additional edge would require an already oriented edge
in $[u_1,R_{11}]$.

The remaining possibility is a whole cycle of length
five of form (a).
Since $u_1\in R_{22}$, this cycle is contained in
$F_0\cup R$.
The direction at $u_1$ is being chosen for the first time,
and the direction $u_1\to R_{11}$ gives $\nu(u_1)\le4$.
Thus, \ref{item:cycle5-whole} gives
$\nu_{H_0}(u_1)\le4$.
This completes the case $u_1\in V(H_0)$.

Hence, assume that $u_1\notin V(H_0)$.
Suppose that $u_2\to u_1$ is anti.
If it is type-I, it must be oriented in
\Cref{step:DUb}, and $\nu_H(u_1)\le6$.
If it is type-II, then $u_1\in R_2$ and $u_2\in R_1$.
Since $\mathrm{dist}(u_2,L_{21})\le2$, we have
$m=4$, $u_2\in R_{12}$ and $u_3\in R_{21}$.
The initial dipath $u_2\to u_3\to u_4$ gives
$\nu_{H_0}(u_2)\le4$.

Now, all edges of $v_1v_0u_1u_2$ are unoriented in $H_0$.
This is a path: $v_1\ne u_1$ by their locations, and
$v_1\ne u_2$ by the choice of $u_0\cdots u_4$.
Since $\mu_{H_0}(v_1)=1$, we have
$\mu_{H_0}(v_1)+3+\nu_{H_0}(u_2)\le8$,
contrary to \Cref{step:sumatmost8}.

Therefore, $u_1u_2$ is not anti.
If $u_2\in D$, the crossing arc $u_1\to u_2$
gives $\nu_H(u_1)\le6$.
Thus, assume that $u_2\in U\setminus D$.
It now suffices to prove $\D(u_2)\le5$.

If $m=3$, then
$u_2\in R_{21}\cup X_2\cup X_3\cup L_{31}$,
and its adjacency to $u_3\in L_{21}$ gives
$\nu_{H_0}(u_2)\le3$ by the initial rules.
Hence, assume that $m=4$. Let $k$ be the smallest index such that $u_k\notin R$.
Since $u_0,u_1\in R$ and $u_4\in L_{21}$,
it follows that $k\in\{2,3,4\}$. 

\medskip
\noindent{\bf Case 2.1.} $k\in\{2,3\}$.

Suppose first that $k=2$.
If $u_2\in L\cup X_2$, then an edge incident to $u_1$
is oriented initially or belongs to $[F_0,F_1]$,
contrary to $u_1\notin V(H_0)$.

If $u_2\in X_3$, then $\nu_{H_0}(u_2)\le4$.
Indeed, a neighbor in $L_2\cup R_2$ gives a bound
of three by the initial rules.
Otherwise, $u_2$ has a neighbor in $X_2$, and every
$F_1$-neighbor of $u_2$ also belongs to
$X_3\cap N(X_2)$.
If there is such a neighbor, the rules for equal
$f$-values and
\Cref{prop:F118,prop:cycleauv,prop:F12224}
give the bound of four.
If there is none, then $u_2\in\tilde X_3$;
its neighbor $u_1\in R_3$ puts $u_1$ in $R_{31}$,
and the initial rules give $u_2\to X_2$.

Thus, assume that $u_2\in X_4$.
Then $u_1\in R_3$ and $u_3\in L_3\cup X_3$.
If $u_3\in L_3$, then $u_2\in X_{41}$ and
$\nu_{H_0}(u_2)\le4$.
The remaining possibility has
$u_2\in X_{42}^{R}$ and $u_3\in X_3$. 
We consider this case together with the remaining case
for $k=3$ below.

Suppose now that $k=3$.
Then $u_3\in L\cup X_2\cup X_3$.
If $u_3\in L$, the initial arc $u_2\to u_3$
gives $\nu_{H_0}(u_2)\le4$.
If $u_3\in X_2$, then $u_2\in R_2$, and
$u_2\to u_3$ by the argument in Case 1.
If $u_3\in X_3$ and $u_2\in R_2$, the same direction
is prescribed initially, again giving
$\nu_{H_0}(u_2)\le4$.
Thus, the remaining possibility has
$u_2\in R_3$ and $u_3\in X_3$.

In both remaining possibilities,
$u_3\in X_3\cap N(L_{21})$, so
$\mu_{H_0}(u_3)=\nu_{H_0}(u_3)=3$ and $u_3\in M_a$.
If $u_2u_3$ is not anti, then $\D(u_2)\le4$.
If $u_3\to u_2$ belongs to $H_0$, its weight gives
$\nu_{H_0}(u_2)\le5$.

Thus, assume that $u_3\to u_2$ is oriented in
\Cref{step:DUb}.
Then $\nu_H(u_2)\le6$.
If equality holds, $u_2\in U_b\setminus M_b$.
Since $u_3\in M_a$, it follows that $u_2\in C$
and $u_3$ is its unique neighbor in $A$.
If $u_1\in U_a$, this contradicts $u_1\ne u_3$.
Otherwise, $u_1\in B_U$, and (\ref{partitionM})
gives $u_2\in M_r\subseteq D_b$, both when
$u_2\in R_3$ and when $u_2\in X_{42}^{R}$.
This contradicts $u_2\in U\setminus D$.
Therefore, $\nu_H(u_2)\le5$, as required.

\medskip
\noindent{\bf Case 2.2.} $k=4$.

Then $u_3\in R_{21}$, and the initial rules give
$\mu_{H_0}(u_3)=2$, $\nu_{H_0}(u_3)\le3$
and $u_3\to u_4$.
If $u_2u_3$ is not anti, then $\D(u_2)\le4$.
Thus, assume that $u_3\to u_2$.

This arc is not type-II.
Indeed, $u_3$ already satisfies both (a1) and (a2)
after the initial step, so it cannot be the chosen
$R_2$-vertex in \Cref{step:L23R23}.
If it is oriented in \Cref{step:DUb}, then
$u_2\in U_b\setminus D_b$.
All edges of $v_1v_0u_1u_2u_3$ are therefore
unoriented in $H_0$.
This path has weight at most
$\mu_{H_0}(v_1)+4+\nu_{H_0}(u_3)\le8$,
contrary to \Cref{step:sumatmost8}.

Hence, $u_3\to u_2$ belongs to $H_0$.
If its weight is at most eight, then
$\nu_{H_0}(u_2)\le5$.
It remains to consider weight nine, so
$\nu_{H_0}(u_2)=6$.

Consider the operation that orients $u_3u_2$. 
By \Cref{claimcenter}, this operation belongs to \Cref{step:cycle(a)}.
Neither endpoint of $u_3u_2$ belongs to $F_0$, so this edge is not
an additional edge in \ref{item:cycle5-short}.
If the operation orients a whole cycle, that cycle has length five:
a shorter cycle would give weight at most eight, since
$\alpha_k+\beta_k=4$ for $k\in\{2,3,4\}$.
Otherwise, let $\overrightarrow P$ be its selected path and let $H'$
be the partial orientation immediately after the operation.
By monotonicity and \Cref{prop:F12a}, 
$9=\omega_{H_0}(u_3\to u_2)
\le\omega_{H_0}(\overrightarrow P)
\le\omega_{H'}(\overrightarrow P)\le9$.
 Thus, the equality assertion of \Cref{prop:F12a} applies to the
selected path. In either case, the operation gives a dicycle of
length five meeting $F_0$ and containing its selected path or cycle. 
For any vertex $q$ of this dicycle in $F_0$,
the directed subpath from $u_2$ to $q$ has at least
three edges, since
$\nu_{H_0}(u_2)=6$ and $\nu_{H_0}(q)\le3$.
Consequently, the dicycle has the form
$q\to u_3\to u_2\to p\to r\to q$,
where $q\in R_{11}$ is its unique vertex in $F_0$.

The edge $q\to u_3$ was oriented initially.
Thus, the operation cannot orient a whole cycle.
Its selected path contains $u_3\to u_2$, has an
endpoint in $F_0$, and cannot contain $q\to u_3$.
It must therefore be
$u_3\to u_2\to p\to r\to q$.
This is a four-edge path of form (a). 
Let $K$ be the partial orientation immediately before this operation.
The initial rules give $\nu_K(u_3)\le3$.
Hence, the opposite direction has score
$\alpha_4+4+\nu_K(u_3)\le1+4+3=8$.
The selected direction has score at least nine,
since its arc $u_3\to u_2$ has weight nine in $H_0$.
This contradicts \ref{item:cycle5-four}. 

Thus, $\D(u_2)\le5$ in every case.
It follows that $\D(u_1)\le6$, completing Case 2
and the proof.
\end{proof}

\section{Conclusion}

In this paper, we prove that $f(4)\le 16$, improving the previous upper
bound  $f(4)\le18$. A main methodological contribution is the introduction
of the notion of a \emph{good path}, which records directed routes that
remain realizable in a partially oriented graph. Unlike earlier approaches
based mainly on predetermined local configurations, short cycles, or an
R--S orientation, this framework preserves global reachability information
while the orientation is being constructed.

The parameters $\U$ and $\D$, defined through good paths, measure the
minimum adjusted lengths of such realizable routes and guide the orientation
of the remaining edges. In this way, local orientation decisions are linked
to global distance estimates: the resulting parameters $\mu$ and $\nu$ are
controlled by $\U$ and $\D$, apart from the explicitly described exceptional
cases. Together with the layer-by-layer estimates, this yields an orientation
of directed diameter at most $16$.

The good-path framework may also be useful for studying $f(d)$ for larger
diameters. One possible approach is to choose a suitable central edge or
subgraph, partition the graph into distance layers, and define corresponding
good-path parameters to guide the remaining orientations. Although the
definitions and bounds would need to be adapted, the underlying idea of
encoding extendable directed paths is not specific to diameter $4$.

\section*{Acknowledgements}
Yaokun Feng, Xiaopan Lian and Zijian Ren are supported by Fundamental and Interdisciplinary Disciplines Breakthrough Plan of the Ministry of Education of China (JYB2025XDXM207). Xiaopan Lian is also supported by the National Natural Science Foundation of China (No.\,12671404).  Hui Lei is funded by the National Natural Science Foundation of China (Nos.\,12371351, 12431013), the Natural Science Foundation of Tianjin (24JCYBJC01670), and the Fundamental Research Funds for the Central Universities, Nankai University.
\medskip 

\noindent{\bf Declaration on the use of AI}

During the preparation of this paper, ChatGPT was used for proofreading and language editing to improve clarity and fluency. The
mathematical ideas, proofs, and all technical mathematical content are the author’s own work.

\end{document}